\documentclass{amsart}

\usepackage[a4paper]{geometry}
\newlength{\plarg}
\usepackage{hyperref}
\usepackage{xcolor}
\hypersetup{
    colorlinks=true,
    linkcolor={blue},
    citecolor={blue},
    urlcolor={blue}}

\usepackage{amsmath, amssymb,latexsym,amsthm,comment}
\usepackage[
backend=biber,
style=alphabetic
]{biblatex}

\usepackage{amsmath,amsfonts,amssymb,amsthm}
\usepackage{latexsym,pdfsync,xcolor,graphicx}
\usepackage{mathtools}
\usepackage[T1]{fontenc}
\usepackage{mathrsfs}

\usepackage{csquotes}
\usepackage[title]{appendix}
\usepackage{tikz}
\usepackage{appendix}
\usepackage{hyperref}
\usepackage{float}
\usepackage{caption}
\usepackage{subcaption}
\usepackage{array}
\usepackage{hhline}
\usepackage{multirow}
\usepackage{indentfirst}
\usepackage{xparse}

\newcommand{\ghat}{\hat{G}}
\newcommand{\xhat}{\hat{X}}

\newcommand{\expfamily}{\mathcal{N}}
\newcommand{\lang}{\mathscr{L}}

\usepackage{enumitem}
\newlist{steps}{enumerate}{1}
\setlist[steps, 1]{label = Step \arabic*:}

\DeclareMathOperator{\Aut}{Aut}
\DeclareMathOperator{\diam}{diam}
\DeclareMathOperator{\lcm}{lcm}
\DeclareMathOperator{\fix}{Fix}
\DeclareMathOperator{\stab}{Stab}
\usetikzlibrary{arrows.meta}
\DeclareUnicodeCharacter{22C5}{.}

\usepackage[english]{babel}
\newtheorem{theorem}{Theorem}
\numberwithin{theorem}{section}
\newtheorem{lemma}[theorem]{Lemma}
\newtheorem{corollary}[theorem]{Corollary}

\newtheorem{proposition}[theorem]{Proposition}

\theoremstyle{definition}
\newtheorem{definition}[theorem]{Definition}
\newtheorem*{notation}{Notation}

\newtheorem{question}{Question}

\theoremstyle{remark}
\newtheorem{remark}[theorem]{Remark}

\newcommand{\agl}{\text{AGL}(1,\mathbb{F}_{p})}
\newcommand{\gstar}{G^{*}}
\newcommand{\Cen}{\mathrm{Cen}}
\newcommand{\classc}{\mathcal{ST}(p)}
\newcommand{\classwst}[1]{\mathcal{WST}(#1)}
\newcommand{\inductiveG}[2]{G^{#1}_{#2}}
\newcommand{\classwstprime}[1]{\mathcal{WST'}(#1)}
\newcommand{\classwstprimezero}[1]{\mathcal{WST}'_0(#1)}
\newcommand{\classcprime}{\mathcal{ST}'(p)}
\newcommand{\classcprimezero}{\mathcal{ST}_{0}'(p)}

\renewcommand{\hat}{\widehat}
\newcommand{\calQ}{\mathcal{Q}}

\newcommand{\apices}{\hat{v}(\mathcal{Q})}
\newcommand{\rinj}{r_{\text{inj}}}
\newcommand{\invset}[2]{\mathcal{I}_{#1}^{#2}}
\newcommand{\transset}[1]{\mathcal{TR}_{#1}}

\begin{document}
\title{Non-split sharply 2-transitive groups of bounded exponent}
\author{Marco Amelio}

\begin{abstract}
    We construct here the first known examples of non-split sharply 2-transitive groups of bounded exponent in odd positive characteristic for every large enough prime $p \equiv 3 \pmod{4}$. In fact, we show that there are countably many pairwise non-isomorphic countable non-split sharply 2-transitive groups of characteristic $p$ for each such $p$. Furthermore, we construct non-periodic non-split sharply 2-transitive groups (of these same characteristics) with centralizers of involutions of bounded exponent. As a consequence of these results, we answer two open questions about sharply 2-transitive and 2-transitive permutation groups. The constructions of groups as announced rely on iteratively applying (geometric) small cancellation methods in the presence of involutions. To that end, we develop a method to control some small cancellation parameters in the presence of even-order torsion.
\end{abstract}

\date{\today}
\maketitle

\section{Introduction}
\label{section introduction}

\footnotetext[0\def\thefootnote{}]{The author was supported by the German Research Foundation (DFG) under Germany’s Excellence Strategy EXC 2044–390685587, Mathematics Münster: Dynamics–Geometry–Structure and by CRC 1442 Geometry: Deformations and Rigidity.}

Let $n\geq 1$ be an integer and let $G$ be a group acting on a set $X$ with at least $n$ elements. The action is said to be $n$\emph{-transitive} if for any two $n$-tuples of distinct elements $(x_{1}, \dots , x_{n})$ and $(y_{1}, \dots, y_{n})$ of $X^{n}$ there exists an element $g \in G$ such that $g \cdot x_{i}=y_{i}$ for $1 \leq i \leq n$. Similarly, the action is said to be $n$-\emph{sharp} if for any two such $n$-tuples there is at most one element $g \in G$ with the aforementioned property. Finally, the action is said to be \emph{sharply} $n$\emph{-transitive} if it is $n$-transitive and $n$-sharp. A group $G$ is called \emph{sharply $n$-transitive} if there is a set $X$ with at least $n$ elements on which $G$ acts sharply $n$-transitively.

Clearly, every group acts sharply 1-transitively (that is, regularly) on itself by left multiplication. On the other hand, for $n \geq 4$, there are only finitely many sharply $n$-transitive groups. Moreover, these groups are necessarily finite and they are completely classified. In fact, Jordan proved in \cite{jordan} that the only finite sharply $n$-transitive groups for $n \geq 4$ are the symmetric groups $S_{n}$ and $S_{n+1}$, the alternating group $A_{n+2}$, and the Mathieu groups $M_{11}$ and $M_{12}$ for the cases $n=4$ and $n=5$ respectively. Furthermore, in \cite[Chapitre IV, Théorème I]{tits}, Tits proved that there are no infinite sharply $n$-transitive groups for $n \geq 4$. Zassenhaus gave a complete classification of the finite sharply $n$-transitive groups for the cases $n=2$ and $n=3$ in \cite{zassenhaus1} and \cite{zassenhaus2}.  For $n=2$ and $n=3$, there do exist also infinite sharply $n$-transitive groups: for a skew-field $K$, the affine group $\text{AGL}(1,K)\cong K_{+} \rtimes K^{*}$ acts sharply 2-transitively on $K$, and for any (commutative) field $K$, the projective linear group $PGL(2,K)$ acts sharply 3-transitively on the projective line. These groups are infinite whenever $K$ is infinite.

An important feature associated with a sharply 2-transitive group action is its \emph{characteristic}, which we define as follows. Let $G\curvearrowright X$ be a sharply 2-transitive group action. It is easy to see that $G$ has involutions (that is, elements of order 2), and that involutions form a unique conjugacy class. Moreover, either no involution has fixed points or very involution has exactly one. In this last case, there is a $G$-equivariant bijection between the set $X$ and the set of involutions of $G$ (where we consider on this set the action of $G$ by conjugation), and therefore the translations (that is, products of two distinct involutions) form a conjugacy class. In this case, we define the characteristic of $G\curvearrowright X$ to be the order of a translation if this order is finite (in which case it is necessarily a prime number $\geq  3$), or as $0$ in case this order is infinite. If involutions have no fixed points, we say that the action $G\curvearrowright X$ has characteristic $2$. We will thus talk throughout this article of a sharply 2-transitive group of characteristic $p>2$ without specifying the set on which the group acts, as it should be understood to be the set of involutions of the group.

Until recently, it was not known whether a sharply 2-transitive group $G$ necessarily splits in the form $A\rtimes H$ for some non-trivial normal abelian subgroup $A$ (in which case we simply say that $G$ is \emph{split}). The first examples of non-split sharply 2-transitive groups were exhibited by Rips, Segev and Tent in \cite{rips_segev_tent} in characteristic 2 and by Rips and Tent in \cite{rips_tent} in characteristic 0. Then, the first examples of infinite simple sharply 2-transitive groups were constructed by André and Tent in \cite{andre_tent} and by André and Guirardel in \cite{andre_guir_fin_gen_simple} (with additional properties, in particular, finite generation), all of them in characteristic 0. Later, in \cite{amelio_andre_tent}, the author together with André and Tent constructed the first examples of non-split sharply 2-transitive groups in characteristic $p>3$. In fact, we proved that, for every large enough prime number $p$, there exist $2^{\aleph_0}$-many pairwise non-isomorphic non-split sharply 2-transitive groups of characteristic $p$. Notice that, by a well-known result due to Kerby (see \cite[Theorem 9.5]{kerby}), every sharply $2$-transitive group in characteristic $3$ splits. Furthermore, the methods of \cite{amelio_andre_tent} necessarily yield a very large value of such $p$. Thus, the problem of the existence of non-split sharply 2-transitive groups of `small' characteristic $p \geq 5$ remains, to the best of our knowledge, open.

Let us notice that all of the examples previously mentioned contain elements of infinite order (in fact, by definition in characteristic 0 this will always be the case). Furthermore, all of these groups have infinite order elements fixing a point. Again, in characteristic 0, this will always be the case: the centralizer of an involution in one such group will contain a subgroup isomorphic to $\mathbb Q ^*$, the multiplicative group of the field of rational numbers.

Throughout the years, a number of results have been proved that relate bounded exponent to splitting of sharply 2-transitive groups. Zassenhaus gave a complete classification of finite sharply 2-transitive groups in \cite{zassenhaus1} and \cite{zassenhaus2}, proving in particular that all of them split. Later, Suchkov proved in \cite{suchkov} proved that if the stabilizer of a point in a sharply 2-transitive group is a 2-group, then the group is finite (and thus split). In addition, Mayr proved in \cite{mayr} the same result for the case in which the stabilizer of a point has exponent 3 or 6. This was generalized by Jabara in \cite{jabara} to the case where the point-stabilizers are nilpotent of order $2^n 3$ for some positive integer $n$. In fact, the following question is raised in \cite{mayr}.

\begin{question}
\label{question mayr}
    Is a sharply 2-transitive group with point-stabilizers of bounded exponent necessarily finite?
\end{question}

Within the more general realm of 2-transitive permutation groups, Mazurov proved in \cite{mazurov} that every 2-transitive permutation group with an abelian stabilizer of a point is isomorphic to the affine group of a field $K$. In particular, no such group with an infinite cyclic stabilizer of a point exists. Notice that, by considering affine groups $\text{AGL}(1,K)$ over fields $K$ of positive characteristic, we can obtain infinite periodic 2-transitive permutation groups (in fact, countable sharply 2-transitive such groups), as well as (sharply) 2-transitive permutation groups with infinite order elements centralizing involutions, and such that every element not centralizing an involution has order bounded by an integer $n$. The following question by Sysak, appearing in the Kourovka Notebook as Problem 10.64, asks whether it is possible to have the converse situation.

\begin{question}\emph{\cite[Problem 10.64]{kourovka}}
\label{question kourovka}
    Does there exist a non-periodic doubly transitive permutation group with a periodic stabilizer of a point?
\end{question}

As a consequence of our main results, we will provide answers to both of these questions (see the details below).

The main result of this article is the following theorem, stating the existence of non-split sharply 2-transitive groups of bounded exponent for every large enough odd characteristic $p$ with $p \equiv 3 \pmod{4}$.

\begin{theorem}
\label{main theorem}
    There exists an odd number $q'$ with the following property: let $p \geq q'$ be a prime number such that $p \equiv 3 \pmod{4}$, and let $q_1 , q_2 \geq q'$ be a pair of odd numbers. Then, there exists a countable non-split sharply 2-transitive group $G$ of characteristic $p$ and exponent $\lcm  (q_1, q_2, p , p-1)$. Moreover, there exist elements $g$ and $g'$ in $G$ such that neither of them is a translation, $g$ centralizes no involution and is of order $q_1$, and $g'$ is of order $q_2$ and centralizes an involution.

    In addition, for every element $g$ of $G$, either $g$ is contained in a subgroup of $G$ that embeds into $\agl$ or $g$ falls into one of the following cases.
        \begin{enumerate}[label=(\arabic*)]
            \item The element $g$ centralizes no involution and is contained in a subgroup isomorphic to $C_{q_1}$.
            \item The element $g$ centralizes an involution and is contained in a subgroup isomorphic to $C_{2q_2}$.
        \end{enumerate}
\end{theorem}

This provides, in particular, a negative answer to Question \ref{question mayr}.

As noted by Olshanskii in private communication with Hull and Osin (see \cite[Section 6]{hull_osin}), if a group $G$ of exponent $n$ acts faithfully and $k$-transitively on a set $X$, then all integers $m \leq k$ must divide $n$: indeed, the stabilizer of a subset of $X$ of size $k$ maps surjectively onto the symmetric group $S_k$. In particular, no group of odd exponent admits a faithful 2-transitive action. Theorem \ref{main theorem} shows that this is not the case for groups of even exponent. Let us remark that the groups constructed in this article are not finitely generated. Thus, the question of the existence of a Burnside group admitting a faithful 2-transitive action (as raised in \cite{hull_osin}) remains, to the best of our knowledge, an open problem.

Notice also that Theorem \ref{main theorem} can be phrased in terms of near-fields and near-domains: it shows, for instance, that there exist (infinite) near-domains that are not near-fields with bounded exponent multiplicative group (see, for example, \cite{tent} for an explanation on how this interpretation arises).

As an immediate consequence of Theorem \ref{main theorem}, we obtain the following corollary (just by considering all possible distinct prime values of $q_1=q_2$).

\begin{corollary}
\label{non-split sh 2-trans of bounded exponent}
    There exists a prime number $q'$ with the following property: let $p \geq p'$ be a prime number such that $p \equiv 3 \pmod{4}$. There exist infinitely many countable pairwise non-isomorphic non-split sharply 2-transitive groups of characteristic $p$ of bounded exponent.
\end{corollary}

Furthermore, we will derive from Theorem \ref{main theorem}, using a model-theoretic compactness argument, the following result.

\begin{theorem}
\label{theorem kourovka problem}
    There exists an odd number $q'$ with the following property: let $p \geq q'$ be a prime number such that $p \equiv 3 \pmod{4}$ and $q_2 \geq q'$ an odd number. There exists a non-periodic non-split sharply 2-transitive group of characteristic $p$ such that the centralizer of every involution is of exponent bounded by $\lcm (q_2,p,p-1)$ and it contains an element of order $q_2$.
\end{theorem}

This provides, in particular, a positive answer to Question \ref{question kourovka}.

Let us remark now a few facts about the construction of groups as in Theorem \ref{main theorem}. The main difficulty in constructing non-split sharply $2$-transitive groups in characteristic $p>3$ comes from the fact that the methods used so far in order to construct non-split sharply $2$-transitive groups have proceeded through HNN-extensions, which create translations of infinite order. Therefore, since in a sharply 2-transitive group of characteristic $p>2$ all translations have order $p$, to obtain sharply $2$-transitive groups in characteristic $p$, it is necessary to add new relations of the form $(rs)^p=1$ for distinct involutions $r,s$. This was achieved in \cite{amelio_andre_tent} by taking small cancellation quotients similar to the quotients used in the solution of the famous Burnside problem about the existence of infinite finitely generated groups of finite exponent (posed by Burnside in \cite{burnside}), in order to guarantee that any translation has order $p$. In this article, we will further adapt the construction of \cite{amelio_andre_tent}, to be able to impose torsion on every element of infinite order, not just on the set of translations. It is a well-known fact that small cancellation gets considerably more complicated in the presence of even torsion. To illustrate this, let us mention that the original solution to the Burnside problem by Adian and Novikov in \cite{adian_novikov} was produced in 1968, when they proved that every free Burnside group of odd exponent (in at least two generators) is infinite provided the exponent is sufficiently large. Meanwhile, the analogue result for the even exponent case was proved only decades later, independently by Ivanov in 1994 (see \cite{ivanov_2}) and by Lysenok in 1996 (see \cite{lysenok}). All of these results are proved using some form of iterated small cancellation, and the difficulties in the presence of even torsion come, among others, from the fact that the algebraic structure of finite subgroups of free Burnside groups is more intricate in this case: every finite subgroup of a free Burnside group of odd order is cyclic, while infinite Burnside groups of even exponent contain arbitrarily long chains of direct products of finite dihedral groups. For further reading on the Burnside problem, see, for example, \cite{adian, olshanskii, ivanov, delzant_gromov, coulon_2, coulon_4, atkarskaya_rips_tent}.

As already observed above, sharply 2-transitive groups contain plenty of involutions, and one of the challenges we have to face is to keep these involutions under control when taking small cancellation quotients (and it is in order to maintain this control that we need to restrict ourselves to characteristic $p \equiv 3 \pmod{4}$, see Remark \ref{remark why we need tameness} for a more detailed explanation). The framework we use for this purpose is that of \emph{geometric small cancellation}.

In the 1910's, Dehn proved that for the fundamental group of a closed orientable surface of genus at least two the word problem is solvable. His work involved negative curvature, and was a precursor for small cancellation theory. Small cancellation conditions were formulated explicitly for the first time by Tartakovskii in 1947. Then, small cancellation theory was developed notably by Greendlinger in the early 1960's and by Lyndon and Schupp around the same time to study groups given by group presentations where defining relations have small overlaps with each other. However, the geometric origins of small cancellation theory were gradually forgotten in favour of combinatorial and topological methods. According to Gromov, `the role of curvature was reduced to a metaphor (algebraists do not trust geometry)', and he proposed to return to the geometric sources of small cancellation theory. This point of view appears in Gromov's paper \cite{gromov_mesoscopic}, and was then developed extensively by Delzant and Gromov in \cite{delzant_gromov}, by Arzhantseva and Delzant in \cite{arzhantseva_delzant}, by Coulon in \cite{coulon_3, coulon_2, coulon_1, coulon_4}, by Cantat, Lamy and de Cornulier in \cite{cantat_lamy}, and by Dahmani, Guirardel and Osin in \cite{dahmani_guirardel_osin} (see also \cite{coulon_bourbaki}). In this article, we will develop a further adaptation of the methods of Coulon, as was already done in \cite{amelio_andre_tent}.

\subsection*{Structure of the paper.} In Section \ref{section preliminaries}, we give some preliminaries about sharply 2-transitive groups and we introduce two classes of groups (a modification of the ones from \cite{amelio_andre_tent}) that will be key when proving Theorem \ref{main theorem}. In Section \ref{section outline proof}, we give a proof of Theorems \ref{main theorem} and \ref{theorem kourovka problem}, assuming two technical results whose proofs are postponed to Sections \ref{section pp-quotient of wstprime} and \ref{section HNN-ext of wst}. The remainder of the article is devoted to developing the necessary background to prove these two results: in Section \ref{section hyperbolic spaces and actions} we introduce some background on hyperbolic spaces and group actions on them. Later, in Section \ref{section small cancellation} we introduce the small cancellation framework that will be applied iteratively in Section \ref{section partial periodic quotients} to obtain partial periodic quotients of groups with some negative curvature features.

\subsection*{Acknowledgement} The author would like to thank Simon André, Rémi Coulon and Katrin Tent for several helpful and stimulating discussions on the topics of this paper.

\section{Preliminaries}
\label{section preliminaries}

In this section we will recall and introduce some basics on sharply 2-transitive groups.

We begin by fixing some terminology and notation. Let $G$ be a group, we will call an element $r$ of order 2 in $G$ an \emph{involution}. For two distinct involutions $r$ and $s$ of $G$, we call their product $rs$ a \emph{translation}. For an arbitrary element $g \in G$, we write $\lvert g \rvert$ for the order of $g$. Furthermore, for elements $g$ and $h$ of $G$, we adopt the convention that the conjugate of $g$ by $h$ is $h^{-1}gh$.

\begin{notation}
    For a group $G$ and an arbitrary subset $S \subseteq G$, we will write:
    \begin{itemize}
        \item $\invset{S}{}$ for the set of involutions of $S$,
        \item $\invset{S}{(2)}$ for the set of ordered pairs of distinct involutions of $S$,
        \item $\transset{S}$ for the set of translations of $S$, and
        \item for a pair $(r,s) \in \invset{G}{(2)}$, $D_{r,s}$ for $\langle r,s \rangle$.
    \end{itemize}
\end{notation}

Moreover, for a subgroup $H \leq G$, we put $N_G (H)$ for the normalizer of $H$ in $G$, and for $g \in G$ we put $\Cen(g)$ for the centralizer of $g$ in $G$. Furthermore, for a prime number $p$ we will write $\mathbb{F}_p$ for the finite field with $p$ elements. Also, for $n \geq 2$ we will write $C_{n}$ for the cyclic group of order $n$ and $D_{n}$ for the dihedral group of order $2n$

\begin{definition}
\label{definition sharply 2-transitive action}
    An action of a group $G$ on a set $X$ is \emph{sharply 2-transitive} if for every two ordered pairs $(x_1,y_1)$ and $(x_2, y_2)$ of distinct elements of $X$ there exists a unique $g \in G$ such that $g \cdot x_1 =x_2$ and $g \cdot y_1 = y_2$.
\end{definition}

We collect in the following lemma a number of classical results about sharply 2-transitive groups (see for example \cite{tent}).

\begin{lemma}
\label{translations form a conj class}
    Let $G$ be a group acting sharply 2-transitively on a set $X$. Then, $\invset{G}{}$ forms a single non-empty conjugacy class. In particular, either every involution has a (necessarily unique) fixed point, or no involution has one.

    Furthermore, if involutions have fixed points, then there is a $G$-equivariant bijection $\invset{G}{} \longrightarrow X$ given by $r \longmapsto \text{Fix}(r)$, where we consider the action of $G$ by conjugation on $\invset{G}{}$. In particular, $\transset{G}$ also forms a conjugacy class.
\end{lemma}

In view of the previous lemma, we can now define the \textit{characteristic} of a sharply 2-transitive group action.

\begin{definition}
\label{definition characteristic}
    Let $G$ be a group acting sharply 2-transitively on a set $X$. The \emph{characteristic} is defined as
    \begin{itemize}
        \item 2 if involutions of $G$ have no fixed points,
        \item p if involutions of $G$ have fixed points and all translations of $G$ have order $p$, or
        \item 0 if involutions of $G$ have fixed points and all translations of $G$ have infinite order.
    \end{itemize}
\end{definition}

Notice that, since being a translation is closed under taking powers, the characteristic of a sharply 2-transitive group is necessarily 0 or a prime number.

As a direct consequence of Lemma \ref{translations form a conj class} and Definition \ref{definition characteristic} we have the following fact.

\begin{lemma}
\label{sh 2-trs of odd char iff on invol}
    A group $G$ acts sharply 2-transitively on a set $X$ with characteristic $\neq 2$ if and only if $G$ acts freely and transitively on $\invset{G}{(2)}$ with the same characteristic.
\end{lemma}

In particular, when we consider a sharply 2-transitive action of characteristic $\neq 2$, as is the case for all actions considered in this article, we may omit mentioning the set $X$ (assuming it to be $\invset{G}{(2)}$) and talk about a sharply 2-transitive group $G$ without explicit mention of the action. We adopt from now on this convention.

We introduce now the last remaining concept related to the objects of study of this article, that is, that of a \emph{split} sharply 2-transitive group, as well as a criterion to characterize split sharply 2-transitive groups in terms of its set of translations.

\begin{definition}
\label{definition split s2t}
    A sharply 2-transitive group \emph{splits} (we also say it is split) if it contains a non-trivial normal abelian subgroup.
\end{definition}

\begin{theorem}\emph{(see \cite{neumann})}
\label{non-splitting criterion}
    A sharply 2-transitive group $G$ splits if and only if $\transset{G} \cup \{1\}$ forms an abelian subgroup.
\end{theorem}

The previous criterion by Bernhard Neumann will be central in proving that the sharply 2-transitive groups we construct do not split: we will show that they contain non-commuting translations.

The remainder of this section will be aimed at introducing two classes, $\classwst{p,q_1,q_2}$ and $\classwstprime{p,q_1,q_2}$, which will be used in the construction of non-split sharply 2-transitive groups of bounded exponent. The conditions defining class $\classwst{p,q_1,q_2}$ are similar to those for class $\classc$ in \cite{amelio_andre_tent} (as well as those considered in \cite{rips_tent,andre_tent,andre_guir_fin_gen_simple}), modified to obtain groups of bounded exponent. In order to be able to take small cancellation quotients as discussed in Section \ref{section introduction}, we introduce as well class $\classwstprime{p,q_1,q_2}$ (similar to class $\classcprime$ in \cite{amelio_andre_tent}), whose elements are pairs composed of a group $G$ and a tree $X$ endowed with an action of $G$ satisfying a number of technical assumptions that will allow us to take the aforementioned quotients.

We begin by reintroducing in Definition \ref{def paffine and pminimal} two concepts from \cite{amelio_andre_tent} to classify pairs of distinct involutions of a group that will be key in the inductive steps necessary to prove Theorem \ref{main theorem}: pairs of $p$-minimal and $p$-affine type. Intuitively, these two classes can be thought of in the following way: for a pair of involutions $(r,s) \in \invset{G}{(2)}$, the automorphisms of $D_{r,s}$ induced by conjugation by some element of $G$ are as few as possible (only inner automorphisms of $D_p$) in case the pair is of $p$-minimal type, and as many as possible (the whole group $\Aut (D_p)$) in case the pair is of $p$-affine type.

\begin{definition}
\label{def paffine and pminimal}
    Let $G$ be a group, $p$ an odd prime number, and $(r,s) \in \invset{G}{(2)}$.
    \begin{itemize}
        \item We say that the pair $(r,s)$ is of $p$\emph{-minimal} type if $\lvert rs \rvert =p$ and $N_G (\langle rs \rangle)=D_{r,s}$.
        \item We say that the pair $(r,s)$ is of $p$\emph{-affine} type if $D_{r,s}$ is contained in a subgroup $H$ of $G$ isomorphic to $\agl$.
    \end{itemize}
\end{definition}

\begin{remark}
\label{remark on paff and pmin}
    Let $(r,s) \in \invset{G}{(2)}$.
    \begin{enumerate}[label=(\arabic*)]
        \item Since $\langle rs \rangle$ is a characteristic subgroup of $D_{r,s}$, we have that $N_G (D_{r,s}) \leq N_G (\langle rs \rangle)$. In particular, if the pair is of $p$-minimal type we have that $N_G (D_{r,s})=D_{r,s}$.
        \item It is a known fact that $\agl$ has a unique subgroup $D$ isomorphic to $D_p$ and that every involution of $\agl$ is contained in this subgroup $D$. Therefore, if the pair $(r,s)$ is of $p$-affine type, then the isomorphism from $H$ onto $\agl$ as in Definition \ref{def paffine and pminimal} induces an isomorphism from $D_{r,s}$ onto $D$ (so in particular $\lvert rs \rvert=p$), and $H$ acts sharply 2-transitively on $\invset{D_{r,s}}{}$. Furthermore, if no non-trivial element of $G$ centralizes $r$ and $s$ then $N_G(D_{r,s})=H$, so $N_G(D_{r,s}) \cong \agl$ (since $H$ already acts 2-transitively by conjugation on $\invset{D_{r,s}}{}$). From this discussion together with the fact that every pair of distinct involutions of $D_p$ generate the whole subgroup, it also follows that $\agl \cong \Aut (D_p)$.
        \item \label{paff and pmin invariant under conj} Clearly, if $(r,s) \in \invset{G}{(2)}$ is of $p$-affine (respectively, $p$-minimal) type, then so is every conjugate of $(r,s)$. 
    \end{enumerate}
\end{remark}


We are ready to introduce now the auxiliary notions of a \emph{weakly sharply 2-transitive group of characteristic $p$} and of one such group of $(q_1,q_2)$\emph{-almost bounded exponent}.

\begin{definition}
\label{definition weakly s2t}
    Let $G$ be a group, $p$ be an odd prime number such that $p \equiv 3 \pmod{4}$, $q_1$ and $q_2$ odd integers. We will say that $G$ is \emph{weakly sharply 2-transitive of characteristic $p$} if it satisfies the following conditions.
    \begin{enumerate}[label=(\arabic*)]
        \item \label{def class C pairs of minimal or affine type}Every translation is either of order $p$ or of infinite order, and every pair $(r,s) \in \mathcal{I}_G^{(2)}$ such that $rs$ is of order $p$ is either of $p$-minimal type or of $p$-affine type.
        \item \label{def class C G trans on affine trans} The set of pairs $(r,s) \in \mathcal{I}_G^{(2)}$ of $p$-affine type is non-empty and $G$ acts transitively on it by conjugation.
        \item \label{def class C cent of trans is cyclic} For every pair $(r,s) \in \mathcal{I}_G^{(2)}$, the subgroup $\Cen _G(rs)$ is cyclic and generated by a translation.
    \end{enumerate}
    In addition, we say that $G$ is of $(q_1,q_2)$\emph{-almost bounded exponent} if the following holds.
    \begin{enumerate}[label=(\arabic*)]
        \setcounter{enumi}{3}
        \item \label{def class C almost bounded exponent} For every subgroup $E$ of finite order, either $E$ embeds into $\agl$ or $E$ falls into one of the following cases.
        \begin{enumerate}
            \item The subgroup $E$ is contained in a subgroup isomorphic to $C_{q_1}$ and no element of $E$ centralizes an involution.
            \item The subgroup $E$ is contained in a subgroup isomorphic to $C_{2q_2}$ (and thus every element of $E$ centralizes an involution).
        \end{enumerate}
    \end{enumerate}
\end{definition}

\begin{remark}
\label{wsttr of almost bdd expo no sgrp of order 4}
    Notice that a weakly sharply 2-transitive group of characteristic $p$ of $(q_1 , q_2)$-almost bounded exponent for odd integers $q_1$ and $q_2$ contains no subgroup of order 4: indeed, $C_{q_1}$ and $C_{2q_2}$ clearly contain no subgroups of order 4. Meanwhile, the order of $\agl$ is $p(p-1)$, and since $p \equiv 3 \pmod{4}$, this is an integer not divisible by 4.
\end{remark}

\begin{remark}
\label{weakly s2t can be s2t}
    The following observation justifies the terminology: if $G$ is weakly sharply 2-transitive of characteristic $p$, has no translation of infinite order and every pair $(r,s) \in \mathcal{I}_G^{(2)}$ is of $p$-affine type, then $G$ is sharply 2-transitive of characteristic $p$. Indeed, we have to prove that the set $\mathcal{I}_G^{(2)}$ is non-empty and that $G$ acts transitively and freely on it. The only non-obvious point is that $G$ acts freely on $\mathcal{I}_G^{(2)}$, or equivalently that no non-trivial element of $G$ centralizes two distinct involutions. Suppose towards a contradiction that one such element $g\in G$ centralizes distinct involutions $r,s$. Then $g$ is in $\text{Cen}(rs)=\langle h\rangle$ for some translation $h\in G$, but $\text{\text{Cen}}(rs)$ contains the translation $rs$, which is of order $p$, so $\text{\text{Cen}}(rs)=\langle rs\rangle$. However, $rs$ does not commute with $r$, and we arrive at a contradiction.

    Furthermore, if a weakly sharply 2-transitive group of characteristic $p$ is of $(q_1,q_2)$-almost bounded exponent with no elements of infinite order, then it is in fact of exponent bounded by $\lcm (q_1,q_2,p,p-1)$.
\end{remark}

\begin{remark}
    The similar notion of an almost sharply 2-transitive group of characteristic $p$ was introduced in \cite{amelio_andre_tent}. In that article, the authors require the extra condition that the (normal) subgroup $\langle \transset{G} \rangle$ contains no involutions. The purpose of this assumption is to control the small cancellation parameters that the authors use to produce non-split sharply 2-transitive groups of characteristic $p$. In the setting of this article, no condition of that kind is required on $\langle \transset{G} \rangle$ for the small cancellation results introduced in Sections \ref{section small cancellation} and \ref{section partial periodic quotients}, and therefore we do not need to include any such assumption. The control of the parameters is achieved instead by the requirement that $p \equiv 3 \pmod{4}$ (through Remark \ref{wsttr of almost bdd expo no sgrp of order 4}), which is not included in the definition of an almost sharply 2-transitive group of characteristic $p$ in \cite{amelio_andre_tent}.
\end{remark}

As it was explained before, in this article we will construct non-split sharply 2-transitive groups of odd characteristic and bounded exponent by successive steps of alternating HNN-extensions with small cancellation quotients. The following definitions and results have the purpose of keeping the small cancellation parameters under control when taking HNN-extensions.

\begin{definition}\emph{(See \cite[Definition 2.6]{amelio_andre_tent})}
\label{def jointly quasi-malnormal}
    Let $G$ be a group, $K$ and $K'$ subgroups of $G$. We say that $K$ is \emph{quasi-malnormal} if for all $g \in G \backslash K$ we have that $\lvert K \cap g^{-1}K g \rvert \leq 2$. We say that the pair $(K,K')$ is \emph{jointly quasi-malnormal} if $K$ is quasi-malnormal and for all $g \in G$ we have that $\lvert K \cap g^{-1}K' g \rvert \leq 2$.
\end{definition}

The next result appears as Lemma 2.7 and Remark 2.8 in \cite{amelio_andre_tent}.

\begin{lemma}
\label{paff and pmin pair is quasimalnormal}
    Let $G$ be a group, $(r,s)$ and $(r',s')$ pairs in $\invset{G}{(2)}$.
    \begin{enumerate}[label=(\arabic*)]
        \item The pair $(r,s)$ is of $p$-minimal type if and only if $D_{r,s}$ is quasi-malnormal.
        \item If $(r,s)$ is of $p$-minimal type and $(r',s')$ is of $p$-affine type, then the pair $(D_{r,s}, D_{r',s'})$ is jointly quasi-malnormal.
    \end{enumerate}
\end{lemma}

\begin{definition}\emph{(Compare \cite[Definition 2.9]{amelio_andre_tent})}
\label{def class WSTp}
    A group $G$ is in \emph{class $\classwst{p,q_1,q_2}$} if it is weakly sharply 2-transitive of characteristic $p$ of $(q_1,q_2)$-almost bounded exponent and all of its elements are of finite order.
\end{definition}

\begin{remark}\emph{(Compare \cite[Remark 2.10]{amelio_andre_tent})}
\label{wst with invs of paff type is s2t}
    Class $\classwst{p,q_1,q_2}$ is non-empty since it contains $\agl$. Furthermore, a group in class $\classwst{p, q_1, q_2}$ has every translation of order $p$, and if every pair in $\invset{G}{(2)}$ is of $p$-affine type, then it is sharply 2-transitive of characteristic $p$.
\end{remark}

Let us state again how the necessity for small cancellation quotients arises, so as to motivate the introduction of class $\classwstprime{p,q_1,q_2}$. As was stated before, in Section \ref{section outline proof} we will outline the construction of non-split sharply 2-transitive groups of characteristic $p$ of bounded exponent by a sequence of HNN-extensions that will ensure that every pair of distinct involutions is conjugate. More concretely, we will take HNN-extensions of groups in class $\classwst{p,q_1,q_2}$ conjugating pairs of distinct involutions that were not conjugate in the base group. However, in doing this, we will create elements of infinite order, some of which will be translations. In particular, a pair of involution whose product gives such translation cannot possibly be conjugate to a pair generating a finite dihedral group, and such a group will clearly not be of bounded exponent. Therefore, we need to take a `controlled quotient' in order to come back to $\classwst{p,q_1,q_2}$. Geometric small cancellation provides the framework for this, where we consider the action of the HNN-extension on its Bass-Serre tree. To that purpose, we introduce class $\classwstprime{p,q_1,q_2}$ associated to a group action on a metric space. Definition \ref{def class wstprime} involves parameters of this action in consideration, which will be introduced in Sections \ref{section hyperbolic spaces and actions}, \ref{section small cancellation} and \ref{section partial periodic quotients}.

\begin{definition}
\label{def class wstprime}
    Let $G$ be a group acting by isometries and without inversion of edges on a simplicial tree $X$. We say that the pair $(G,X)$ is in \emph{class $\classwstprime{p,q_1,q_2}$} if $G$ is weakly sharply 2-transitive of characteristic $p$ of $(q_1,q_2)$-almost bounded exponent and the following conditions are satisfied.
    \begin{enumerate}[label=(\arabic*')]
        \item \label{class wstprime action non-elementary} The action of $G$ on $X$ is non-elementary and acylindrical.
        \item \label{class wstprime parameters} The action is tame and is such that $\tau (G,X) \leq 5$ and $\Omega(G,X)=0$ (see Subsection \ref{subsection invariants} for the definition of the parameters); and the integers $p$, $q_1$ and $q_2$ are at least $n'_1$ (where the value of $n'_1$ will be specified in Remark \ref{remark further rescaling for rinj less 2}, it is at least $n_1$, where $n_1$ is the value of the parameter provided by Theorem \ref{limit step quotient} for these parameters, $\rinj (G,X) \geq 1$ and hyperbolicity constant $\delta =0$).
        \item \label{class wstprime infinte order implies loxo} Every element of infinite order of $G$ is loxodromic by its action on $X$.
    \end{enumerate}
\end{definition}

\begin{remark}
\label{remark centralizer of order at least 3}
    In Subsection \ref{subsection invariants} we will prove Lemma \ref{cent of element of order >3 in wstprime is elliptic}, implying the following fact: let $(G,X)$ be a pair in class $\classwstprime{p,q_1,q_2}$ and $g \in G$ of finite order $\geq 3$. Then, $N_G (\langle g \rangle)$ is elliptic (and therefore so is $\Cen_G (g)$).
\end{remark}

\section{Outline of the proofs}
\label{section outline proof}

In this section, we prove Theorems \ref{main theorem} and \ref{theorem kourovka problem}, modulo proving Propositions \ref{proposition hnn of wst} and \ref{proposition pp-quotient of wstprime}. The remainder of the article will be devoted to developing the necessary framework for proving these results. In Subsection \ref{subsection sh 2-trans of bounded exponent} we prove Theorem \ref{main theorem}, while in Subsection \ref{subsection kourovka problem} we show how Theorem \ref{main theorem} together with a model-theoretic compactness argument proves Theorem \ref{theorem kourovka problem}.

We now state Propositions \ref{proposition hnn of wst} and \ref{proposition pp-quotient of wstprime} for future reference in this section. The first of them will be proved in Section \ref{section HNN-ext of wst}. The second one will be proved in Section \ref{section pp-quotient of wstprime}.

\begin{proposition}
\label{proposition hnn of wst}
    Let $G$ be a group in class $\classwst{p,q_1,q_2}$ for integers $p $, $q_1$ and $q_2$ at least $ n'_1$. Let $(r,s)$ and $(r',s')$ be pairs in $\invset{G}{(2)}$ with $(r,s)$ of $p$-affine type and $(r',s')$ of $p$-minimal type (so that both $D_{r,s}$ and $D_{r',s'}$ are isomorphic to $D_p$). Then the following holds.
    \begin{enumerate}[label=(\arabic*)]
        \item \label{hnn of wst trivial subgroups} Let $\gstar= G \ast \mathbb{Z}$ and $X$ the Bass-Serre tree of the splitting of $\gstar$ as an HNN-extension of $G$ with trivial associated subgroups. Then, the pair $(\gstar , X)$ is in class $\classwstprime{p,q_1,q_2}$.
        \item \label{hnn of wst trivial subgroups dihedral} Let $\gstar$ be the following HNN-extension: \[ \langle G,t \, \vert \, t^{-1}rt=r'  , \, t^{-1}st=s' \rangle, \](an HNN-extension of $G$ with associated subgroups $D_{r,s}$ and $D_{r',s'}$). Let $X$ be the Bass-Serre tree of this splitting of $\gstar$. Then, the pair $(\gstar , X)$ is in class $\classwstprime{p,q_1,q_2}$.
    \end{enumerate}
    Moreover, the group $\gstar$ has the following additional properties.
    \begin{enumerate}[label=(\arabic*')]
        \item \label{hnn of wst trans of inf order} In case \ref{hnn of wst trivial subgroups}, $\gstar$ contains a translation of infinite order and translation length at most 2.
        \item \label{hnn of wst inf order element cent no inv} In case \ref{hnn of wst trivial subgroups}, $\gstar$ contains an element of infinite order that is not a translation, that has translation length 1 and that centralizes no involution.
        \item \label{hnn of wst elem of inf order cent an inv} In case \ref{hnn of wst trivial subgroups dihedral}, if $\lvert D_{r,s} \cap D_{r',s'} \rvert=2$, then $\gstar$ contains an element of infinite order (which is not a translation), that has translation length 1 and that centralizes an involution.
    \end{enumerate}
\end{proposition}

We will show in Subsection \ref{subsection sh 2-trans of bounded exponent} that at some point of the inductive process carried out to construct non-split sharply 2-transitive groups of bounded exponent, we will indeed find ourselves taking HNN-extensions of groups satisfying the conditions of Property \ref{hnn of wst elem of inf order cent an inv} of Proposition \ref{proposition hnn of wst}.

The next proposition shows how to `come back' to class $\classwst{p,q_1,q_2}$ by taking a quotient of a group in $\classwstprime{p,q_1,q_2}$ (for example, after applying Proposition \ref{proposition hnn of wst}).

\begin{proposition}
\label{proposition pp-quotient of wstprime}
     Let $(G,X)$ be a pair in class $\classwstprime{p,q_1,q_2}$ for some prime $p$ and odd numbers $q_1$ and $q_2$. Then, $G$ has a quotient group $\bar G$ that is in class $\classwst{p,q_1,q_2}$ with the following additional properties.
    \begin{enumerate}[label=(\arabic*)]
        \item \label{no new invol pp-quotient of wstprime} Every involution of $\bar G$ is the image of an involution of $G$.
        \item \label{elliptics embed pp-quotient of wstprime} If $F$ is an elliptic subgroup of $G$ (for its action on $X$), then the projection map $G \twoheadrightarrow \bar G$ induces an isomorphism from $F$ onto its image.
        \item \label{image of paff is paff and of pmin is pmin} The image of a pair $(r,s) \in \invset{G}{(2)}$ of $p$-affine (respectively, of $p$-minimal) type is again of $p$-affine (respectively, of $p$-minimal) type. Moreover, a pair $(\bar r , \bar s) \in \invset{\bar G}{(2)}$ is of $p$-affine type if and only if every preimage of the pair in $\invset{G}{(2)}$ is of $p$-affine type.
        \item \label{cent of element of order at least 3 in pp-quotient of wstprime} Let $g \in G$ be an element of finite order $\geq 3$, and let $\bar g$ be its image on $\bar G$. Then, the projection map $G \twoheadrightarrow \bar G$ induces an isomorphism from $N_G (\langle g \rangle)$ onto $N_{\bar G}(\langle \bar g \rangle)$ (and thus also from $\Cen_G (g)$ onto $\Cen_{\bar G}( \bar g)$).
        \item \label{trans of inf order implies non-commuting translations} If $G$ contains a translation of infinite order and translation length at most 2, then $\bar G$ contains non-commuting translations.
        \item \label{pp-quotient of wstprime elem not cent an inv} If $G$ contains an element of infinite order that is not a translation, that has translation length 1 and that centralizes no involution, then $\bar G$ contains an element of order $q_1$ that is not a translation and centralizes no involution.
        \item \label{pp-quotient of wstprime elem cent an inv} If $G$ contains an element of infinite order (which is not a translation), that has translation length 1 and that centralizes an involution, then $\bar G$ contains an element of order $q_2$ which is not a translation and centralizes an involution.
    \end{enumerate}
\end{proposition}

We will also make use of the following observation, which is an immediate consequence of the definition of class $\classwst{p,q_1,q_2}$ (Definition \ref{def class WSTp}).

\begin{remark}
\label{remark union of chain of wst}
    If a group $G$ is the union of an infinite ascending chain of subgroups $H_\lambda$ for $\lambda < \gamma$, all of them in class $\classwst{p,q_1,q_2}$, then $G$ itself is in class $\classwst{p,q_1,q_2}$.
\end{remark}

\subsection{Non-split sharply 2-transitive groups of bounded exponent}
\label{subsection sh 2-trans of bounded exponent}

In this section we will prove the main result of our article, Theorem \ref{restatement main theorem}. It is a strengthening of Theorem \ref{main theorem}. We closely follow the proof of Theorem 2.17 in \cite{amelio_andre_tent}.

\begin{theorem}
\label{restatement main theorem}
    There exists an odd number $q'$ with the following property: let $p \geq q'$ be a prime number such that $p \equiv 3 \pmod{4}$, and let $q_1,q_2 \geq q'$ be a pair of odd numbers. Let $G \in \classwst{p,q_1,q_2}$. Then, $G$ embeds into a non-split sharply 2-transitive group $\mathbf{G}$ of characteristic $p$, exponent $ \lcm (q_1,q_2, p,p-1)$ and cardinality $\max \{ \aleph _0 \, , \lvert G \rvert \}$.
    
    Moreover, there exist elements $g$ and $g'$ in $\mathbf{G}$ such that neither of them is a translation, $g$ centralizes no involution and is of order $q_1$, and $g'$ centralizes an involution and is of order $q_2$.

    In addition, the following holds: for every element $g$ of $\mathbf{G}$, either $g$ is contained in a subgroup of $\mathbf{G}$ that embeds into $\agl$ or $g$ falls into one of the following cases.
        \begin{enumerate}
            \item The element $g$ centralizes no involution and is contained in a subgroup isomorphic to $C_{q_1}$.
            \item The element $g$ centralizes an involution and is contained in a subgroup isomorphic to $C_{2q_2}$.
        \end{enumerate}
\end{theorem}

\begin{proof}
    Let $q'$ be $n'_1$, the odd integer given in Definition \ref{def class wstprime} Condition \ref{class wstprime parameters}. Let $p \geq q'$ be a prime number such that $p \equiv 3 \pmod{4}$, and let $q_1,q_2 \geq q'$ be odd numbers. Let $G$ be a group in class $\classwstprime{p,q_1,q_2}$ and put $\gstar= G \ast \mathbb{Z}$ and $X$ for the Bass-Serre tree of the splitting of $\gstar$ as an HNN-extension of $G$ with trivial associated subgroups. Since $G$ is in class $\classwst{p,q_1,q_2}$, Proposition \ref{proposition hnn of wst} gives that the pair $(\gstar , X)$ is in class $\classwstprime{p,q_1,q_2}$. Furthermore, $\gstar$ contains translations of infinite order and of translation length at most 2, and an element of infinite order that is not a translation, that has translation length 1 and that centralizes no involution. Write $\inductiveG{0}{0}=\bar {\gstar}$, where $\bar{\gstar}$ is the group obtained from the pair $(\gstar , X)$ by applying Proposition \ref{proposition pp-quotient of wstprime}. In particular, Consequence \ref{elliptics embed pp-quotient of wstprime} of this proposition implies that, since (the isomorphic image in $\gstar$ of) $G$ is elliptic by its action on $X$, $G$ embeds into $\inductiveG{0}{0}$. Furthermore, Consequence \ref{trans of inf order implies non-commuting translations} gives that $\inductiveG{0}{0}$ contains non-commuting translations, and Consequence \ref{pp-quotient of wstprime elem not cent an inv} implies that $\inductiveG{0}{0}$ contains an element $g$ that is not a translation, centralizes no involution and has order $q_1$.

    We now fix a pair of involutions $(r,s) \in \invset{\inductiveG{0}{0}}{(2)}$ of $p$-affine type, and we enumerate all pairs of involutions in $\invset{\inductiveG{0}{0}}{(2)}$ as $\{(r_{0}^{\lambda}, s_{0}^{\lambda}) \, : \, \lambda < \gamma \} $. We will build inductively a sequence of groups $\inductiveG{\alpha}{0}$ for $\alpha < \gamma$. For a successor ordinal $\alpha +1$, suppose that $\inductiveG{\alpha}{0}$ has already been built, that this group is in class $\classwst{p,q_1,q_2}$, that $\inductiveG{0}{0}$ embeds in $\inductiveG{\alpha}{0}$, that $g$ is not a translation in this group, that $g$ centralizes no involution of $\inductiveG{\alpha}{0}$, and that every pair $(r_{0}^{\beta}, s_{0}^{\beta}) \, : \, \beta < \alpha $ is of $p$-affine type.  Consider the pair $(r_{0}^{\alpha},s_{0}^{\alpha})$. If this pair is of $p$-affine type, we put $\inductiveG{\alpha+1}{0}=\inductiveG{\alpha}{0}$. If it is of $p$-minimal type, we set \[ (\inductiveG{\alpha}{0})^* = \langle \inductiveG{\alpha}{0},t \vert t^{-1}rt=r_{0}^{\alpha} , t^{-1}st=s_{0}^{\alpha} \rangle. \]This is a well-defined HNN-extension since both $D_{r,s}$ and $D_{r_{0}^{\alpha}, s_{0}^{\alpha}}$ are isomorphic to $D_p$. Clearly $\inductiveG{0}{0}$ embeds into $(\inductiveG{\alpha}{0})^*$ and in this group the pair $(r_{0}^{\alpha}, s_{0}^{\alpha})$ is of $p$-affine type. Moreover, by Lemma \ref{proposition hnn of wst} the pair $((\inductiveG{\alpha}{0})^*,X)$ is in class $\classwstprime{p,q_1,q_2}$ (where $X$ is the Bass-Serre tree of the HNN-extension). Thus, Remark \ref{remark centralizer of order at least 3} gives that $N_{(\inductiveG{\alpha}{0})^*} (\langle g \rangle)$ is elliptic, and therefore $g$ is not a translation and centralizes no involution of $(\inductiveG{\alpha}{0})^*$. Now, by Proposition \ref{proposition pp-quotient of wstprime}, there is a quotient $\overline{(\inductiveG{\alpha}{0})^*}$ of $(\inductiveG{\alpha}{0})^*$ such that this group is in class $\classwst{p,q_1,q_2}$. In addition, since (the isomorphic image in $(\inductiveG{\alpha}{0})^*$ of) $\inductiveG{0}{0}$ is elliptic by its action on $X$, then it embeds into $\overline{(\inductiveG{\alpha}{0})^*}$. Similarly, the subgroup $H$ of $(\inductiveG{\alpha}{0})^*$ isomorphic to $\agl$ containing $D_{r_{0}^{\alpha}, s_{0}^{\alpha}}$ is finite and therefore elliptic, and thus it embeds into $\overline{(\inductiveG{\alpha}{0})^*}$. In particular, $(r_{0}^{\alpha}, s_{0}^{\alpha})$ is of $p$-affine type in $\overline{(\inductiveG{\alpha}{0})^*}$, and thus every pair $(r_{0}^{\lambda},s_{0}^{\lambda})$ of distinct involutions is of $p$-affine type for $\lambda \leq \alpha$. Moreover, Consequence \ref{cent of element of order at least 3 in pp-quotient of wstprime} of Proposition \ref{proposition pp-quotient of wstprime} implies that $g$ is not a translation and centralizes no involution in this group. Set then $\inductiveG{\alpha +1}{0}= \overline{(\inductiveG{\alpha}{0})^*}$.
    
    If $\alpha$ is a limit ordinal, we set $\inductiveG{\alpha}{0}= \bigcup \limits_{\beta < \alpha} \inductiveG{\beta}{0}$. By Remark \ref{remark union of chain of wst} this group is in class $\classwst{p,q_1,q_2}$, $\inductiveG{0}{0}$ embeds into $\inductiveG{\alpha}{0}$ and every pair of distinct involutions $(r_{0}^{\beta}, s_{0}^{\beta})$ for $\beta < \alpha$ is of $p$-affine type. Clearly $g$ is not a translation and centralizes no involution in this union.

    Set now $\inductiveG{0}{1}= \bigcup \limits_{\lambda < \gamma} \inductiveG{\lambda}{0}$. As in the previous paragraph, this group is in class $\classwst{p,q_1,q_2}$, $\inductiveG{0}{0}$ embeds into $\inductiveG{0}{1}$, $g$ is not a translation in this group and centralizes no involution of $\inductiveG{0}{1}$, and every pair in $\invset{\inductiveG{0}{0}}{(2)}$ is of $p$-affine type in $\inductiveG{0}{1}$. Furthermore, by construction the cardinality of $\inductiveG{0}{1}$ is the maximum of the cardinality of $G$ and $\aleph_0$.

    Now, we build $\inductiveG{0}{i+1}$ from $\inductiveG{0}{i}$ in a completely analogous way to how the construction of $\inductiveG{0}{1}$ from $\inductiveG{0}{0}$: we enumerate the pairs $(r_{i}^{\beta}, s_{i}^{\beta})$ of $\invset{\inductiveG{0}{i}}{(2)}$ and conjugate pairs of $p$-minimal type to the pair of $p$-affine type $(r,s)$. Assume that at step $\alpha +1$ we have built a group $\inductiveG{\alpha}{i}$ in class $\classwst{p,q_1,q_2}$ such that $\inductiveG{\alpha}{0}$ embeds into it (and therefore so does $\inductiveG{0}{0}$), such that every pair $(r_{i}^{\beta}, s_{i}^{\beta})$ is of $p$-affine type in $\inductiveG{\alpha}{i}$ for $\beta < \alpha$, and such that $g$ is not a translation and centralizes no involution in this group. If at this step we need to take an HNN-extension, Propositions \ref{proposition hnn of wst} and \ref{proposition pp-quotient of wstprime} ensure that we can construct a group $\inductiveG{\alpha +1}{i}$ with the same properties as we mentioned for $\inductiveG{\alpha}{i}$ and with $(r_{i}^{\alpha}, s_{i}^{\alpha})$ of $p$-affine type. Finally, when taking unions at the limit steps and when taking $\inductiveG{0}{i+1}= \bigcup \limits_{\lambda < \gamma} \inductiveG{\lambda}{i}$, Remark \ref{remark union of chain of wst} ensures that $\inductiveG{0}{i+1}$ is in class $\classwst{p,q_1,q_2}$. Furthermore, $\inductiveG{0}{i}$ (and thus also $\inductiveG{0}{0}$) embeds into $\inductiveG{0}{i+1}$, $g$ is not a translation and centralizes no involution of $\inductiveG{0}{i+1}$, and in this group every pair of $\invset{\inductiveG{0}{i}}{(2)}$ is of $p$-affine type, since one such pair is conjugate to $(r,s)$. Again, we have by construction that the cardinality of $\inductiveG{0}{i+1}$ is the maximum of the cardinality of $G$ and $\aleph_0$.

    Notice that the group $\gstar$ has a dihedral subgroup of infinite order containing the involution $r$ (take for example $D_{r, t^{-1}rt}$, where $t$ is the generator of the $\mathbb Z$ factor of the free product). The image of this pair in $\inductiveG{0}{0}$, which we denote by $(r,r')$, is necessarily of $p$-minimal type by Consequence \ref{image of paff is paff and of pmin is pmin} of Proposition \ref{proposition pp-quotient of wstprime}. Now, the pair $(r,r')$ is guaranteed to be of $p$-affine type (and thus, conjugate to $(r,s)$) in $\inductiveG{0}{2}$, which implies that one of the HNN-extensions considered until this step, in fact, satisfied Condition \ref{hnn of wst elem of inf order cent an inv} of Proposition \ref{proposition hnn of wst}. Thus, by Consequence \ref{pp-quotient of wstprime elem cent an inv} of Proposition \ref{proposition pp-quotient of wstprime} there is some $\beta < \gamma$ such that $\inductiveG{\beta}{1}$ contains an element $g'$ contained in a subgroup isomorphic to $C_{2q_2}$ which is not a translation. Then, an argument using normalizers, Consequence \ref{cent of element of order at least 3 in pp-quotient of wstprime} of Proposition \ref{proposition pp-quotient of wstprime} (analogous to the one used for $g$) gives that, throughout all of the induction process, $g'$ is contained in a subgroup isomorphic to $C_{2q_2}$ and is not a translation.

    Set now $\mathbf{G}= \bigcup \limits_{i \in \mathbb N} \inductiveG{0}{i}$. By Remark \ref{remark union of chain of wst} this group is in class $\classwst{p,q_1,q_2}$. We claim that this group satisfied the announced properties. In fact, every pair of involutions $(r',s') \in \invset{\mathbf G}{(2)}$ is of $p$-affine type: if $j$ is the minimal integer such that $(r',s') \in \invset{\inductiveG{0}{j}}{(2)}$, then the pair $(r',s')$ is guaranteed to be of $p$-affine type in $\inductiveG{0}{j+1}$ (and in consequence also in $\mathbf G$). In particular, by Remark \ref{wst with invs of paff type is s2t} the group is sharply 2-transitive of characteristic $p$. In addition, by Remark \ref{weakly s2t can be s2t} this group is of exponent at most $\lcm (q_1,q_2,p,p-1)$. By construction, $\inductiveG{0}{0}$ embeds into $\mathbf G$, so this group contains non-commuting translations. In particular, by Theorem \ref{non-splitting criterion} $\mathbf{G}$ is non-split. The element $g$ has order $q_1$, and it centralizes no involution of $\mathbf{G}$ since it does not centralize an involution in any $\inductiveG{0}{j}$ for $j \in \mathbb N$. The element $g'$ is contained in a subgroup of $\mathbf G$ isomorphic to $C_{q_2}$. Neither $g$ nor $g'$ can be a translation in $\mathbf G$, since they are not translations in any of the intermediate steps $\inductiveG{0}{j}$ for $j \in \mathbb N$. In particular, since $\agl$ also embeds into $\mathbf G$, this group contains elements of order $q_1$, $q_2$, $p$ and $p-1$, so, in fact, its exponent is exactly $\lcm (q_1,q_2,p,p-1)$. The claim that every element of $\mathbf G$ that is not contained in a subgroup that embeds into $\agl$ either centralizes no involution and is contained in a subgroup isomorphic to $C_{q_1}$ or centralizes an involution and is contained in a subgroup isomorphic to $C_{2q_2}$ follows directly from the fact that $\mathbf G$ is in class $\classwst{p,q_1,q_2}$. Finally, once again by construction we have that the cardinality of $\mathbf{G}$ is the maximum of the cardinality of $G$ and $\aleph_0$. Thus, $\mathbf{G}$ is a group as claimed by Theorem \ref{restatement main theorem}.
\end{proof}

\subsection{Non-split non-periodic sharply 2-transitive groups with bounded exponent stabilizers}
\label{subsection kourovka problem}

In this subsection we will prove Theorem \ref{theorem kourovka problem}. For the sake of completeness, we restate it here.

\begin{theorem}
\label{restatement theorem kourovka problem}
    There exists an odd number $q'$ with the following property: let $p \geq q'$ be a prime number such that $p \equiv 3 \pmod{4}$ and $q_2 \geq q'$ an odd number. There exists a non-periodic non-split sharply 2-transitive group of characteristic $p$ such that the centralizer of every involution is of exponent bounded by $\lcm (q_2,p,p-1)$ and it contains an element of order $q_2$.
\end{theorem}

In particular, by taking all possible values of the prime number $q_2$, we get the following corollary.

\begin{corollary}
\label{corollary kourovka problem}
    There exists a prime number $p'$ with the following property: let $p \geq p'$ be a primer number such that $p \equiv 3 \pmod{4}$. There exist infinitely many countable pairwise non-isomorphic non-periodic non-split sharply 2-transitive groups of characteristic $p$ such that the centralizer of every involution has bounded exponent.
\end{corollary}

Now, centralizers of involutions coincide with point stabilizers of the action of a group $G$ on the set $\invset{G}{}$. Therefore, since the action in consideration on Theorem \ref{restatement theorem kourovka problem} and Corollary \ref{corollary kourovka problem} is precisely the action by conjugation on the set of involutions, these results provide a positive answer to Question \ref{question kourovka} (see Section \ref{section introduction}).

As was stated in the beginning of this section, the proof of Theorem \ref{restatement theorem kourovka problem} uses Theorem \ref{main theorem} and some (very mild) model-theoretic methods. We assume the reader to be familiar with some very basic model-theoretic concepts such as \emph{language, first-order sentence, first-order theory, satisfiability, definability} and a \emph{model} of a first-order theory.

The key result is the very well-known model-theoretic \emph{Compactness theorem}. We say that a set of first order sentences $\Sigma$ (over a language $\lang$) is \emph{consistent} if it has a model. We say that such a set of sentences is finitely satisfiable if every finite subset of $\Sigma$ is consistent. The Compactness Theorem says that these two notions actually coincide.

\begin{theorem}\emph{Compactness Theorem, (see for example \cite[Theorem 2.2.1]{tent_ziegler}).}
\label{compactness theorem}
    Let $\Sigma$ be a set of first order sentences. Then, $\Sigma$ is consistent if and only if it is finitely satisfiable.
\end{theorem}

We exhibit in Lemmas \ref{lemma definability sh 2-trs} to Lemma \ref{lemma definability non-periodicity} a series of first-order properties (in the language of groups $\lang _{\text{Grp}} = \{ \cdot, e \}$, where the inverse function $^{-1}$ is definable) that will allow us to produce non-periodic groups with the desired properties. The first order sentences defining each of these properties are not explicitly stated, but they can be easily constructed with some elementary first-order logic (the interested reader may check the first chapters of \cite{tent_ziegler}).

\begin{lemma}
\label{lemma definability sh 2-trs}
    The property of a group $G$ of being non-split sharply 2-transitive of characteristic $p$ is definable.
\end{lemma}

Call $\varphi_{\text{nssh2tr}}(p)$ the first-order sentence provided by Lemma \ref{lemma definability sh 2-trs}.

\begin{lemma}
\label{lemma cent of inv of bdd exponent definable}
    The property of a group $G$ of having centralizers of involutions of exponent at most a given positive integer $n$ is definable. Furthermore, it is also definable if we ask the centralizers of involutions to contain an element of a given order $n'$.
\end{lemma}

Call $\varphi_{\text{exp(Cinv)}}(n,n')$ the first-order sentence provided by Lemma \ref{lemma cent of inv of bdd exponent definable}.

\begin{lemma}
\label{lemma definability unbounded exponent}
    The property of a group $G$ of having an element of order larger than a given $n \in \mathbb N$ is definable. Moreover, if we add the extra condition that the element centralizes no involution, it is still a first-order property.
\end{lemma}

Call $\varphi'_{\text{exp}}(n)$ and $\varphi_{\text{exp}}(n)$ the first and second first-order sentences provided by Lemma \ref{lemma definability unbounded exponent}.

\begin{lemma}
\label{lemma definability non-periodicity}
    The following property of a group $G$ is definable: there is an element $g \in G$ that is not a translation, centralizes no involution and if it has order larger than a given positive integer $n$, then its order is at least another given integer $n' >n$.
\end{lemma}

Call $\varphi_{\text{exp ninv}}(n,n')$ the first-order sentence provided by Lemma \ref{lemma definability non-periodicity}.

We are now ready to prove Theorem \ref{restatement theorem kourovka problem}.

\begin{proof}[Proof of Theorem \ref{restatement theorem kourovka problem}]
    We denote by $\Sigma' (p,q_2)$ the set of first-order sentences built as follows. First, it contains the axioms for groups. It also contains the sentences $\varphi_{\text{nssh2tr}}(p)$ and $\varphi_{\text{exp(Cinv)}}(p(p-1)q_2,q_2)$. Notice that any group satisfying this set of sentences will be non-split sharply 2-transitive of characteristic $p$ with centralizers of involutions of exponent at most $p(p-1)q_2$, and it will contain an element of order $q_2$ centralizing each involution. By Theorem \ref{restatement main theorem}, there is an odd integer $q'$ such that for every prime $p \geq q'$ such that $p \equiv 3 \pmod{4}$ and odd integer $q_2 \geq q'$ this set of sentences is consistent.

    Now, we write $\Sigma (p,q_2,q_1)$ for the set of sentences consisting of the union of $\Sigma' (p,q_2)$ and the set consisting of $\varphi _{\text{exp}}(p(p-1))$ and $\varphi _{\text{exp ninv}}(p(p-1),q_1)$. Once again, by Theorem \ref{restatement main theorem}, for $q'$, $p$ and $q_2$ as in the previous paragraph the set $\Sigma (p,q_2,p_1)$ is consistent for every \emph{prime} integer $p_1 \geq q'$ (we add the assumption that the third parameter is prime so that every element not centralizing an involution and not contained in a subgroup embedding into $\agl$ has the same order $p_1$).
    
    Now, let $\Sigma (p,q_2)$ be the following set of sentences: $\Sigma (p,q_2) = \bigcup \limits_{q_1 \in \mathbb N} \Sigma (p,q_2,q_1) $. This set of sentences is finitely satisfiable: any finite subset requires the order of elements not centralizing involutions or contained in a subgroup embedding into $\agl$ to be at least a given positive integer, and thus by the previous paragraph this finite subset is in fact consistent. By Theorem \ref{compactness theorem} this set is consistent, and thus there is a group $G$ satisfying it.

    Now, we have that, as was observed before, $G$ is a non-split sharply 2-transitive group of characteristic $p$, with centralizers of involutions of exponent at most $p(p-1)q_2$ and containing an element of order $q_2$ centralizing an involution. By construction, since this group contains $\varphi _{\text{exp ninv}}(p(p-1),q_1)$ for all positive integers $q_1$, then any element not centralizing an involution or contained in a subgroup embedding into $\agl$ has infinite order. Since this group also satisfies the sentence $\varphi_{\text{exp}}(p(p-1))$, one such element of $G$ exists, and thus $G$ is non-periodic.
\end{proof}

\section{Hyperbolic metric spaces and group actions}
\label{section hyperbolic spaces and actions}

In this section, we recall some concepts from metric geometry. We begin with a brief overview on the basics of \emph{length metric spaces} and \emph{quasi-geodesics} in Subsection~\ref{length metric spaces} (see for example \cite[Chapter 2]{burago_burago_ivanov}). In Subsection~\ref{subsection hyperb metric spaces} we recall some properties of Gromov-hyperbolic spaces, and in Subsection~\ref{subsubsection isom of hyp spaces} we study the actions by isometries of groups on these spaces. Finally, in Section~\ref{subsection invariants} we study some invariants of a group action on a hyperbolic space, some of them appearing in \cite[Section 3.5]{coulon_1}, and some other newly defined in this article. We closely follow the exposition by Coulon in \cite[Sections 2 and 3]{coulon_1}.

For a metric space $X$ and two points $x$ and $x'$ of $X$, we will denote by $d_{X}(x,x')$ (or eventually just $d(x,x')$ if the metric space is clear from the context) the distance between $x$ and $x'$.

\subsection{Length metric spaces and quasi-geodesics}
\label{length metric spaces}

For this subsection fix a metric space $(X,d)$. By a \emph{path} in $X$ we mean a continuous map $\gamma : I \longrightarrow X$, where $I \subseteq \mathbb R$ is an interval (possibly consisting of a single point).

\begin{definition}
\label{def length fun by metric}
    Let $\gamma: [a,b] \longrightarrow X$ be a path.
    \begin{itemize}
        \item A \emph{partition} of the interval $[a,b]$ is a finite subset $Y=\{ y_{0}, \dots y_{N} \}$ such that $a=y_{0}\leq y_{1} \leq \dots \leq y_{N}=b$.
        \item The \emph{sum (in $\gamma$) of a partition $Y$} is \[ \Sigma(Y)= \sum_{i=1}^{N} d(\gamma(y_{i-1}),\gamma(y_{i})). \]
        \item The \emph{length} $L_{d}(\gamma)$ of $\gamma$ with respect to $d$ is \[ L_{d}(\gamma)= \sup \{ \Sigma(Y) \, : \, Y \text{is a partition of} \ [a,b] \}. \]
    \end{itemize}
\end{definition}

\begin{definition}
\label{length induces metric}
\label{definition length space}
    Let $(X,d)$ be connected by rectifiable paths. The \emph{intrinsic metric} $d_{\ell}$ on $X$ is 
    \[ d_{\ell}(x,y)= \inf \{ L_d(\gamma) \, : \, \gamma:[a,b] \longrightarrow X , \, \gamma(a)=x, \, \gamma(b)=y \}. \]Where the  infimum is taken on the set of all paths from $x$ to $y$. The metric space $(X,d)$ is a \emph{length space} if $d_{\ell}=d$. If, in addition, $(X,d)$ has the property that there is always a path $\gamma$ that achieves the infimum, then $(X,d)$ is a \emph{geodesic space}, and one such path $\gamma$ achieving the infimum is called a \emph{geodesic}.
\end{definition}

We now recall the concepts of a \emph{quasi-isometric embedding} and of a \emph{quasi-geodesic}.

\begin{definition}
\label{def quasi-isometry}
    Let $X_{1}$ and $X_{2}$ be two metric spaces, $\ell,L \geq 0$ and $k \geq 1$, $f: X_{1} \longrightarrow X_{2}$ a map.
    \begin{itemize}
        \item The map $f$ is a \emph{$(k,\ell)$-quasi-isometric embedding} if for every $x,y \in X_{1}$ we have:\[ \frac{1}{k}d_{X_{2}}(f(x),f(y))-\ell \leq d_{X_{1}}(x,y) \leq kd_{X_{2}}(f(x),f(y))+\ell .\]
        \item If it also holds that for every $x_2 \in X_2$ there is some $x_1 \in X_1$ such that $d_{X_2}(f(x_1),x_2) \leq \ell$, then $f$ is a \emph{$(k,\ell)$-quasi-isometry}.
        \item The map $f$ is an \emph{$L$-local $(k,\ell)$-quasi-isometric embedding} if its restriction to any subset of diameter at most $L$ is a $(k,\ell)$-quasi-isometric embedding. 
    \end{itemize}
\end{definition}

\begin{definition}
\label{def quasi-geodesic}
    Let $I \subseteq \mathbb{R}$ be an interval and $\gamma: I \longrightarrow X$ a path. If $\gamma$ is a $(k,\ell)$-quasi-isometric embedding we call it a \emph{$(k,\ell)$-quasi-geodesic}. If it is an $L$-local $(k,\ell)$-quasi-isometric embedding we call it an \emph{$L$-local $(k,\ell)$-quasi-geodesic}. 
\end{definition}

\begin{remark}Note that for $k=1$ and $\ell=0$, $f$ is a genuine isometric embedding in Definition~\ref{def quasi-isometry}, and $\gamma$ is a genuine geodesic in Definition~\ref{def quasi-geodesic}.
\end{remark}

\begin{remark}\label{existence of qg in length} We will repeatedly make use of the following very useful fact about length spaces: if $(X,d)$ is such a space, by definition of the infimum, for every $x,y \in X$ and every $\ell>0$, there exists a path $\gamma : [a,b]\rightarrow X$ such that $d(x,y)\leq L_d(\gamma)\leq d(x,y)+\ell$. After reparametrizing $\gamma$ (by arc length) if necessary, we can assume that $a=0$ and $b=L_d(\gamma)$ and thus that $L_d(\gamma)=\vert b-a\vert$. Hence, $\gamma$ is a $(1,\ell)$-quasi-geodesic (and thus it is a $(k,\ell)$-quasi-geodesic for every $k\geq 1$).
\end{remark}

\subsection{Hyperbolic metric spaces}
\label{subsection hyperb metric spaces}

 For a metric space $X$, a point $x$ in $X$ and a subset $Y$ of $X$, we will write \[ d_{X}(x,Y)= \inf_{y \in Y} \{ d_{X}(x,y) \} \] for the distance between $x$ and $Y$. 
 Also, for a subset $Y$ of $X$ we will write $\text{diam}(Y)$ for the diameter of $Y$, that is, \[ \text{diam}(Y)=\sup_{y,y' \in Y}(d_{X}(y,y')). \] We will put $B_{X}(x,r)$ (or simply $B(x,r)$ if the metric space $X$ is clear by context) for the ball of radius $r$ centered at $x$.

\begin{definition}
\label{gromov prod def}
    Let $x$, $y$ and $z$ be three points of $X$. The \emph{Gromov product of $x$ and $y$ with respect to $z$} is \[\langle x,y \rangle_{z}=\frac{1}{2} \{ d(x,z)+d(y,z)-d(x,y) \}.\]
\label{delta hyp def}A metric space $X$ is said to be \emph{$\delta$-hyperbolic} (in the sense of Gromov) if for every four points $x,y,z,t \in X$ we have \[\langle x,z \rangle_{t} \geq \min \{ \langle x,y \rangle_{t} , \langle y,z \rangle_{t} \} - \delta.\] We will say that $X$ is \emph{hyperbolic} if it is $\delta$-hyperbolic for some $\delta \geq 0$.
\end{definition}

\begin{remark}
  For simplicity of notation, from now on we assume that the hyperbolicity constant $\delta$ is positive. However, notice that this is not a serious constraint: if $X$ is $\delta$-hyperbolic for $\delta \geq 0$, then it is $\delta '$-hyperbolic for every $\delta ' \geq \delta$. In particular, a 0-hyperbolic space can be thought of as being $\delta$-hyperbolic for arbitrarily small $\delta$.
\end{remark}

\begin{definition}
\label{quasi-geodesic n-gons}
    Let $X$ be a metric space, $x_1 , \cdots , x_n$ distinct points of $X$ and $l \geq 0$. A \emph{$(1,l)$-quasi-geodesic $n$-gon with vertices $x_1 , \cdots , x_n$} is the union of the image of $n$ paths $\gamma_{x_i, x_{i+1}}$ for $1 \leq i \leq n$ (and the value of $i$ is taken $\pmod{n}$) such that the initial point of $\gamma_{t,t'}$ is $t$, the endpoint is $t'$ and each path is a $(1,l)$-quasi-geodesic.
\end{definition}

From now on, we will not distinguish between paths (that are actually maps from an interval of the real line to $X$) from their images in $X$, and we will assume that these paths are parametrized by arc length. Notice that if $X$ is a length space, we immediately get from Remark~\ref{existence of qg in length} that for any $x_1 , \cdots , x_n \in X$ and any $l >0$, there exists a $(1,l)$-quasi-geodesic $n$-gon with vertices $x_1 , \cdots , x_n$. We will denote one such $n$-gon by $[x_1, \cdots , x_n]_{\ell}$ and its sides by $[x_i,x_{i+1}]_{l}$.

If $X$ is a $\delta$-hyperbolic geodesic metric space, then every geodesic triangle in $X$ is \emph{$2\delta$-thin}, that is, for every geodesic triangle in $X$, every side of the triangle is contained in the closed $2\delta$-neighbourhood of the union of the other two sides \cite[Lemma 11.28]{drutu_kapovich}.

We recall now a similar result for quasi-geodesic quadrangles for later reference.

\begin{lemma}\emph{\cite[Lemma 3.13]{amelio_andre_tent}}
\label{distances within quadrangles}
    Let $X$ be a $\delta$-hyperbolic length space, $[p,q,r,s]_{\ell}$ a $(1,\ell)$-quasi-geodesic quadrangle with $d(p,q)=d(r,s)$. Then, for any pair of points $x \in [p,q]_{\ell}$ and $y \in [r,s]_{\ell}$ with $d(p,x)=d(s,y)$ we have that \[ d(x,y) \leq 5 \max (\{ d(s,p) \, , \, d(q,r) \}) + 16 \delta + 38\ell. \]
\end{lemma}

We now state a result on quasi-geodesics on hyperbolic spaces, which is (a version of what is) called in the literature \emph{stability of quasi-geodesics}.

\begin{lemma}\emph{\cite[Corollary 2.6]{coulon_2}}
\label{res: stability (1,l)-quasi-geodesic}
	Let $l_0 \geq 0$ be a positive real number. There exists a positive number $L=L(l_0,\delta)$ depending only on $\delta$ and $l_0$ such that the following holds. Let $l \leq l_0$ and $\gamma : I \longrightarrow X$ be an $L$-local $(1,l)$-quasi-geodesic.
	\begin{enumerate}
		\item The path $\gamma$ is a $(2,l)$-quasi-geodesic.
		\item For every $t,t',s \in I$, such that $t \leq s \leq t'$, we have $ \langle \gamma(t),\gamma(t') \rangle_{\gamma(s)} \leq l/2 + 5 \delta$.
		\item For every $x \in X$, for every $y,y'$ lying on $\gamma$, we have $d(x,\gamma) \leq \langle y,y' \rangle _x + l + 8 \delta$.
		\item The Hausdorff distance between $\gamma$ and any other $L$-local $(1,l)$-quasi-geodesic joining the same endpoints, possibly in $\partial X$, is at most $2l+5\delta$.
	\end{enumerate}
\end{lemma}
		
Using a rescaling argument, one can see that the best value for the parameter $L=L(l,\delta)$ satisfies the following property: for all $l,\delta \geq 0$ and $\lambda >0$, $L(\lambda l, \lambda \delta) = \lambda L(l,\delta)$. With this in mind, we can state the following definition.

\begin{definition}
\label{def parameter LS}
    Let $L(l, \delta)$ be the best value of the parameter $L$ as provided by Lemma \ref{res: stability (1,l)-quasi-geodesic}. We denote by $L_S$ the smallest positive integer larger than 500 and such that $L(10^5\delta,\delta) \leq L_S\delta$.
\end{definition}

Notice from the discussion preceding Definition \ref{def parameter LS} that the value of $L_S$ does not depend on $\delta$.

\subsubsection{Quasi-convex and strongly quasi-convex subsets} 

We now define the concepts of a \emph{quasi-convex} subset of a metric space and of a \emph{strongly quasi-convex} subset of a hyperbolic length space. For a more comprehensive overview, see \cite[Subsection 2.3]{coulon_1}.

For a subset $Y$ of a metric space $X$, we  write $Y^{\alpha}$ (respectively, $Y^{+\alpha}$) for the open (respectively, closed) $\alpha$-neighbourhood of $Y$.

\begin{definition}
\label{quasi convex def}
    Let $X$ be a metric space, $\alpha \geq 0$. A subset $Y$ of $X$ is \emph{$\alpha$-quasi-convex} if for every pair of points $y,y' \in Y$ and every point $x \in X$ we have that \[ d(x,Y) \leq \langle y,y' \rangle_{x} + \alpha. \]
\end{definition}

\begin{remark}
    If $X$ is a geodesic space, the usual definition of an $\alpha$-quasi-convex subset $Y$ is that every geodesic joining two points of $Y$ is contained in $Y^{+\alpha}$. If $X$ is a $\delta$-hyperbolic geodesic space, a subset is $\alpha$-quasi-convex in the usual sense if and only if it is $(\alpha + 4\delta)$-quasi-convex in the sense of Definition~\ref{quasi convex def}. 
\end{remark}

\begin{definition}
\label{strong quasi convex def}
    Let $X$ be a $\delta$-hyperbolic length space, $\alpha \geq 0$. Let $Y$ be a subset of $X$ connected by rectifiable paths. Denote by $d_{Y}$ the length metric on $Y$ induced by the restriction of the length structure on $X$ to $Y$ (see \cite[Section 2]{burago_burago_ivanov} for the precise definition of a length structure). The subset $Y$ is said to be \emph{strongly quasi-convex} if it is $2\delta$-quasi-convex and for every pair of points $y,y' \in Y$ we have that \[ d_{Y}(y,y') \leq d_{X}(y,y') + 8 \delta . \]
\end{definition}

We now state some useful facts about quasi-convex subspaces of a hyperbolic space.

\begin{lemma}\emph{\cite[Chapitre 10, Proposition 1.2]{coornaert_delzant_papadopoulos}}
\label{neighbourhood of a convex subset}
    Let $Y$ be an $\alpha$-quasi-convex subset of a $\delta$-hyperbolic space $X$, and let $A \geq \alpha$. Then, we have that $Y^{+A}$ is $2 \delta$-quasi-convex
\end{lemma}

\begin{lemma}\emph{\cite[Lemma 2.13]{coulon_2}}
\label{intersection of quasi-convex subsets}
    Let $Y_1, \dots , Y_m$ be a collection of subsets of a $\delta$-hyperbolic space $X$ such that $Y_j$ is $\alpha _j$-quasi-convex for $j \in \{1, \dots , m\}$. For all $A \geq 0$ we have that \[ \diam (Y_1 ^{+A} \cap \dots \cap Y_m ^{+A} ) \leq \diam (Y_1 ^{+\alpha _1 + 3 \delta} \cap \dots \cap Y_m ^{+\alpha _m + 3 \delta} ) + 2A + 4 \delta . \]
\end{lemma}

\begin{definition}
\label{definition hull}
    Let $X$ be a $\delta$-hyperbolic length space, $Y$ a subset of $X$. The \emph{hull of $Y$}, denoted by $\text{hull}(Y)$, is the union of all $(1,\delta)$-quasi-geodesics joining two points of $Y$.
\end{definition}

\begin{lemma}\emph{\cite[Lemma 2.15]{coulon_2}}
\label{hull is quasi-convex}
    Let $X$ be a $\delta$-hyperbolic length space, $Y$ a subset of $X$. The hull of $Y$ is $6\delta$-quasi-convex.
\end{lemma}

\subsubsection{The boundary at infinity} 

Let $X$ be a $\delta$-hyperbolic metric space, and $x \in X$. A sequence $(y_{n})_{n \in \mathbb{N}}$ is said to \emph{converge to infinity} if $\langle y_{m},y_{m'} \rangle_{x}$ tends to infinity as $m$ and $m'$ tend to infinity. Note that the hyperbolicity of the space gives that this does not depend on the choice of $x$. The set $\mathcal{S}$ of sequences converging to infinity is endowed with a relation $R \subseteq \mathcal{S}^{2}$ defined as follows: two sequences $(y_{n})_{n \in \mathbb{N}}$ and $(z_{n})_{n \in \mathbb{N}}$ in $\mathcal{S}$ are related if \[ \lim_{n \to \infty} \langle y_{n},z_{n} \rangle_{x}= + \infty. \] Again, hyperbolicity gives that this is in fact an equivalence relation.

\begin{definition}
\label{definition of the boundary}
    Let $X$ be a hyperbolic metric space, $\mathcal{S}$ the set of sequences of points of $X$ converging to infinity. The \emph{boundary at infinity of $X$}, denoted as $\partial X$, is the quotient of $\mathcal{S}$ by the equivalence relation $R$.
\end{definition}

This definition does not depend on the choice of the base point $x$ since $X$ is hyperbolic. We will write $[(x_{m})_{m \in \mathbb{N}}]$ for the equivalence class of the sequence $(x_{m})_{m \in \mathbb{N}}$. For a subset $Y$ of $X$, we will write $\partial Y$ for the set of elements of $\partial X$ that are limits of sequences of points of $Y$.

\begin{remark}
    Notice that, in case $X$ is a proper geodesic metric space, Definition \ref{definition of the boundary} coincides with the definition of the boundary using equivalence classes of geodesics, in the following sense: write $\partial ' X$ for the geodesic boundary, then, there is a homeomorphism $h : X \cup \partial ' X \longrightarrow X \cup \partial X$ extending the identity on $X$. 
\end{remark}

By construction, if a group $G$ acts by isometries on a hyperbolic space $X$ (see Definition \ref{definition action by isometries} below) this action extends in a natural way to an action on the boundary $\partial X$: for $ \eta = [(x_{m})_{m \in \mathbb{N}}]\in\partial X$, put $g \cdot \eta = [(g \cdot x_{m})_{m \in \mathbb{N}}] $.

\subsection{Group actions on hyperbolic spaces}
\label{subsubsection isom of hyp spaces}

Fix throughout this subsection a group $G$ and a $\delta$-hyperbolic length space $X$. We begin by recalling the definition of an action by isometries of a group $G$ on a metric space $(X,d)$.

\begin{definition}
\label{definition action by isometries}
    Let $G$ be a group and $(X,d)$ be a metric space. We say that the group $G$ acts by isometries (or simply that it acts) on the metric space $(X,d)$ if $G$ acts on the underlying set $X$ in such a way that for all $g \in G$ and for all $x,y \in X$ we have that $d(x,y)=d(g \cdot x , g \cdot y)$.
\end{definition}

 Since all actions on metric spaces under consideration in this article will be actions by isometries, for simplicity of notation we may omit mentioning explicitly the metric $d$ and just talk about an action of a group $G$ on a metric space $X$.
 
 For a group $G$ acting by isometries on a hyperbolic length space $X$, we denote by $\partial G$ the set of accumulation points of $G \cdot x$ in $\partial X$ (note again that this definition does not depend on the choice of $x$). Then either one (and hence every) orbit of $G$ is bounded or $\partial G$ is non-empty (see for example \cite[Proposition 3.5]{coulon_1}).
 
Recall that if $g$ is an isometry of $X$, then $g$ is of one of the following types: 
    \begin{itemize}
        \item elliptic, i.e. $\partial \langle g \rangle$ is empty,
        \item parabolic, i.e.  $\partial \langle g \rangle$ has exactly one element, or
        \item loxodromic, i.e.  $\partial \langle g \rangle$ has exactly two elements.
    \end{itemize}

For a loxodromic isometry $g$, the two elements of $\partial \langle g \rangle$ are \[ g^{- \infty} = [(g^{-m} \cdot x)_{m \in \mathbb{N}}] \, \, \text{and} \, \, g^{+ \infty} = [(g^{m} \cdot x)_{m \in \mathbb{N}}]. \] 

\begin{lemma}\emph{\cite[Chapitre 10, Proposition 6.6]{coornaert_delzant_papadopoulos}}
\label{only two points fixed by loxo}
    The points $g^{- \infty}$ and $g^{+ \infty}$ are the only points of $\partial X$ fixed by $g$.
\end{lemma}

 Conversely, we have the following well-known lemma.
 
\begin{lemma} \emph{\cite[Proposition 3.6]{coulon_1}}
\label{two points in boundary implies loxodromic}
     If $\partial G$ has at least two points, then $G$ contains a loxodromic isometry.
\end{lemma}

Next we introduce two notions of translation lengths that can be used, among other things, to give a characterization of loxodromic isometries.

\begin{definition}
\label{translation length definition}
    Let $g$ be an isometry of $X$. The \emph{translation length} of $g$, denoted by $[g]_{X}$ (or simply $[g]$ if the space $X$ is clear from the context) is \[ [g]_{X}= \inf \{ d(x, g \cdot x) \, : \, x \in X \}. \] The \emph{asymptotic translation length} of $g$, denoted by $[g]^{\infty}_{X}$ (or simply $[g]^{\infty}$) is \[ [g]^{\infty}_{X}= \lim_{n \to + \infty} \frac{1}{n}d(x, g^{n} \cdot x).\]
\end{definition}

Once again, notice that the definition of the asymptotic translation length does not depend on the choice of $x$. These two concepts are related as follows:

\begin{lemma} \emph{\cite[Chapitre 10, Propositions 6.3 and 6.4]{coornaert_delzant_papadopoulos}}
\label{asympt and trans length relation}
    The quantities $[g]$ and $[g]^{\infty}$ satisfy \[ [g]^{\infty} \leq [g] \leq [g]^{\infty} + 32 \delta, \]
 and $g$ is loxodromic if and only if $[g]^{\infty} > 0$.
\end{lemma}

We now state a result from \cite{coulon_2} for later reference.

\begin{lemma}\emph{\cite[Lemma 2.26]{coulon_2}}
\label{lemma tr distance gromov product}
    Let $x,x'$ and $y$ be three points of $X$, and let $g$ be an isometry of $X$. Then, we have that \[ d(y , g \cdot y) \leq \max (\{ d(x,g \cdot x) , d(x', g \cdot x') \}) + \langle x,x' \rangle_y + 6 \delta. \]
\end{lemma}

\subsubsection{The axis of an isometry}
\label{subsubsection axis of isom}

We now introduce the concepts of the \emph{axis} of an isometry and the \emph{cylinder} of a loxodromic isometry, which will play an important role in the small cancellation results introduced in Chapter~\ref{section small cancellation}.
For a hyperbolic length space $X$ and two distinct points $\zeta$ and $\eta$ of $\partial X$, we say that a path $\gamma: \mathbb{R} \longrightarrow X$ joins $\zeta$ and $\eta$ if \[ \{ [(\gamma(-m))_{m \in \mathbb{N}}], [(\gamma(m))_{m \in \mathbb{N}}] \} = \{ \zeta, \eta \}. \]


\begin{definition}
\label{def axis of isometry}
    The \emph{axis} of an isometry $g$ of $X$, denoted as $A_{g}$, is the set \[ \{ x \in X \, : \, d(x, g \cdot x) < [g] + 8 \delta \}. \]
\end{definition}

Note that the axis is defined for any isometry of $G$, not necessarily a loxodromic one.

\begin{lemma}
\label{distance in terms of distance to axis}
    Let $g$ be an isometry of $X$ and $x \in X$. The following facts hold.
    \begin{itemize}
        \item The axis $A_g$ is $10 \delta$-quasi-convex.
        \item If $d(x,g \cdot x) \leq [g]+A$, then $d(x,A_g) \leq A/2 + 3 \delta$.
    \end{itemize}
\end{lemma}

We now introduce the concept of an \emph{$l$-nerve} of an isometry $g$. It can be thought of as a `nice' $g$-invariant bi-infinite quasi-geodesic that can be used to simplify some arguments (see for example the proof of Proposition \ref{int of axis in term of Omega}).

\begin{definition}\emph{(see \cite[Definition 3.3]{coulon_1})}
\label{definition nerve}
    Let $g$ be an isometry of $X$ and $l \geq 0$. We say that a path $\gamma : \mathbb{R} \longrightarrow X$ is an \emph{$l$-nerve} of $g$ if there is some number $T$ such that $[g] \leq T \leq [g]+l$, $\gamma$ is a $T$-local $(1,l)$-quasi-geodesic and for every $t \in \mathbb{R}$ we have that $\gamma (t+T)= g \cdot \gamma (t)$. The parameter $T$ is called the \emph{fundamental length} of $\gamma$.
\end{definition}

We collect some facts about $l$-nerves of an isometry that appeared in \cite{coulon_1}.

\begin{lemma}
\label{existence of nerves}
    Let $g$ be an isometry of $X$. Then, for every $l>0$, there exists an $l$-nerve of $g$. Moreover, if $[g]>L_S \delta$ and $l \leq 10^5 \delta$, then an $l$-nerve $\gamma$ is $(l+8 \delta)$-quasi-convex. Furthermore, $\gamma$ joins the accumulation points of $\langle g \rangle$ at $\partial X$.
\end{lemma}

\begin{definition}
\label{def cylinder of isometry}
    Let $g$ be a loxodromic isometry of $X$. We denote by $\Gamma_{g}$ the union of all $L_{S} \delta$-local $(1, \delta)$-quasi-geodesics joining $g^{- \infty}$ and $g^{+ \infty}$. The \emph{cylinder} of $g$, denoted as $Y_{g}$ is the open $20\delta$-neighbourhood of $\Gamma_{g}$.
\end{definition}

The next result relates the axis and the cylinder of a loxodromic isometry (see \cite[Lemmas 2.32 and 2.33]{coulon_2} and \cite[Lemma 3.13]{coulon_1}).

\begin{lemma}
\label{axis vs cylinder}
    Let $g$ be a loxodromic isometry of $X$, $A_{g}$ the axis of $g$ and $Y_{g}$ the cylinder of $g$.
    \begin{enumerate}
        \item \label{cylinder contained in neoghbourhood of axis} Let $Y$ be a $g$-invariant $\alpha$-quasi-convex subset of $X$. Then, $Y_g$ is contained in the $(\alpha + 42 \delta)$-neighbourhood of $Y$. In particular, $Y_{g} \subseteq A_{g}^{+52 \delta}$.
        \item Suppose that $[g]> L_S \delta$. Let $l \leq \delta$ and $\gamma$ be an $L_S \delta$-local $(1,l)$-quasi-geodesic joining the accumulation points of $\langle g \rangle $ in $\partial X$. Then, $A_g$ is contained in the $(l+9 \delta)$-neighbourhood of $\gamma$. In particular, $A_{g} \subseteq Y_{g}$.
        \item The cylinder $Y_{g}$ is a strongly quasi-convex subset of $X$.
    \end{enumerate} 
\end{lemma}

We now state Lemma \ref{res : quasi-geodesic behaving like a nerve}, which, broadly speaking, says that a quasi-geodesic near the axis of an isometry behaves almost like a nerve.

\begin{lemma}\emph{\cite[Lemma 2.34]{coulon_2}}
\label{res : quasi-geodesic behaving like a nerve}
	Let $g$ be an isometry of $X$ such that $ [g] > L_S \delta$, let $l \leq \delta$ and let $\gamma: [a,b] \longrightarrow X$ be a $[g]$-local $(1,l)$-quasi-geodesic contained in the $C$-neighborhood of $A_g$. Then there exists $\epsilon \in \{ \pm 1 \}$ such that for every $s \in  [a,b -[g]]$  we have \[ d(g^\epsilon \cdot \gamma(s) ,\gamma(s+  [g])) \leq 4C + 4l + 88\delta.\]
\end{lemma}

\subsubsection{Elementary subgroups}\label{def element subgroup} 

Let $G$ be a group acting by isometries on $X$ and let $H$ be a subgroup of $G$. We say that $H$ is \emph{elementary} if $\partial H$ has at most two points. Otherwise, we say it is \emph{non-elementary}. We say that an elementary subgroup $H$ is
    \begin{itemize}
        \item \emph{elliptic} if its orbits are bounded (equivalently, if $\partial H$ is empty),
        \item \emph{parabolic} if $\partial H$ has exactly one point, or
        \item \emph{loxodromic} if $\partial H$ has exactly two points.
    \end{itemize}

Notice that any finite subgroup of $G$ is elliptic. We can associate to an elliptic subgroup a set of `almost fixed points' in the sense of the following definition. For a subset $S \subseteq G$ and a non-negative real number $r$, write $\fix (S,r)$ for the set $\{ x \in X \, : \, d(x, g \cdot x) \leq r \, \forall g \in G\}$.

\begin{definition}
\label{def charact subset of elliptic}
    Let $F$ be an elliptic subgroup of $G$. The \emph{characteristic set} of $F$ is $\fix (F, 11 \delta)$, that is, \[ C_F= \{ x \in X \, : \, \forall g \in F, \, d(g \cdot x,x) \leq 11 \delta \}. \]
\end{definition}

\begin{lemma}\cite[Proposition 2.36 and Corollaries 2.37 and 2.38]{coulon_2}
\label{charac subset of elliptic is quasi convex}
    Let $F$ be an elliptic subgroup of $G$. The characteristic set $C_F$ is non-empty and $9 \delta$-quasi-convex.

    Moreover, let $Y$ be a non-empty $F$-invariant $\alpha$-quasi-convex subset of $X$. Then, for every $A \geq \alpha$, the $A$-neighbourhood of $Y$ contains a point of $C_F$.
\end{lemma}

\subsubsection{Acylindrical group actions}

We now recall the notion of an \emph{acylindrical} group action on a metric space (a weakening of the proper and cocompactness property in the usual definition of a hyperbolic group that still allows for interesting consequences for groups admitting one such action). This notion goes back to Sela’s paper \cite{Sela}, where it was considered for groups acting on trees. In the context of general metric spaces, the following definition was introduced by Bowditch in \cite{bowditch}.


\begin{definition}
\label{acylindr definition}
Let $G$ be a group acting by isometries on a $\delta$-hyperbolic metric space $X$. The action is said to be \emph{acylindrical} if for every $ \varepsilon \geq 0$ there exist $M, L > 0$ such that for every $x,y \in X$ with $d(x,y) \geq L$ we have: \[ \lvert \{ g \in G : d(x,g \cdot x) \leq \varepsilon, \, d(y,g \cdot y) \leq \varepsilon \} \rvert \leq M .\]    
\end{definition}

\begin{remark}
\label{acylind characterization}
   By \cite[Proposition 5.31]{dahmani_guirardel_osin} it suffices to check this condition for $\varepsilon=100\delta$. Even though this result is stated for geodesic spaces, this also holds for length spaces (see \cite[Proposition 5.6]{coulon_4}). 
\end{remark}

The next two lemmas are the structural properties of an acylindrical action that will be the most relevant in the proof of our main result. Lemma \ref{lox subg in wpd is virt cyc} is stated in \cite{coulon_1} for the most general case of a WPD action.

\begin{lemma}\emph{\cite[Theorem 1.1]{osin}}
\label{no_parabolic_osin}
Let $G$ be a group acting acylindrically by isometries on a hyperbolic metric space. Then $G$ has no parabolic subgroup.
\end{lemma}


\begin{lemma}
\label{lox subg in wpd is virt cyc}
Let $G$ be a group acting acylindrically on a hyperbolic length space $X$. Then, every loxodromic subgroup $H$ of $G$ is contained in a unique maximal loxodromic subgroup of $G$, namely the setwise stabilizer of the pair $\partial H\subset \partial G$, denoted by $M_G(H)$. Moreover, $M_G(H)$ (and thus $H$) are virtually cyclic.
\end{lemma}

Recall that an infinite virtually cyclic group $H$ maps either onto $\mathbb{Z}$ or onto $D_{\infty}$, with finite kernel (which is the unique maximal normal finite subgroup of $H$). In the first case, we say that $H$ is of \emph{cyclic type}, and in the second case we say that $H$ is of \emph{dihedral type}. An element of $H$ is called \emph{primitive} if it maps to an element of $\mathbb{Z}$ (in the first case) or of $D_{\infty}$ (in the second case) that has infinite order and does not admit a proper root. This terminology can be extended to a loxodromic element $h$ of $G$: the element $h$ is called \emph{primitive} if it is primitive as an element of the virtually cyclic subgroup $M_G(\langle h\rangle)$ (equivalently, $h$ has minimal asymptotic translation length among the loxodromic elements of $M_G(\langle h\rangle)$).

The next lemma relates the cylinder of a loxodromic isometry with the characteristic subset of finite subgroups normalized by this element.

\begin{lemma}\cite[Lemma 3.33]{coulon_1}
\label{charac subs of max finite sgrp}
    Let $G$ be a group with a WPD action by isometries on a hyperbolic length space $X$. Let $g$ be a loxodromic element of $G$ and $H$ a subgroup fixing the set $\{g^{\pm \infty}\}$ pointwise. Let $F$ be the maximal normal finite subgroup of $H$. Then, the cylinder $Y_g$ is contained in the $51 \delta$-neighbourhood of the characteristic subset $C_F$.
\end{lemma}

\subsection{Invariants of the group action}
\label{subsection invariants}

Fix now a group $G$ acting acylindrically on a $\delta$-hyperbolic space $X$. In order to control the order of the torsion we are imposing in the quotients that we can obtain with the small cancellation results in Section \ref{subsection small cancellation}, we need to control certain invariants of the action of the group in our hyperbolic space. The first of them, $r_{\text{inj}}(Q,X)$, already appeared in the small cancellation assumptions in \cite{coulon_1}, \cite{coulon_4} and \cite{amelio_andre_tent}. For the sake of completeness, we will reintroduce it. The other invariants, $\tau(G,X)$ and $\Omega(G,X)$, are modifications of $\nu(G,X)$ and $A(G,X)$ (respectively) from the aforementioned papers designed to deal with even torsion, under the additional assumption that the even order elements of the group are in some sense `mild', captured by the notion of tameness, also introduced in this section.

 \begin{definition}
 \label{def rinj}
     Let $Q$ be a subset of $G$. The \emph{injectivity radius} of $Q$ is \[ r_{\text{inj}}(Q,X)= \inf \{ \left[ g\right]^{\infty} : g \in Q, \, \, g \, \, \text{loxodromic} \}. \]
 \end{definition}

 \begin{definition}
\label{def param nu}
    The invariant $\nu (G,X)$ (or simply $\nu$) is the smallest positive integer $m$ satisfying the following property: let $g$ and $h$ be two isometries of $G$ with $h$ loxodromic. If $g$, $h^{-1}gh$,..., $h^{-\nu}gh^\nu$ generate an elliptic subgroup, then $g$ and $h$ generate an elementary subgroup of $G$.
\end{definition}

The proof of Lemma 6.12 in \cite{coulon_1} yields the following bound for $\nu(G,X)$ for acylindrical actions with positive injectivity radius.

\begin{lemma}
\label{nu finite acyl}
    Assume the action of $G$ on $X$ is acylindrical and with positive injectivity radius. Call $L$ and $M$ the parameters in the definition of an acylindrical action (Definition~\ref{acylindr definition}) corresponding to $\varepsilon = 97\delta$, and put $M'$ as the smallest positive integer such that $M' r_{\text{inj}}(G,X) \geq L$. Then, $\nu(G,X) \leq M'+M$.
\end{lemma}

For  $g_{1}, \dots , g_{m}\in G$ we put \[A(g_{1},\dots, g_{m}) =  \text{diam}\left(A_{g_{1}}^{+13\delta} \cap \dotsc \cap A_{g_{m}}^{+13\delta}\right).\] 

\begin{definition}
\label{def param A}
    Assume the action of $G$ on $X$ has finite parameter $\nu = \nu(G,X)$. 
       
    We denote by $\mathcal{A}$ the set of $(\nu+1)$-tuples $(g_{0},\dots,g_{\nu})$ such that $g_{0},\dots,g_{\nu}$ generate a non-elementary subgroup of $G$ and for all $j \in \{0, \dots, \nu \}$ we have $[g_{j}] \leq L_{S}\delta$. We define \[ A(G,X) = \sup_{(g_{0},\dots,g_{\nu}) \in \mathcal A} (\{A\left(g_{0}, \dots,g_{\nu}\right)\}). \]
		
\end{definition}

We now introduce the concept of a \textit{tame} action. The structural consequences that this assumption has on the loxodromic subgroups of the action will be key when proving that the invariant $\nu$ is well-behaved when developing our small cancellation theory in Section \ref{section small cancellation}.

\begin{definition}
\label{def tame action}
    Let $g$ be a loxodromic element of $G$, let $E(g)=M_G(\langle g \rangle)$ be the maximal loxodromic subgroup containing $g$ and $F(g)$ be the maximal normal finite subgroup of $E(g)$. We say that the action of $G$ on $X$ is \emph{tame} if $G$ contains no subgroup of order 4 and, for every loxodromic element $g \in G$, $F(g)$ has order at most 2.
\end{definition}

\begin{remark}
\label{virt cyc subgroups of tame action}
    Recall that, in virtue of the classification of loxodromic subgroups in acylindrical actions (Proposition \ref{lox subg in wpd is virt cyc} and the subsequent paragraph), we get a classification of loxodromic subgroups of $G$ for a tame acylindrical action. More concretely, one such loxodromic subgroup $H$ will fall in one of these three cases:
    \begin{enumerate}
        \item $H \cong \mathbb{Z}$;
        \item $H \cong C_2 \times \mathbb{Z}$; or
        \item $H \cong D_{\infty}$;
    \end{enumerate}
\end{remark}

We now include for later reference an easy lemma for normalizers of finite elements of a tame acylindrical action.

\begin{lemma}
\label{cent of element of order >3 in wstprime is elliptic}
    Let $G$ be a group with a tame acylindrical action on a hyperbolic space $X$ and $F \leq G$ be a subgroup of finite order $\geq 3$. Then, $N_G (F)$ is elliptic.
\end{lemma}

\begin{proof}
    Notice that no loxodromic element of $G$ can normalize $F$ (since the action of $G$ on $X$ is tame, the maximal normal finite subgroup of every loxodromic subgroup is of order at most 2). Therefore, $N_G (F)$ is neither non-elementary nor loxodromic, since in both cases it would, in fact, contain a loxodromic element. Thus, $N_G (F)$ must be elliptic.
\end{proof}

For the remainder of the section, we assume that the action of $G$ on $X$ is tame. For simplicity of notation, we will now introduce a new parameter, $\tau(G,X)$, which will be key to control $\Omega (G,X)$ (a modified version of the parameter $A(G,X)$).

\begin{definition}
\label{def parameter tau}
    The parameter $\tau$ is defined as $\tau(G,X) = \max \{\nu(G,X), 3\}$.
\end{definition}

\begin{proposition}\emph{(Compare \cite[Proposition 3.41]{coulon_1})}
\label{chains generating loxodromic sgrps}
    Let $G$ be a group acting acylindrically on a hyperbolic space $X$. Suppose that the action is tame and that $\nu(G,X)$ is finite. Let $g$ and $h$ be two elements of $G$ with $h$ loxodromic and let $m \geq \tau$ be an integer such that $g , \, h^{-1}gh, \, \dots \, , h^{-m}gh^{m}$ generate an elementary subgroup of $G$. Then, $g$ and $h$ generate an elementary subgroup of $G$.
\end{proposition}

\begin{proof}
   If $g$ is trivial, then it is immediate that $\langle g,h \rangle$ is elementary, so we may assume $g$ to be non-trivial.
    
    Write $H$ for the subgroup generated by $g , \, h^{-1}gh, \, \dots \, , h^{-m}gh^{m}$.
    
    Assume first that $g$ is loxodromic. Then so is $H$, and thus all of $H$ fixes the set of accumulation points $\partial H = \{ g^{ \pm \infty } \}$. As a loxodromic element of $H$, $h^{-1}gh$ also fixes pointwise $\partial H$. Therefore, $g$ fixes pointwise $h \cdot \partial H$, but since $g^{ \pm \infty }$ are the only two points of $\partial G$ fixed by $G$, then $h$ stabilizes $\partial H$, so $\langle g , h \rangle$ is contained in the elementary subgroup of $G$ stabilizing $\partial H$.

    Assume now that $H$ (and thus in consequence also $g$) is elliptic. Then, since $m \geq \nu(G,X)$, by definition $g$ and $h$ generate an elementary subgroup.
    
    Finally, assume that $H$ is loxodromic and that $g$ is elliptic. Let $p$ be the largest integer such that $g , \, \dots \, , h^{-p}gh^{p}$ generate an elliptic subgroup $E$. If $p \geq \nu(G,X)$, then again as in the previous case we get that by definition $g$ and $h$ generate an elementary subgroup. Thus, we may assume that $p \leq \nu (G,X)-1 \leq m-1$. If for some $k$ we have that $g=h^{-k}gh^{k}$, then $g$ and $h^{k}$ centralize each other, so $g$ fixes the accumulation points of $h^{k}$ (which coincide with those of $h$) and thus again $g$ and $h$ generate an elementary subgroup. Therefore, we can assume that the elements of the chain are pairwise distinct. Now, $E$ is an elliptic subgroup contained in a loxodromic subgroup, so it follows from the classification of Remark \ref{virt cyc subgroups of tame action} that in fact $p =0 \leq m-2$. Now, $E_1=\langle E, h^{-1}Eh \rangle$ is a loxodromic subgroup with accumulation points $\partial H$, and the same holds for $E_2 = h^{-1} E_1 h$. A loxodromic element of $E_2$ fixes pointwise $\partial H$, and is necessarily an $h$-conjugate of a loxodromic element $h'$ of $E_1$. But then, $h'$ has to fix pointwise $\partial H$ (as an element of $E_1$) and $h \cdot \partial H$ (as an $h$-conjugate of an element of $E_2$). Therefore, by Lemma \ref{only two points fixed by loxo}, $h$ fixes $\partial H$, and since so does $g$, then $\langle g,h \rangle$ is contained in the elementary subgroup of $G$ stabilizing $\partial H$.
\end{proof}

Now, we introduce parameter $\Omega(G,X)$ (a modification of parameter $A(G,X)$), the final one needed to control the small cancellation assumptions.

\begin{definition}
\label{def param Omega}
    Assume the action of $G$ on $X$ has finite parameter $\tau = \tau(G,X)$.

    We denote by $\mathcal{A}'$ the set of $(\tau+1)$-tuples $(g_{0},\dots,g_{\tau})$ such that $g_{0},\dots,g_{\tau}$ generate a non-elementary subgroup of $G$ and for all $j \in \{0, \dots, \tau \}$ we have $[g_{j}] \leq L_{S}\delta$. We define \[ \Omega(G,X) = \sup_{(g_{0},\dots,g_{\tau}) \in \mathcal A'} (\{A\left(g_{0}, \dots,g_{\tau}\right)\}). \]		
\end{definition}

The next two results are an adaptation to our context of Proposition 3.44 and Corollary 3.45 (respectively) of \cite{coulon_1}. The statements and the proofs are almost identical, modulo putting $\tau$ in place of $\nu$ and $\Omega(G,X)$ in place of $A(G,X)$ when appropriate. However, for the sake of completeness, we include proofs for both results.

\begin{proposition}
\label{int of axis in term of Omega}
    Let $g$ and $h$ be two elements of $G$ generating a non-elementary subgroup.

    \begin{enumerate}[label=(\arabic*)]
        \item \label{int of axis omega when tr length small} If $[g] \leq L_S \delta$, then $A(g,h)\leq \tau [h] + \Omega(G,X) + 154 \delta$.
        \item \label{int of axis omega general} In general, we have that \[ A(g,h) \leq [g] + [h] + \tau \max \{[g],[h]\} + \Omega(G,X) + 680 \delta. \]
    \end{enumerate}
\end{proposition}

\begin{proof}
    We will prove \ref{int of axis omega when tr length small} by contradiction. To this purpose, suppose that $A(g,h) > \tau [h] + \Omega(G,X) + 154 \delta$, and let $\eta \in (0,\delta)$ be such that \[ A(g,h) > \tau ([h] + \eta) + \Omega(G,X) + 4\eta + 154 \delta. \] If we had that $[h]\leq L_S \delta$, then by definition of $\Omega(G,X)$ we would have that $g$ and $h$ generate an elementary subgroup.

    Consider now $\gamma : \mathbb R \longrightarrow X$ an $\eta$-nerve of $h$ and denote by $T$ its fundamental length (see Definition \ref{definition nerve}). In particular, we have that $T \leq [h] + \eta$. By Lemma \ref{axis vs cylinder}, we see that $A_h$ is contained in the $ (\eta + 9\delta)$-neighbourhood of $\gamma$. Now, $\gamma$ is $9 \delta$-quasi-convex and $A_g$ is $10\delta$-quasi-convex, so by Lemma \ref{intersection of quasi-convex subsets} we get that \[ \diam (A_g^{+13 \delta} \cap \gamma^{+12\delta}) > \tau ([h] + \eta) + \Omega(G,X) + 2 \eta + 106 \delta. \]Therefore, there exist $x=\gamma(s)$ and $x'=\gamma(s')$ two points in $\gamma$ that are in the $25\delta$-neighbourhood of $A_g$ and such that \[ d_{X}(x,x')> \tau([h] + \eta) + \Omega(G,X) + 2 \eta + 82 \delta \geq \tau T + \Omega(G,X) + 2 \eta + 82 \delta. \]We may assume that $s<s'$ (after maybe replacing $h$ by $h^{-1}$). By Lemma \ref{res: stability (1,l)-quasi-geodesic}, we get that for all $t \in [s,s']$, $\langle x,x' \rangle_{\gamma(t)} \leq \eta / 2 + 5 \delta$. Now, by Lemma \ref{neighbourhood of a convex subset}, the $25 \delta$-neighbourhood of $A_g$ is $2\delta$-quasi-convex, and so $\gamma(t)$ lies in the $(\eta/2 + 32 \delta)$-neighbourhood of $A_g$. In consequence, the triangle inequality yields that
    \begin{equation}
    \label{axis of elements eq aux 1}
        d_{X}(\gamma(t), g \cdot \gamma(t)) \leq [g] + \eta + 72 \delta.
    \end{equation}

    Now, by the choice of $s$ and $s'$, there is some $t \in [s,s']$ such that $d_{X}(x, \gamma(t))= \Omega(G,X) + 2\eta + 82\delta$. We write $y= \gamma(t)$. We furthermore have that

    \begin{equation}
    \label{axis of elements eq aux 2}
        s'-t \geq d_X(x',y) \geq d_X(x,x')-d_X(y,x) \geq \tau T.
    \end{equation}

    Let $m \in \{0, \dots , \tau\}$. By the definition of an $\eta$-nerve, we get that $h^m \cdot x = \gamma(s+mT)$ and $h^m \cdot y = \gamma(t+mT)$. By Equation (\ref{axis of elements eq aux 2}), both $s +mT$ and $t+mT$ are in $[s,s']$, so by Equation (\ref{axis of elements eq aux 1}), \[ \max \{ d_X(gh^{m} \cdot x , h^{m} \cdot x), d_X(gh^{m} \cdot y , h^{m} \cdot y) \} \leq [h^{m}gh^{-m}] + \eta + 72\delta. \]

    By Lemma \ref{distance in terms of distance to axis}, we see that $x$ and $y$ are in the $(\eta /2 + 39\delta)$-neighbourhood of $h^m \cdot A_g$. Since this holds for every non-negative integer $m \leq \tau$, $x$ and $y$ are two points in \[ A_g ^{+ \eta/2 + 39\delta} \cap \dots \cap h^\tau \cdot A_g ^{+ \eta/2 + 39\delta}. \]Now, Lemma \ref{intersection of quasi-convex subsets} gives \[ A(g,hgh^{-1}, \dots , h^\tau g h^{-\tau}) \geq d_X (x,y) - \eta - 82\delta > \Omega(G,X). \]Furthermore, since the translation length is conjugation invariant, we obtain $[h^m g h^{-m}] \leq L_S \delta$, so by the definition of $\Omega(G,X)$, the elements $g,\dots, h^{\tau}gh^{-\tau}$ generate an elementary subgroup of $G$, and thus, by the definition of $\tau(G,X)$, so do $g$ and $h$.

    We now prove \ref{int of axis omega general}. By the previous point, we may assume that $[g],[h] \geq L_S \delta$. Without loss of generality, we may assume that $[h] \geq [g]$. Assume towards a contradiction that \[ A(g,h) > [g] + (\tau +1)[h] + \Omega(G,X) + 680 \delta. \]Let $\eta \in (0,\delta)$ be such that \[ A(g,h) >[g] + (\tau +1)[h] + \Omega(G,X) + 680 \delta + 15 \eta. \]
    Consider now $\gamma$ an $\eta$-nerve of $h$ and denote by $T$ its fundamental length. As before, $A_h$ is contained in the $(\eta +9 \delta)$-neighbourhood of $\gamma$, so \[ \diam (A_g^{+13 \delta} \cap \gamma^{+12\delta}) > [g] + (\tau + 1) [h] + \Omega(G,X) + 13 \eta + 632 \delta. \]In particular, there exist $x=\gamma(s)$ and $x' = \gamma(s')$ in the $25\delta$-neighbourhood of $A_g$ such that \[ d_{X}(x,x')> [g] +  (\tau +1)[h] + \Omega(G,X) + 13 \eta + 608 \delta.\]As in the previous case, we may assume that $s \leq s'$ and we get that the restriction of $\gamma$ to $[s,s']$ is contained in the $(\eta /2 + 32 \delta)$-neighbourhood of $A_g$. By Lemma \ref{res : quasi-geodesic behaving like a nerve} we have (after possibly replacing $g$ by $g^{-1}$) that for every $t \in [s,s']$, if $t \leq s' - [g]$ then \[ d_X(g \cdot \gamma(t), \gamma(t+[g])) \leq 6 \eta +222 \delta. \]Therefore, for every $t \in [s,s']$ with $t \leq s'-[g]-T$ we obtain \[ d_X(hg \cdot \gamma(t), gh \cdot \gamma(t)) \leq d_X(g \cdot \gamma(t + T), h \cdot \gamma(t+[g]))  + 6 \eta +222 \delta \leq 12 \eta + 444\delta. \]This means that the translation length of the element $u=h^{-1}g^{-1}hg$ is less than $L_S \delta$. Furthermore, for all $t \in [s,s']$, if $t \leq s' - [g] -T$, then $\gamma(t)$ is in the $(6\eta +225\delta)$-neighbourhood of $A_u$. Denote by $y=\gamma(t)$ a point such that $d_X(x',y)=[g]+T$. We get that \[ d_X(x,y) \geq d_X(x,x')-d(x',y) > \tau T + \Omega(G,X) + 12 \eta + 608 \delta, \]and both $x$ and $y$ are in the $(6 \eta + 225 \delta)$-neighbourhood of both $A_u$ and $A_h$. In consequence, \[ A(g,u) \geq d_X(x,y) - 12 \eta - 454 \delta > \tau [h] + \Omega(G,X) + 154 \delta. \]From point \ref{int of axis omega when tr length small} we have that $h$ and $u$ generate an elementary subgroup of $G$, and thus so do $h$ and $h'=g^{-1}hg$. Since $h$ is loxodromic, the only fixed points on the boundary of the loxodromic isometries $h$ and $h'$ are $\{ h^{\pm \infty}  \}$. Therefore, since $h$ must fix $g \cdot \{ h^{\pm \infty} \}$, then $g$ must fix $\{ h^{\pm \infty}  \}$ as well. In consequence, $h$ and $g$ generate an elementary subgroup, and we arrive at a contradiction.
\end{proof}

\begin{corollary}
\label{less than tau elements gen non-elem}
    Let $m \leq \tau(G,X)$ be an integer, let $g_0, \dots , g_m$ be elements of $G$ generating a non-elementary subgroup. Then, \[ A(g_0 , \dots , g_m) \leq (\tau +2) \max \{[g_0], \dots , [g_m]\} + \Omega(G,X) + 680 \delta. \]
\end{corollary}

\begin{proof}
    If we have that $[g_i] \leq L_S \delta$ for all $0 \leq i \leq m$, then by the definition of $\Omega(G,X)$ we get that $A(g_0 , \dots , g_m) \leq \Omega(G,X)$. If there is some $j$ such that $[g_j] > L_S \delta$, then $g_j$ is loxodromic. Now, suppose that the corollary is false, and that the elements $g_0, \dots , g_m$ generate a non-elementary subgroup and satisfy \[ A(g_0 , \dots , g_m) > (\tau +2) \max \{[g_0], \dots , [g_m]\} + \Omega(G,X) + 680 \delta. \] Then, for all $i \in \{0, \dots , m\}$, Proposition \ref{int of axis in term of Omega} applied to $g_i$ and $g_j$ gives that these elements generate an elementary subgroup, which is necessarily loxodromic (since it contains the loxodromic element $g_j$). Therefore, for all $i \in \{0, \dots , m\}$, $g_i$ is in the maximal elementary subgroup containing $g_j$, so $g_0, \dots , g_m$ generate an elementary subgroup of $G$, and we arrive at a contradiction.
\end{proof}

\begin{remark}
\label{invariants in rescaled space}
    If $G$ acts on a $\delta$-hyperbolic metric space $X$ and $\lambda X$ is a rescaling of $X$ (that is, a metric space with the same underlying set and distances multiplied by $\lambda$), then $\lambda X$ is a $\lambda \delta$-hyperbolic metric space endowed with an action of $G$. Moreover, the action of $G$ on $\lambda X$ will be tame if and only if so is the action of $G$ on $X$, and the same holds for acylindricity. Furthermore, the invariants satisfy $r_{\text{inj}}(Q, \lambda X)= \lambda r_{\text{inj}}(Q, X)$ (for any subset $Q$ of $G$), $\nu(G, \lambda X)= \nu(G,X)$, $\tau(G, \lambda X)= \tau(G,X)$, $A(G, \lambda X)= \lambda A(G,X)$ and $\Omega(G, \lambda X)= \lambda \Omega(G,X)$.
\end{remark}

\section{Small cancellation theory}
\label{section small cancellation}

The goal of this section is to introduce, adapt and redevelop some of the small cancellation methods from \cite{coulon_1} and \cite{coulon_4}. As a remainder, the key difference between the three settings is in the approach to control even torsion in the group under consideration. In \cite{coulon_1}, the groups are assumed to contain no elements of even order, and the goal is to obtain odd order periodic quotients of these groups. Meanwhile, in \cite{coulon_4}, the situation is quite the opposite: the goal is to obtain even order periodic exponents of groups (with the largest power of 2 dviding the exponent arbitrarily large).

In our case, we are in somewhat of an intermediate situation: the groups we want to consider will indeed have involutions, but the exponents we wish to impose are odd. This imposes \emph{a priori} extra difficulties, which can be thought of, in a very simplified way, as coming from the following fact: if a virtually cyclic subgroup $E$ contains an even order element generating a subgroup that is not in the maximal normal finite subgroup of $E$, then, in the quotient, (the image of) this element may show up in the maximal normal finite subgroup of a virtually cyclic group $\hat E '$, and we lose control over the action of the infinite order elements of $\hat E '$ on this maximal normal finite subgroup. As it was stated before, our approach to control the small cancellation parameters in the quotient comes from the `mildness' of the 2-torsion in our group, captured by the tameness of the actions under consideration.

\subsection{The Cone-Off Construction}
\label{subsection cone-off}

In this subsection, we introduce the \emph{cone-off} construction over certain families of subspaces of a metric space, and explain how to extend the action of a group on the metric space to an action on the cone-off. This construction will allow us to iteratively apply the Small Cancellation Theorem introduced in Section~\ref{subsection small cancellation}. For the remainder of this section, we fix the number $\rho >0$.

\begin{definition}
\label{definition cone over X}
    Let $X$ be a metric space. The \emph{cone over $X$ of radius $\rho$}, denoted by $Z_{\rho}(X)$ (or, if the value of $\rho$ is clear by context, simply $Z(X)$), is the topological quotient of $X \times [0,\rho]$ by the equivalence relation identifying all points of the form $(x,0)$ for $x\in X$.
\end{definition}

The equivalence class of $(x,0)$ is called the \emph{apex} of the cone. The cone over $X$ is endowed with a metric characterized as follows (see \cite[Chapter I.5, Proposition 5.9]{bridson_haefliger}). Let $x=(y,r)$ and $x'=(y',r')$ be two points of $Z(X)$, then \[ \cosh(d_{Z(X)}(x,x'))=\cosh(r)\cosh(r')-\sinh(r)\sinh(r')\text{cos}\left(\min \left( \{ \pi , \dfrac{d_{X}(y,y')}{\sinh(\rho)} \} \right)\right).\] 

In addition, if $X$ is a length space, then so is $Z(X)$. An action by isometries of a group $G$ over $X$ naturally extends to an action by isometries on $Z(X)$ as follows: for $x=(y,r)$ in $Z(X)$ and $g \in G$, put $g \cdot x=(g \cdot y,r)$. Note that in this case the apex of $Z(X)$ is a global fixed point.

We can compare the original metric space $X$ with its cone $Z(X)$ by defining a \emph{comparison map} $\psi: X \longrightarrow Z(X)$ such that $x \longmapsto (x, \rho)$.

Now we are ready to introduce the cone-off construction.

\begin{definition}
\label{definition cone off}
    Let $X$ be a hyperbolic length space, $\mathcal{Y}$ a collection of strongly quasi-convex subsets of $X$. For $Y \in \mathcal{Y}$, denote by $d_{Y}$ the metric on $Y$ induced by the length structure on $Y$ induced by the restriction of the length structure of $X$ to $Y$. Write $Z(Y)$ for the cone over $Y$ (endowed with the distance $d_{Y}$) of radius $\rho$ and $\psi_{Y}$ for the corresponding comparison map.

    The \emph{cone-off of radius $\rho$ over $X$ relative to $\mathcal{Y}$}, denoted by $\dot{X}_{\rho}(\mathcal{Y})$ (or simply $\dot{X}$ if $\rho$ and $\mathcal{Y}$ are clear by context) is the quotient of the disjoint union of $X$ and the $Z(Y)$ for all $Y \in \mathcal{Y}$ by the equivalence relation that, for all $Y \in \mathcal{Y}$ and $y \in Y$, identifies $y$ with $\psi_Y(y) \in Z(Y)$.
\end{definition}

Since the hyperbolic length space $X$ embeds into the cone-off $\dot X$, we will identify it with its image under this embedding. Notice, however, that this embedding will not be, in general, isometric (or even quasi-isometric). The cone-off is naturally endowed with a metric induced by the length structure on $\dot{X}$ induced by the length structures on $X$ and $Z(Y)$ for all $Y \in \mathcal{Y}$ (see \cite[Section 4.2]{coulon_1}).

 The following lemma gives conditions under which the cone-off is hyperbolic, with certain control over the hyperbolicity constant. For this purpose, we introduce a parameter that controls the overlap between the elements of $\mathcal{Y}$. We write \[ \Delta(\mathcal{Y})= \sup_{Y_{1} \neq Y_{2} \in \mathcal{Y}}\left(\text{diam}\left(Y_{1}^{+5\delta} \cap Y_{2}^{+5\delta}\right)\right). \]
 
Denote by $\boldsymbol{\delta}$ the hyperbolicity constant of the hyperbolic plane.

\begin{lemma}\emph{\cite[Proposition 6.4]{coulon_2}}
\label{cone off hyperb}
    There exist positive numbers $\delta_{0}$, $\Delta_{0}$ and $\rho_{0}$ that satisfy the following property. Let $X$ be a $\delta$-hyperbolic length space with $\delta \leq \delta_{0}$. Let $\mathcal{Y}$ be a family of strongly quasi-convex subsets of $X$ with $\Delta(\mathcal{Y}) \leq \Delta_{0}$. Let $\rho \geq \rho_{0}$. Then, the cone-off $\dot{X}_{\rho}(\mathcal{Y})$ is $\dot{\delta}$-hyperbolic, with $\dot{\delta}=900 \boldsymbol{\delta}$.
\end{lemma}

For the remainder of this section, we fix a length space $X$ and a family $\mathcal{Y}$ as in Definition~\ref{definition cone off}. 

Lemmas~\ref{dist on x compared cone off} and~\ref{dist on Y compared cone off} provide some insight on how the metric on $\dot{X}$ relates to the metric on $X$ and to the cones $Z(Y)$. In order to state the first of these results, we introduce the map $\mu: \mathbb{R}_{\geq 0} \longrightarrow \mathbb{R}_{\geq 0}$ characterized by \[ \cosh(\mu(t))=\cosh^{2}(\rho)-\sinh^{2}(\rho)\text{cos}\left(\min\left( \{ \pi, \dfrac{t}{\sinh(\rho)}\} \right)\right) \] for all $t \geq 0$. The map $\mu$ has the following properties that will be used later.

\begin{lemma} \emph{\cite[Proposition 4.2]{coulon_1}}
\label{bound on mu for small t}
    The map $\mu$ is continuous, concave (down) and non-decreasing. Furthermore, the following properties hold.
    \begin{itemize}
        \item For all $t \geq 0$, we have: $t-\dfrac{1}{24}\left(1+\dfrac{1}{\sinh^{2}(\rho)}\right)t^{3} \leq \mu(t) \leq t$, and
        \item for all $t \in [0, \pi \sinh(\rho)]$, we have: $t \leq \pi \sinh\left(\dfrac{\mu(t)}{2}\right)$.
    \end{itemize}
\end{lemma}

\begin{lemma}\cite[Lemma 5.8]{coulon_2}
\label{dist on x compared cone off}
    For every $x,x' \in X$ we have: \[ \mu(d_{X}(x,x')) \leq d_{\dot{X}}(x,x') \leq d_{X}(x,x'). \]
\end{lemma}

\begin{lemma}\emph{\cite[Lemma 5.7]{coulon_2}}
\label{dist on Y compared cone off}
    Let $v$ be the apex of a cone $Z(Y)$ for some $Y \in \mathcal{Y}$. Then, $B_{\dot{X}}(v,\rho)=Z(Y) \backslash Y$.
\end{lemma}

\subsubsection{Group action on the cone-off} For the remainder of this section, we assume $\rho \geq \max  (\{ \rho_{0}, 10^{10}\boldsymbol{\delta},10^{20}L_{S}\boldsymbol{\delta}\})$, and we fix a real number $\delta \leq \delta_{0}$, a $\delta$-hyperbolic length space $X$, a family $\mathcal{Y}$ of strongly quasi-convex subsets of $X$ with $\Delta(\mathcal{Y}) \leq \Delta_{0}$, where $\rho_{0}$, $\delta_{0}$ and $\Delta_{0}$ are the parameters provided by Lemma~\ref{cone off hyperb}.


Consider a group $G$ acting by isometries on $X$ and acting on the family $\mathcal{Y}$ by left translation, that is, such that $g \cdot Y \in \mathcal{Y}$ for all $Y \in \mathcal{Y}$. We can extend this action by homogeneity to an action of $G$ on the cone-off as follows. Let $Y \in \mathcal Y$ and $x = (y,r)$ be a point on the cone $Z(Y)$. For $g\in G$ we define $g\cdot x = g\cdot (y,r) = (g \cdot y,r)\in Z(g \cdot Y)$. It follows from the definition of the metric of $\dot X$ that this action is an isometry on $\dot X$.

The next result uses techniques from \cite[Proposition 4.10]{coulon_1} for WPD actions on hyperbolic length spaces, and from \cite[Proposition 5.40]{dahmani_guirardel_osin} for acylindrical actions on hyperbolic geodesic spaces.

\begin{lemma}
\label{acyl passes to cone-off}
    If the action of $G$ on $X$ is acylindrical, then so is the induced action on $\dot{X}$.
\end{lemma}

\begin{proof}
    We will apply Remark~\ref{acylind characterization}. The action of $G$ on $X$ is acylindrical, therefore, there are positive numbers $L'$ and $M'$ such that, for all $x,x' \in X$, if $d_{X}(x,x')$ is at least $L'$, then there are at most $M'$ elements moving $x$ and $x'$ less than $\pi \sinh(300\dot{\delta})$. We will show that we can take $M'$ and $L'+4 \rho$ as the parameters $M$ and $L$ of Remark~\ref{acylind characterization}.

    Now, let $a,b \in \dot X$ be such that $d(a,b) \geq L' + 4 \rho$, and consider a $(1, \dot \delta)$-quasi-geodesic segment $[a,b]_{\dot \delta}$. If an element $g \in G$ moves both $a$ and $b$ by less than $100 \dot \delta$, we can apply Lemma~\ref{distances within quadrangles} to conclude that any point in this quasi-geodesic segment is moved by less than $600 \dot \delta$ by $g$. Furthermore, since by Lemma \ref{dist on Y compared cone off} the diameter of the ball around an apex of the cone is at most $2\rho$, this quasi-geodesic must contain points of $X$ at a distance of at least $L'$. Therefore, we may assume that $a,b \in X$ and that we need to bound the number of elements of $G$ moving $a$ and $b$ by at most $600 \dot \delta$. By Lemma~\ref{dist on x compared cone off}, we have that \[\mu(d_{X}(a,g \cdot a)) \leq d_{\dot X}(a,g \cdot a) \leq 600 \dot \delta.\] By the choice of $\rho$, we have that $\mu(d_{X}(a,g \cdot a)) < \pi \sinh(\rho)$, and therefore Lemma~\ref{bound on mu for small t} gives \[ d_{X}(a, g \cdot a) \leq \pi \sinh (300 \dot \delta). \] Similarly, we obtain that \[ d_{X}(b, g \cdot b) \leq \pi \sinh (300 \dot \delta), \] so the number of elements $g$ satisfying this property is, indeed, at most $M'$.
\end{proof}

\subsection{The small cancellation theorem}
\label{subsection small cancellation}

In this subsection, we will state a small cancellation theorem, following closely the expositions in \cite{coulon_1} and \cite{amelio_andre_tent} for a WPD action of a group without 2-torsion, and in \cite{coulon_4} for the more general setting of a gentle action. Throughout this section, fix a group $G$ with a non-elementary acylindrical action on a $\delta$-hyperbolic length space $X$, and fix a parameter $\rho \in \mathbb{R}$ (to be thought of as a very large distance). Consider a family $\calQ$ of pairs $(H,Y)$, where $H$ is a subgroup of $G$ and $Y$ an $H$-invariant strongly quasi-convex subset of $X$ with the following properties:

\begin{itemize}
    \item there exists an odd integer $n' \geq 100$ such that for all subgroups $H$ there is a primitive loxodromic element $h' \in G$ such that $H= \langle h'^{n} \rangle$ for some odd $n \geq n'$;
    \item $Y$ is the cylinder $Y_h$ of $h=h'^{n}$; and
    \item the group $G$ acts on the family $\calQ$ via $g \cdot (H,Y) = (gHg^{-1}, g \cdot Y)$ for $g \in G$ and $(H,Y) \in \calQ$.
\end{itemize}

Let $K$ be the normal subgroup generated by the family $Q=\{H \, : \, (H,Y) \in \calQ\}$. The aim is to understand the quotient $\hat{G}=G/K$, and for that purpose we will define a metric space $\xhat$ on which $\ghat$ acts. We will do that in two steps.

First, notice that, since the family $\mathcal{Y}= \{Y \, : \, (H,Y) \in \calQ\}$ is composed of strongly quasi-convex subsets of $X$ (see Lemma~\ref{axis vs cylinder}), we can construct the cone-off $\dot{X}$ of radius $\rho$ of $X$ relative to this family (as defined in Subsection \ref{subsection cone-off}). The group $G$ has a natural action by isometries on this space induced by the action of $G$ on $X$.

We now set the space $\xhat = \dot{X} / K$. This will be a metric space on which $\ghat$ naturally acts by isometries (see \cite[Section 5.1]{coulon_1}). We write $\zeta: \dot X \longrightarrow \hat{X}$ for the projection map and $v(\mathcal{Q})$ for the subset of $\dot X$ consisting of the apices of the cones $Z(Y)$ for $(H,Y) \in \mathcal{Q}$. Let $\hat{v}(\mathcal{Q})$ denote its image in $\hat{X}$. We will also call the elements of $\hat{v}(\mathcal{Q})$ apices. For an element $g \in G$ (respectively, $x \in \dot X$), we will write $\hat{g}$ (respectively, $\hat{x}$) for its image in $\hat{G}$ (respectively, $\hat{X}$).

In order to get some desired properties of the group $\hat G$, the space $\hat X$, and of the action of $\hat{G}$ on $\hat{X}$, we need the action of $G$ on $X$ to satisfy some small cancellation conditions. These conditions involve two parameters $\Delta(\calQ)$ and $T(\calQ)$ associated to this family $\calQ$, which will play the role of the length of the largest piece and the length of the shortest relator in the usual small cancellation theory, defined as: \[ \Delta(\mathcal{Q})= \sup(\{ \text{diam}(Y_{1}^{+5\delta} \cap Y_{2}^{+5\delta}) \} ) \, : \, (H_{1},Y_{1}) \neq (H_{2},Y_{2}) \in \mathcal{Q}\}) \] and \[ T(\mathcal{Q})= \inf (\{ [h] \, : \, h \in H \, , \, (H,Y) \in \mathcal{Q} \}). \]

The following statement appears as Theorem 4.17 in \cite{coulon_4}, and is a combination of a number of results in \cite{coulon_2}. 

\begin{theorem}\emph{The Small Cancellation Theorem.}
\label{small cancellation theorem}
    There exist positive constants $\rho_{0}$, $\delta_{0}$, $\delta_1$ and $\Delta_{0}$ that are independent of $X$, $G$ and $\mathcal{Q}$ such that for $\delta \leq \delta_{0}$, $\rho \geq \rho_{0}$, $\Delta(\mathcal{Q}) \leq \Delta_{0}$ and $T(\mathcal{Q}) \geq 10\pi \sinh(\rho)$ the following statements hold.

    \begin{enumerate}[label=(\arabic*)]
        \item \label{cone-off is hyperbolic} The cone-off $\dot{X}$ is $\dot{\delta}$ hyperbolic with $\dot{\delta} \leq \delta_1$.
        \item \label{x hat is hyperbolic} The quotient space $\xhat$ is $\hat{\delta}$-hyperbolic with $\hat{\delta} \leq \delta_1$.
        \item \label{apex stabilizers isom to quotient of lox subgroup} Let $(H,Y) \in \calQ$ and let $\hat{v}$ be the image in $\xhat$ of the apex $v$ of $Z(Y)$. The projection $G \twoheadrightarrow \ghat$ induces an isomorphism from $\text{Stab}(Y)/H$ onto the image of $\text{Stab}(Y)$, which coincides with $\text{Stab}(\hat{v})$.
        \item \label{isom sc theorem small ball around apex} Let $(H,Y) \in \calQ$ and let $\hat{v}$ be the image in $\xhat$ of the apex $v$ of $Z(Y)$. The projection map $\zeta : \dot X \longrightarrow \hat X$ induces an isometry from $B(v, \rho /2)/H$ onto $B_{\hat X}(\hat v , \rho /2)$.
        \item \label{isom sc theorem ball far away from vertices} For every number $r \in (0, \rho /20]$ and every $x \in \dot X$, if there is no $v \in v(\mathcal Q)$ such that $d_{\dot X}(x,v) <2r$, then the projection $\zeta : \dot X \longrightarrow \hat X$ induces an isometry from $B_{\dot X}(x,r)$ onto $B_{\hat X}(\hat x , r)$.
        \item \label{kernel acting freely on non-vertices} Let $g \in K \backslash \{ 1 \}$, let $x \in \dot X$ and let $r=d_{\dot X}(x, v(\mathcal Q))$. Then, $d_{\dot X}(x, g \cdot x) \geq \min \{ 2r , \rho /5 \}$. In particular, $K$ acts freely on $\dot X \backslash v(\mathcal Q)$.
    \end{enumerate}
\end{theorem}

Notice that the constants $\delta _0$ and $\Delta _0$ can be chosen arbitrarily small, while the constant $\rho_0$ can be chosen arbitrarily large. Throughout this article, we will need to ensure that many inequalities involving these parameters are satisfied, and to this purpose we will pick the values for these constants very generously. Following \cite{coulon_4}, we assume $\rho _0 > 10^{20} L_S \delta_1$ and $\delta _0 , \Delta_0 < 10^{-10}\delta_1$. These choices are such that \[ \max \{ \delta_0 , \Delta_0 \} \ll \delta _1 \ll \rho_0 \ll \pi \sinh(\rho_0). \] For the remainder of this section, we assume that $X$, $G$ and $\mathcal{Q}$ satisfy the assumptions of Theorem \ref{small cancellation theorem} (in addition to the assumptions introduced in the preceding subsection), and we will write $\hat G$ and $\hat X$ for the quotient group and the quotient space (respectively) as described above. Notice that, up to increasing the constants $\dot \delta$ and $\hat \delta$, we may take $\dot \delta = \hat \delta = \delta _1$. We will adopt this point of view, but we will keep the distinct notation so as to emphasize the space under consideration.

\begin{remark}
\label{H is normal in Stab Y}
    Notice that the value of the parameter $\Delta (\calQ)$ is the value of the parameter $\Delta (\mathcal Y)$ introduced in Subsection \ref{subsection cone-off} for the family $\mathcal{Y}= \{ Y \, : \, (H,Y) \in \calQ \}$ with one caveat: we are now considering distinct \emph{pairs} $(H,Y)$ and $(H' , Y')$, and in principle the same subset $Y$ may appear in two distinct pairs. However, our assumptions imply that this cannot happen. If $(H,Y) $ and $ (H',Y') $ are two distinct pairs in $ \mathcal{Q}$, we must have $Y\neq Y'$: $\Delta(\mathcal{Q})$ is finite and the cylinder of a loxodromic element is unbounded. In particular, this gives that, for $(H,Y)\in \mathcal{Q}$, the subgroup $H$ is normal in $\text{\text{Stab}}(Y)$, and point \ref{apex stabilizers isom to quotient of lox subgroup} of Lemma~\ref{small cancellation theorem} actually makes sense.
\end{remark}

\begin{remark}
\label{apices at distance 2rho}
    Notice that, by construction, the distance between any two distinct apices $\hat v_1$ and $\hat v _2$ of $\apices$ is at least $2 \rho$.
\end{remark}

\subsection{Apex stabilizers in the quotient space}
\label{subsection apex satbilizers of quotient}

In this subsection, we introduce some terminology and study some basic properties of the subgroups of the quotient group $\hat G$ fixing some apex $\hat v \in \apices$.

Recall that, since the action of $G$ on $X$ is acylindrical, every loxodromic subgroup of $G$ is virtually cyclic (Lemma \ref{lox subg in wpd is virt cyc}). Thus, if $E$ is a loxodromic subgroup of $G$, it has a maximal normal finite subgroup $F$ such that either $E \cong F \rtimes \mathbb Z$ or $E/F \cong D_{\infty}$.

Now, by construction, the pairs $(H,Y) \in \mathcal Q$ consist of a cyclic group $H=\langle h^n \rangle$ for a primitive loxodromic element $h$ and the corresponding cylinder $Y=Y_{h^n}$. For notational clarity, we may write $n_h$ for the integer such that $H=\langle h^n \rangle$ for one such pair in $\calQ$. We have that $\stab (Y)$ is the maximal loxodromic subgroup containing $H$. Therefore, by Theorem \ref{small cancellation theorem} \ref{apex stabilizers isom to quotient of lox subgroup}, we get the following classification result for apex stabilizers.

\begin{lemma}
\label{apex stab in the quotient}
    Let $\hat v$ be an apex in $\apices$. Let $(H,Y) \in \calQ$ be such that $\stab(Y)$ is a preimage of $\stab (\hat v)$ and let $F$ be the maximal normal finite subgroup of $\stab (Y)$. Then, the projection map $G \twoheadrightarrow \hat G$ induces an isomorphism from $F$ onto its image $\hat F$, and we have that $\stab (\hat v)/ \hat F$ is isomorphic to:
    \begin{enumerate}
        \item $C_{n_h}$ (if and only if $\stab (Y)/F \cong \mathbb Z$), or
        \item $D_{n_h}$ (if and only if $\stab (Y)/F \cong D_{\infty}$).
    \end{enumerate}
\end{lemma}

In virtue of the classification provided by Lemma \ref{apex stab in the quotient}, we introduce the following terminology, following \cite{coulon_4}.

\begin{definition}
\label{isometries of apex stab geometrically}
    Let $\hat v \in \apices$, let $(H,Y) \in \calQ$ be such that $\hat v$ is the image of the apex corresponding to the cone $Z(Y)$. Let $F$ be the maximal normal finite subgroup of $\stab (Y)$ and $\hat F$ its image in $\hat G$. Let $\hat g \in \stab (\hat v)$. We say that $\hat g$ is:
    \begin{itemize}
        \item \emph{locally trivial} at $\hat v$ if $\hat g \in \hat F$.
        \item a \emph{reflection} at $\hat v$ if its image under the quotient map $\stab (\hat v) \twoheadrightarrow \stab (\hat v) / \hat F$ is a reflection of $D_n$.
        \item a \emph{strict rotation} otherwise.
    \end{itemize}
    Similarly, for a subset $\hat S \subseteq \stab (\hat v)$, we will say that $\hat S$ is:
    \begin{itemize}
        \item \emph{locally trivial} at $\hat v$ if every element of $\hat S$ is locally trivial at $\hat v$.
        \item a \emph{reflection group} (respectively, a \emph{strict reflection group}) at $\hat v$ if it is a subgroup and its image under the quotient map $\stab (\hat v) \twoheadrightarrow \stab (\hat v) / \hat F$ is contained in a subgroup generated by a reflection of $D_n$ (respectively, and it is not locally trivial).
    \end{itemize}
\end{definition}

\begin{remark}
    Let us now expand the explanation (already hinted at the beginning of this section) of the main differences between the small cancellation setting for this article and the one in \cite{coulon_4}. The aim of that article is to construct periodic groups of (large enough) even exponent, where arbitrarily large powers of 2 can divide the exponent. For that purpose, the integers $n_h$ defined above cannot be assumed to be odd (as is our case).

    This nuance has very important consequences for the algebraic structure of the groups being constructed. The periodic quotients will be obtained by iterating the application of (a variant of) Theorem \ref{small cancellation theorem}, and then taking the limit of this construction. Thus, for example, if the initializing group $G$ is torsion-free and the exponent $n_h$ is odd and has the same value $n$ for every pair, then every finite subgroup of the limit quotient group (and of every step of the induction process) will be cyclic of order dividing $n$. If $n_h$ is taken to be the same integer $n$ for every pair, but $n$ is even, the limit quotient group may have arbitrarily long chains of subgroups $D_n \times D_{m} \times \dots \times D_m$, where $m$ is the largest power of 2 dividing $n$ (see, for example, \cite{ivanov_2} and \cite{lysenok}).

    In order to deal with the complicated algebraic structure of finite subgroups, in \cite{coulon_4} the author makes an extra assumption: that whenever for some $\hat v \in \apices$ the subgroup $\stab (\hat v)/\hat F$ has even torsion, then $\stab (\hat v)$ contains a central half-turn (an involution that is a strict rotation at $\hat v$ and that is central in $\stab (\hat v)$). One such element cannot exist if $n$ is odd.
\end{remark}

In view of the previous remark, some of the structural results from \cite{coulon_4} do not directly adapt to the setting of this article. Throughout the rest of the present section, we will retrieve results from \cite{coulon_1} that are needed later in this article, providing no proof here whenever the proof given in that article works without any further change in our setting. We will also prove different versions of the results from \cite{coulon_4} whenever necessary, as well as new results specific to our setting.

We finish this subsection with a result on the structure of (almost) fixed-point sets of elements of $\stab (\hat v)$.

\begin{lemma}\emph{(see \cite[Proposition 4.13]{coulon_4})}
\label{balls around apices and stabilizers sc theorem}
    Let $\hat v \in \apices$ be an apex of $\hat X$.
    \begin{enumerate}[label=(\arabic*)]
        \item \label{stab of apex locally trivial elem fixes all ball} If $\hat g \in \stab (\hat v)$ is locally trivial at $\hat v$, then $B_{\hat X}(\hat v , \rho)$ is contained in $\text{Fix}(\hat g , \hat \delta)$.
        \item \label{stab of apex reflection group fixes point far away} If $\hat A$ is a reflection group at $\hat v$, there is a point $\hat x \in \text{Fix}(\hat A , \hat \delta)$ such that $d_{\hat X}(\hat x , \hat v)>\rho /2$.
        \item \label{stab of apex strict rot only fixes apex} If $\hat g$ is a strict rotation at $\hat v$, then there is $k \in \mathbb Z$ such that for every $\hat x \in B_{\hat X}(\hat v , \rho /3)$ we have that $d_{\hat X}(\hat g ^k \cdot \hat x , \hat x) \geq 2 d_{\hat X}(\hat x , \hat v)- \hat \delta$. In particular, for every $r \in [\hat \delta , \rho /10]$ the set $\text{Fix}(\hat g ^k , r)$ is non-empty and contained in $B_{\hat X}(\hat v , r)$. In consequence, $\hat v$ is the unique apex of $\apices$ fixed by $\hat g$.
    \end{enumerate}
\end{lemma}

\begin{proof}
    The proof of point (i) of Proposition 4.13 in \cite{coulon_4} adapts to this setting without any further change to prove part \ref{stab of apex locally trivial elem fixes all ball}. Meanwhile, the proof of point (iii) of Proposition 4.13 in \cite{coulon_4} yields unchanged a proof of part \ref{stab of apex strict rot only fixes apex} (even though this result is not explicitly stated in the claim of the aforementioned result).

    For part \ref{stab of apex reflection group fixes point far away}: let $(H,Y) \in \mathcal{Q}$ be such that $\hat v$ is the image of the apex of the cone $Z(Y)$. The subgroup $\hat A$ is the image of an elliptic subgroup $A$ of $H$. Thus, since $Y$ is an $A$-invariant and strongly quasi-convex subset of $X$, by Lemma \ref{charac subset of elliptic is quasi convex} and the triangle inequality, $Y$ contains a point $x \in \text{Fix}(A, 100 \delta)$. Then, since the quotient map $\zeta : X \longrightarrow \hat X$ shortens the distances, $\hat A$ moves the image $\hat x$ of $x$ in $\hat X$ by less than $\hat \delta$, and by Theorem \ref{small cancellation theorem} \ref{isom sc theorem small ball around apex} this point is at distance greater than $\rho /2$.
\end{proof}

\begin{remark}
    Lemma \ref{balls around apices and stabilizers sc theorem} has the following consequence: if an element is a strict rotation at an apex $\hat v$, then it cannot stabilize any other apex. In consequence, we may say that an element is a strict rotation without reference to any specific apex.
\end{remark}

\subsection{Lifting properties}
\label{subsection lifting properties}

The purpose of this subsection is studying how certain `pictures' in the quotient space $\hat X$ can be lifted to the cone-off space $\dot X$. We begin with three results whose proofs can be directly adapted from the corresponding ones in \cite{coulon_4}.

For the next result, Lemma \ref{isom proj from dot X of a set far away from apices}, the proof of \cite[Lemma 4.17]{coulon_4} works \emph{verbatim}.

\begin{lemma}
\label{isom proj from dot X of a set far away from apices}
    Let $Z$ be a subset of $\dot X$ such that for every pair of points $z,z'$ of $Z$ and every apex $v \in v(\mathcal{Q})$ we have that $\langle z,z' \rangle_v > 13 \dot \delta$. Then, the map $\zeta: \dot X \longrightarrow \hat X$ induces an isometry from $Z$ onto its image $\hat Z$. Furthermore, the following properties hold.
    \begin{enumerate}
        \item Let $\hat g \in \hat G$ be such that there exist points $z_1,z_2$ in $Z$ such that their images $\hat z_1$ and $\hat z_2$ in $\hat Z$ satisfy $\hat g \cdot \hat z_1 = \hat z_2$. Then, there exists a unique preimage $g$ of $\hat G$ in $G$ such that $g \cdot z_1 = z_2$. Moreover, this element $g \in G$ is such that for every pair of points $z,z'$ in $Z$ we have that $g \cdot z =z'$ if and only if their images $\hat z , \hat z'$ in $\hat Z$ satisfy $\hat g \cdot \hat z = \hat z '$.
        \item The projection map $G \twoheadrightarrow \hat G$ induces an isomorphism from $\stab (Z)$ onto its image $\stab (\hat Z)$.
    \end{enumerate}
\end{lemma}

Lemma \ref{lift of a subset Xhat far from apices} appears in \cite{coulon_4} as Lemma 4.18, and once again the proof of that result works without any further change in our setting.

\begin{lemma}
\label{lift of a subset Xhat far from apices}
    Let $\hat Z$ be a subset of $\hat X$ such that for every pair of points $\hat z_1 , \hat z_2$ of $\hat Z$ and every apex $\hat v \in \apices$ we have that $\langle \hat z_1 , \hat z_2 \rangle_{\hat v} > 13 \hat \delta$. Let $\hat z$ be a point of $\hat Z$ and $z$ a preimage of $\hat z$ in $\dot X$. Then, there exists a unique subset $Z$ of $\dot X$ containing $z$ and such that the projection map $\zeta : \dot X \longrightarrow \hat X$ induces an isometry from $Z$ onto $\hat Z$. In particular, for every $z_1 , z_2$ in $Z$ and $v \in v(\mathcal{Q})$, if we denote by $\hat z_1$, $\hat z_2$ and $\hat v$ their respective images in $\hat X$ we have that $\langle z_1 , z_2 \rangle_v \geq \langle \hat z_1 , \hat z_2 \rangle_{\hat v}$.
\end{lemma}

\begin{remark}
\label{remark on lifting figures far from apices}
We collect some immediate consequences and observations from Lemmas \ref{isom proj from dot X of a set far away from apices} and \ref{lift of a subset Xhat far from apices}.
    \begin{enumerate}[label=(\arabic*)]
        \item \label{proj of far away for quasi-convex} Lemma \ref{isom proj from dot X of a set far away from apices} applies in particular to a subset $Z$ of $\dot X$ that is $\alpha$-quasi-convex and such that, for every $v \in v(\mathcal Q)$, $d_{\dot X}(v,Z)> \alpha + 13 \dot \delta$.
        \item \label{lift of far away for quasi-convex} Lemma \ref{lift of a subset Xhat far from apices} applies in particular to a subset $\hat Z$ of $\hat X$ that is $\alpha$-quasi-convex and such that, for every $\hat v \in \apices$, $d_{\hat X}(\hat v,\hat Z)> \alpha + 13 \hat \delta$.
        \item Let $\hat Z$ be a subset of $\hat X$ satisfying the hypotheses of Lemma \ref{lift of a subset Xhat far from apices}. Then, we can apply Lemma \ref{isom proj from dot X of a set far away from apices} to any lift $Z$ of $\hat Z$ in $\dot X$ (as provided by the aforementioned lemma). In particular, the quotient map $G \twoheadrightarrow \hat G$ induces an isomorphism from $\stab (Z)$ onto its image, which coincides with $\stab (\hat Z)$.
    \end{enumerate}
\end{remark}

Lemmas \ref{isom proj from dot X of a set far away from apices} and \ref{lift of a subset Xhat far from apices} allow us to project and to lift figures that stay far away from the apices of $\dot X$ and $\hat X$ respectively. Lemmas \ref{lift sc theorem close to apices locally trivial} and \ref{lift sc theorem close to apices reflection group} deal with the case where we have quasi-geodesics that come close to some apex.

Lemma \ref{lift sc theorem close to apices locally trivial} appears in \cite{coulon_4} as Proposition 4.19, and the proof of the result in that article works without any further change in our setting.

\begin{lemma}
\label{lift sc theorem close to apices locally trivial}
    Let $x$ and $y$ be two points of $X$, let $\gamma : [a,b] \longrightarrow \dot X$ be a path from $x$ to $y$ such that its image $\hat \gamma : [a,b] \longrightarrow \hat X$ on $\hat X$ is a $(1, \hat \delta)$-quasi-geodesic from the image $\hat x$ of $x$ to the image $\hat y$ of $y$. Let $S$ be a subset of $G$ and denote by $\hat S$ its image on $\hat G$. Assume that for every $g \in S$ we have that $d_{\dot X}(x, g \cdot x) \leq \rho /100$ and $d_{\hat X}(\hat y , \hat g \cdot \hat y) \leq \rho /100$. In addition, we assume that for every apex $\hat v \in \apices$ such that $\langle \hat x , \hat y \rangle_{\hat v} \leq \rho /4$, the set $\hat S \cap \stab (\hat v)$ is locally trivial at $\hat v$. Then, we have that $d_{\dot X}(y, g \cdot y)=d_{\hat X}(\hat y , \hat g \cdot \hat y)$ for every $g \in S$.
\end{lemma}

\begin{lemma}
\label{lift sc theorem close to apices reflection group}
    Let $x$ and $y$ be two points of $X$ and let $S$ be a subset of $G$. Write $\hat S$ for the image of $S$ in $\hat G$. We assume that $d_{\dot X}(x , g \cdot x) \leq \rho /100$ and $d_{\hat X}(\hat y , \hat g \cdot \hat y) \leq \rho /100$ for every $g \in S$. We further assume that for every apex $\hat v \in \apices$ the set $\hat S \cap \stab (\hat v)$ is contained in a reflection group at $\hat v$. Then, one of the following holds.
    \begin{itemize}
        \item The set $\hat S$ lies in a strict reflection group at some apex $\hat v \in \apices$.
        \item There exists $u \in K$ such that $d_{\dot X}(u \cdot y , gu \cdot y)=d_{\hat X}(\hat y , \hat g \cdot \hat y)$ for every $g \in S$.
    \end{itemize}
\end{lemma}

\begin{proof}
    We first assume that for every apex $\hat v \in \apices$ such that $\langle \hat x , \hat y \rangle _{\hat v} \leq \rho /4$, the set $\hat S \cap \stab (\hat v)$ is locally trivial at $\hat v$. Let $\varepsilon < \hat \delta/2$ be a positive number, and let $u \in K$ be such that $d_{\dot X}(x , u \cdot y) \leq d_{\hat X}(\hat x , \hat y)+\varepsilon$ (this element exists by the definition of the distance in $\hat X$). Take a $(1, \varepsilon)$-quasi-geodesic $\gamma$ from $x$ to $u \cdot y$. The image $\hat \gamma$ of $\gamma$ in $\hat X$ is, by construction, a $(1, 2 \varepsilon)$-quasi-geodesic from $\hat x$ to $\hat y$. In consequence, Lemma \ref{lift sc theorem close to apices locally trivial} applies (with the element $u \cdot y$ in place of $y$) and we are in the second case of the claim of this lemma.

    Now, assume that there is some apex $\hat v \in \apices$ such that $\langle \hat x , \hat y \rangle _{\hat v} \leq \rho /4$ and such that $\hat S$ is not locally trivial at $\hat v$. Every element $\hat g \in \hat S$ moves both $\hat x$ and $\hat y$ by at most $\rho /100$, so by Lemma \ref{lemma tr distance gromov product} we have that $\hat g$ moves $\hat v$ by less than $\rho$. Thus, since the distance between any two apices is bounded from below by $\rho$, we get that all of $\hat S$ must fix $\hat v$. By assumption, $\hat S \cap \stab (\hat v)$ is contained in a reflection group at $\hat v$, and also by the assumption of this paragraph, this must be a strict reflection group. Therefore, we are in the first case of the claim of this lemma.
\end{proof}

\subsection{The action of the quotient group}
\label{subsection action ghat on xhat}

In this subsection, we begin a systematic study of the properties of the action of $\hat G$ on $\hat X$. More concretely, we will prove that this action is non-elementary and (with one mild extra assumption) acylindrical. 

The next result is \cite[Lemma 4.21]{amelio_andre_tent}. In fact, the proof is identical to that of the aforementioned result, with the exception that some of the constants involved have been modified. For the sake of completeness, we include a proof here.

\begin{lemma}
\label{acylindricity of xhat}
    Assume that there is some positive integer $m$ such that for every pair $(H,Y) \in \mathcal{Q}$ we have that $\lvert \stab (Y)/H \rvert \leq m$. Then, the action of $\hat G$ on $\hat X$ is acylindrical.
\end{lemma}

\begin{proof}
    We will apply Remark \ref{acylind characterization}. Since the action of $G$ on the cone-off space is acylindrical, there are positive integers $L'$ and $M'$ such that for every pair of points $x,x'$ of $\dot X$, if $d_{\dot X}(x,x') \geq L'$, then the number of elements of $G$ moving both $x$ and $x'$ at most $100 \hat \delta$ is at most $M'$. We will prove that we can take $M= \max \{ M' , m \}$ and $L=L'$ to satisfy the hypotheses of Remark \ref{acylind characterization}.

    Let now $\hat x$ and $\hat x'$ be two points of $\hat X$ at a distance of at least $L'$. Let $\hat Z$ be the hull of $\{ \hat x , \hat x '\}$. By Lemma \ref{hull is quasi-convex} this is a $6 \hat \delta$-quasi-convex subset of $\hat X$. Moreover, if an element $\hat g$ moves both $\hat x$ and $\hat x'$ by at most $100 \hat \delta$, then by Lemma \ref{distances within quadrangles} it moves every element of $\hat Z$ by at most $600 \hat \delta$.

    Suppose now that there is an apex $\hat v \in \apices$ at distance at most $\rho /3$ of $\hat Z$. Then, an element $\hat g$ as in the preceding paragraph will move $\hat v$ by at most $2 \rho /3 + 600 \hat \delta < \rho$. Thus, since the distance between any two distinct apices is at least $\rho$, such a $\hat g$ fixes $\hat v$, and by assumption there are at most $m$ such elements.

    Assume now that there is no apex $\hat v \in \apices$ at distance at most $\rho /3$ of $\hat Z$. Lemma \ref{neighbourhood of a convex subset} gives that the closed $600 \hat \delta$-neighbourhood of $\hat Z$ (denoted by $\hat Z ^{+600 \hat \delta}$) is $2 \hat \delta$-quasi-convex. Notice that there is no apex $\hat v \in \apices$ at distance at most $\rho /4$ of $\hat Z ^{+600 \hat \delta}$. Now, by construction, $\hat g \cdot \hat X$ lies $\hat Z ^{+600 \hat \delta}$. Thus, we can apply Lemmas \ref{isom proj from dot X of a set far away from apices} and \ref{lift of a subset Xhat far from apices} to get that there is a subset $Z^+$ of $\dot X$ such that the projection map $\zeta : \dot X \longrightarrow \hat X$ induces an isometry from $Z^+$ onto $\hat Z ^{+600 \hat \delta}$. In particular, the preimages $x$ and $x'$ of $\hat x$ and $\hat x '$ are at distance at least $L'$. Moreover, we get a preimage $g$ of $\hat g$ such that it moves $x$ and $x'$ by at most $100 \hat \delta$. By the choice of $L'$, there are at most $M$ distinct elements of $G$ with this property, from where the desired conclusion follows.
\end{proof}

For the remainder of this section, we will assume that the family $\mathcal{Q}$ satisfies the assumption of Lemma \ref{acylindricity of xhat}, so that the action of $\hat G$ on $\hat X$ is acylindrical.

The following result follows from Lemma 4.23 and Proposition 4.24 in \cite{coulon_4} (both of which adapt without any further change to our setting) together with our choice of $n'$ at the beginning of Subsection \ref{subsection small cancellation}.

\begin{lemma}
\label{action of ghat non-elementary}
    The action of $\hat G$ on $\hat X$ is non-elementary.
\end{lemma}

\subsection{Elementary subgroups of \texorpdfstring{$\hat G$}{G-hat}}
\label{subsection elementary subgroups of ghat sc theorem}

In this subsection, we will study the elementary subgroups of $\hat G$ for their action on $\hat X$, with some results on their algebraic structure and lifting properties.

Notice that, since the map $X \longrightarrow \hat X$ shortens the distances between the images of points of $X$, then the projection $G \longrightarrow \hat G$ is such that the image $\hat E$ of an elementary subgroup $E$ of $G$ is elementary. However, it may happen that the case in the classification introduced in Subsection \ref{subsubsection isom of hyp spaces} does not coincide for an elementary subgroup and its image: for example, if $(H,Y) \in \mathcal{Q}$, then the image of the loxodromic subgroup $\stab (Y)$ fixes an apex, and thus, it is elliptic.

To the purpose of understanding which elementary subgroups of $\hat G$ come from elementary subgroups of $G$ of the same nature, we introduce, following the terminology in \cite{coulon_4}, the notion of a \emph{lift} of an elementary subgroup of $\hat G$.

\begin{definition}
\label{definition lift of a subgroup}
    Let $\hat E$ be an elliptic (respectively, loxodromic) subgroup of $\hat G$ (by the action on $\hat X$). We say that $\hat E$ \emph{lifts} if there is an elliptic (respectively, loxodromic) subgroup $E$ of $G$ (by its action on $X$) such that the quotient map $G \twoheadrightarrow \hat G$ induces an isomorphism from $E$ onto $\hat E$. We call one such subgroup $E$ a \emph{lift} of $\hat E$.
\end{definition}

The next lemma appears in \cite{coulon_4} as Lemma 4.26, and the proof exhibited in that article works without any changes in our setting.

\begin{lemma}
\label{sc theorem induces iso of elliptic}
    Let $S$ be a subset of $G$ such that $\text{Fix}(S, \rho /10)$ is non-empty. Then, the quotient map $G \twoheadrightarrow \hat G$ is one-to-one when restricted to $S$. In particular, if $E$ is an elliptic subgroup of $G$, then the quotient map induces an isomorphism from $E$ onto its image.
\end{lemma}

The next lemma is a first step towards understanding which elliptic subgroups of $\hat G$ can lift.

\begin{lemma}
\label{image of elliptic no strict rotation sc theorem}
    Let $E$ be an elliptic subgroup of $G$ (for its action on $X$), and let $\hat E$ be its image in $\hat G$. Then, $\hat E$ does not contain a strict rotation at any apex $\hat v \in \apices$.
\end{lemma}

\begin{proof}
    Let $E$ be an elliptic subgroup of $G$. Assume towards a contradiction that there is some $g \in E$ such that its image $\hat g$ is a strict rotation at some apex $\hat v \in \apices$. By Lemma \ref{balls around apices and stabilizers sc theorem}, there is some integer $k$ such that $\fix(\hat g ^k , \hat \delta)$ is contained in $B_{\hat X}(\hat v , \hat \delta)$. On the other hand, since $E$ is elliptic, so is the element $g^k$, and thus by Lemma \ref{charac subset of elliptic is quasi convex} there is some $x \in X$ such that $d_X (x , g^k \cdot x) \leq 11 \delta$. Since the projection map $\zeta : X \longrightarrow \hat X$ shortens distances, then, if we denote by $\hat x$ the image of $x$ on $\hat X$, we get that $d_{\hat X}(\hat x , \hat g ^k \cdot \hat x) \leq \hat \delta$. Now, from Lemma \ref{balls around apices and stabilizers sc theorem} \ref{stab of apex strict rot only fixes apex} we get a contradiction.  
\end{proof}

The next results allow us to compare distinct lifts of a given elliptic subgroup of $\hat G$.

\begin{lemma}\emph{(Compare \cite[Proposition 4.27]{coulon_4})}
\label{lifts of an elliptic subgroup sc theorem}
    Let $E$ be an elliptic subgroup of $G$ (for its action on $X$) and $S_1$ be a subset of $E$. Denote by $\hat S _1$ its image in $\hat G$. Let $\hat h \in \hat G$. Let $S_2$ be a preimage of $\hat h ^{-1} \hat S_1 \hat h$ such that $\text{Fix}(S_2 , \rho /100)$ is non-empty. Then, one of the following holds.
    \begin{itemize}
        \item The set $\hat S_1$ lies in a strict reflection group at some apex $\hat v \in \apices$.
        \item There exists a preimage $h$ of $\hat h$ in $G$ such that, for every $g \in S_1$, $h^{-1}gh$ is the unique preimage of $\hat h ^{-1} \hat g \hat h$ in $S_2$.
    \end{itemize}
\end{lemma}

\begin{proof}
    Since $E$ is elliptic and $S_1$ is a subset of $E$, we have that $\text{Fix}(S_1 , \rho /100)$ is non-empty. Fix now two points $x_1 , x_2 \in X$ such that they lie in $\text{Fix}(S_1 , \rho /100)$ and $\text{Fix}(S_2 , \rho /100)$ respectively, and denote by $\hat x _1$ and $\hat x _2$ their respective images in $\hat X$. Notice that both $\hat x_1$ and $\hat h ^{-1} \cdot \hat x_2$ lie in $\text{Fix}(\hat S_1 , \rho /100)$.

    By Lemma \ref{image of elliptic no strict rotation sc theorem} we have that $\hat S_1 \cap \stab (\hat v)$ is contained in a reflection group at $\hat v$ for every $\hat v \in \apices$.

    Assume now that $\hat S_1$ does not lie in a strict reflection group at any apex $\hat v \in \apices$. Let $h'$ be an arbitrary preimage of $\hat h ^{-1}$. Lemma \ref{lift sc theorem close to apices reflection group} applied to $S=S_1$, $x=x_1$ and $y= h' \cdot x_2$ gives that there exists some $u \in K$ such that for all $g \in S_1$ we have \[ d_{\dot X}(guh' \cdot x_2 , uh' x_2)=d_{\hat X}(\hat g \hat h ^{-1} \cdot \hat x_2 , \hat h ^{-1} \cdot \hat x _2). \]Write $h=(uh')^{-1}$. Let $g \in S_1$ be any element and $g'$ the (unique) preimage of $\hat h ^{-1} \hat g \hat h$ in $S_2$. We have now that $g'$ and $h^{-1}gh$ are two preimages of $\hat h ^{-1} \hat g \hat h$, both of them moving $x_2$ by at most $\rho /100$. Now, the triangle inequality yields \[ d_{\dot X}(g'{} ^{-1}h^{-1}gh \cdot x_2 , x_2) \leq d_{\dot X}(h^{-1}gh \cdot x_2  , x_2) + d_{\dot X}(x_2 , h^{-1}gh \cdot x_2) \leq \rho /50. \]We have that the element $g'{}^{-1} h^{-1}gh$ is in $K$, and furthermore, Theorem \ref{small cancellation theorem} \ref{isom sc theorem small ball around apex} yields $d_{\dot X}(x_2 , v(\mathcal{Q})) \geq \rho/2$. In consequence, Theorem \ref{small cancellation theorem} \ref{kernel acting freely on non-vertices} gives that $g'=h^{-1}gh$, and we obtain the desired conclusion.
\end{proof}

\begin{remark}
    Notice that, since the map $X \longrightarrow \hat X$ shortens the distances, if $\hat h$ is loxodromic, a preimage $h$ of $\hat h$ is necessarily loxodromic. 
\end{remark}

With Lemma \ref{lifts of an elliptic subgroup sc theorem} in mind, we are ready to state the consequences that this result has on the lifts of elliptic subgroups.

\begin{lemma}\emph{(Compare \cite[Proposition 4.28]{coulon_4})}
\label{sc theorem elliptics that share image}
    Let $F_1$ and $F_2$ be subgroups of $G$ such that $F_1$ is elliptic and $F_2$ is generated by a subset $S_2$ such that $\text{Fix}(S_2, \rho /100)$ is non-empty. Let $\hat F _1$ and $\hat F_2$ be their respective images in $\hat G$. Assume that $\hat F _1 = \hat F _2$. Then, one of the following holds.
    \begin{itemize}
        \item The subgroup $\hat F _1$ is contained in a strict reflection group at some apex $\hat v \in \apices$.
        \item There exists $u \in K$ such that $F_2=u^{-1} F_1 u$.
    \end{itemize}
\end{lemma}

\begin{proof}
    Assume that $\hat F_1$ does not lie on a strict reflection group at any apex $\hat v \in \apices$. Let $\hat S$ be the image of $S_2$ in $\hat G$ and $S_1$ the preimage of $\hat S$ in $F_1$. Notice that $\hat S$ does not lie on a strict reflection group at any apex $\hat v \in \apices$ either: since $S_2$ generates $F_2$, its image $\hat S$ generates $\hat F_2 = \hat F_1$, so if $\hat S$ was contained in a strict reflection group at some apex then the same would hold for $\hat F_1$. Thus, we can apply Lemma \ref{lifts of an elliptic subgroup sc theorem} to the sets $S_1$ and $S_2$ and $\hat h =1$, and we get that there is some $u \in K$ (as a preimage of 1) such that for every $s \in S_1$, the element $u^{-1}su$ is the preimage of $\hat s$ on $S_2$. Now, since $S_2$ generates $F_2$, then $uF_2 u^{-1}$ is contained in $F_1$. By Lemma \ref{sc theorem induces iso of elliptic} we have that the quotient map $G \twoheadrightarrow \hat G$ is one-to-one when restricted to $F_1$, and since $u \in K$, the image of $uF_2 u^{-1}$ is $\hat F_2 = \hat F _1$. In consequence, $F_2=u^{-1} F_1 u$.
\end{proof}

Lemma \ref{sc theorem elliptics that share image} has the following immediate consequence.

\begin{lemma}\emph{(Compare \cite[Corollary 4.29]{coulon_4})}
\label{sc theorem elliptics with conjugate image}
    Let $F_1$ and $F_2$ be subgroups of $G$ such that $F_1$ is elliptic and $F_2$ is generated by a subset $S_2$ such that $\text{Fix}(S_2, \rho /100)$ is non-empty. Let $\hat F _1$ and $\hat F_2$ be their respective images in $\hat G$. Assume that $\hat F _1 = \hat h ^{-1} \hat F _2 \hat h$. Then, one of the following holds.
    \begin{itemize}
        \item The subgroup $\hat F _1$ is contained in a strict reflection group at some apex $\hat v \in \apices$.
        \item There exists a preimage $h$ of $\hat h$ such that $F_1=h^{-1} F_2 h$.
    \end{itemize}
\end{lemma}

\begin{proof}
    For the case where $\hat F_1$ is not contained in a strict reflection group, we let $h'$ be any preimage of $\hat h$ in $G$, and we apply Lemma \ref{sc theorem elliptics that share image} to $F_1$ and $h'{}^{-1} F_2h'$. Thus, there exists $u \in K$ such that $u^{-1}F_1u=h'{}^{-1} F_2h'$, and the desired conclusion follows from the fact that, since $u \in K$, $h=h'u^{-1}$ is also a preimage of $\hat h$.
\end{proof}

\begin{remark}
    Notice that the assumptions on $F_2$ in Lemmas \ref{sc theorem elliptics that share image} and \ref{sc theorem elliptics with conjugate image} are immediately satisfied if $F_2$ is elliptic.
\end{remark}

Now, we are ready to classify exactly what elliptic subgroups of $\hat G$ can be lifted. Lemma \ref{characterization of lifting elliptics sc theorem} appears in \cite{coulon_4} as Proposition 4.30 with exactly the same statement.

\begin{lemma}
\label{characterization of lifting elliptics sc theorem}
    An elliptic subgroup $\hat F$ of $\hat G$ cannot be lifted if and only if it contains a strict rotation. In this case, the subgroup $\hat F$ fixes an apex $\hat v \in \apices$, and $\text{Fix}(\hat F , \hat \delta)$ is contained in $B_{\hat X}(\hat v , \hat \delta)$ (and thus $\hat v$ is the unique apex fixed by $\hat F$).
\end{lemma}

\begin{proof}
    The proof of \cite[Proposition 4.30]{coulon_4} adapts completely unchanged to our setting, with the following caveat: Lemma \ref{balls around apices and stabilizers sc theorem} \ref{stab of apex reflection group fixes point far away} is weaker than \cite[Proposition 4.13 (ii)]{coulon_4}. However, we have retrieved precisely the part that we need for this proof to work: for a reflection group $\hat A$ at an apex $\hat v$, there exists a point of $\hat X$ at distance greater than $\rho /2$ of $\hat v$ that is moved at most $\hat \delta$ by every element of $\hat A$.
\end{proof}

Lemma \ref{characterization of lifting elliptics sc theorem} has the following easy consequence, which we state here for future reference.

\begin{lemma}
\label{sc theorem ext of index 2 of lifting lifts}
    Let $\hat A$ be an elliptic subgroup of $\hat G$ containing a subgroup $\hat C$ of index 2 in $\hat A$ that lifts. Then, $\hat A$ itself lifts.
\end{lemma}

\begin{proof}
    Assume towards a contradiction that $\hat C$ lifts but $\hat A$ does not. Then, by Lemma \ref{characterization of lifting elliptics sc theorem} there is a strict rotation $\hat g$ in $\hat A \backslash \hat C$. However, by construction strict rotations have odd order $>3$, and this contradicts the assumption that $\hat C$ has index 2 in $\hat A$.
\end{proof}

We finish this section with a result on the maximal normal finite subgroups of loxodromic subgroups of the quotient group.

\begin{lemma}
\label{maximal normal fin sbgrp of loxo sgrp sc theorem}
    Let $\hat E$ be a loxodromic subgroup of $\hat G$ and let $\hat F$ be the maximal normal finite subgroup of $\hat E$. Then, one of the following holds.
    \begin{itemize}
        \item $\hat F$ is contained in a reflection group at some apex $\hat v \in \apices$.
        \item The subgroup $\hat E$ lifts.
    \end{itemize}
\end{lemma}

\begin{proof}
    Without loss of generality, we may assume that $\hat E$ is a maximal loxodromic subgroup of $\hat G$. Let $\hat g$ be a primitive loxodromic element of $\hat E$ and let $Y= Y_{\hat g}$. Notice that $Y$ is unbounded, and, in consequence, it contains a point in the image $\zeta (X)$ of $X$ under the map $\zeta : X \longrightarrow \hat X$. In particular, by Theorem \ref{small cancellation theorem} \ref{isom sc theorem small ball around apex}, this point is at distance larger that $\rho /2$ of the apex $\hat v$. We distinguish two cases.
    
    \textbf{Case 1:} There is some apex $\hat v \in \apices$ such that $d_{\hat X}(\hat v , Y) \leq \rho /10$. By Lemma \ref{charac subs of max finite sgrp} and the triangle inequality, we have that $Y$ is contained in $\text{Fix}(\hat F , 120 \hat \delta)$. Therefore, the elements of $\hat F$ move $\hat v$ by less than $\rho /4$, thus, $\hat F$ is contained in $\stab (\hat v)$. Now, by Lemma \ref{balls around apices and stabilizers sc theorem} \ref{stab of apex reflection group fixes point far away} the subgroup $\hat F$ cannot contain a strict rotation: for one such element $\hat g$, there is a power $\hat g ^k$ such that $\text{Fix}(\hat g ^k , 120 \hat \delta)$ is contained in $B_{\hat X}(\hat v , 120 \hat \delta)$, but $\text{Fix}(\hat F , 120 \hat \delta)$ is unbounded. Moreover, $\hat F$ cannot contain elements of two distinct strict reflection groups at $\hat v$: in such a case, their product would be a strict rotation. In consequence, we have that $\hat F$ must be contained in a reflection group at $\hat v$.

    \textbf{Case 2:} For every apex $\hat v \in \apices$ we have that $d_{\hat X}(\hat v , Y) > \rho /10$. Now, $Y$ is $2 \hat \delta$-quasi-convex. Therefore, by Lemma \ref{lift of a subset Xhat far from apices} and Remark \ref{remark on lifting figures far from apices} there is a subgroup $E$ of $G$ such that the quotient map $G \twoheadrightarrow \hat G$ induces an isomorphism from $E$ onto its image $\hat E$. In particular, $E$ is virtually cyclic (since the action of $\hat G$ on $\hat X$ is acylindrical, $\hat E$ is virtually cyclic), and in consequence it is elementary. Now, the preimage $g$ of $\hat g$ in $E$ must be loxodromic by its action on $X$ (since the map $\zeta : X \longrightarrow \hat X$ shortens distances). Since $E$ is an elementary subgroup of $G$ containing a loxodromic element, it must be loxodromic. That is, $E$ is a lift of $\hat E$.
\end{proof}

\subsection{Invariants on the quotient space}
\label{subsection invariants in the quotient}

In this subsection we will study the invariants of the action of $\hat G$ on $\hat X$. We keep the assumptions and notation introduced throughout this section. For future reference, we explicitly state a classification of apex stabilizers for quotients of tame actions.

\begin{lemma}
\label{apex stab for tame action}
    Let $\hat v$ be an apex in $\apices$. Let $(H,Y) \in \calQ$ be such that $\stab(Y)$ is a preimage of $\stab (\hat v)$, and let $n_h$ be such $H= \langle h^{n_h} \rangle$ for some primitive loxodromic element $h$. If the action of $G$ on $X$ is tame, then $\stab (\hat v)$ is isomorphic to one of the following groups:
    \begin{enumerate}
        \item $C_{n_h}$ if $\stab (Y) \cong \mathbb Z$;
        \item $C_{n_h} \times C_2$ if $\stab (Y) \cong \mathbb Z \times C_2$; or
        \item $D_{n_h}$ if $\stab (Y) \cong D_{\infty}$;
    \end{enumerate}
\end{lemma}

\begin{proof}
    This is a direct consequence of Remark \ref{virt cyc subgroups of tame action} and Lemma \ref{apex stab in the quotient}.
\end{proof}

We now start the study of the invariants of the action of $\hat G$ on $\hat X$.

\begin{proposition}\emph{(Compare \cite[Proposition 5.27]{coulon_1})}
\label{nu in the quotient}
    Suppose that the action of $G$ on $X$ is tame. Then, the invariant $\nu (\hat G, \hat X)$ is bounded from above by $\max \{ \nu(G,X) , 3 \}$.
\end{proposition}

\begin{proof}
    Let $m \geq \max \{ \nu(G,X) , 3 \}$ be an integer. Let $\hat g , \hat h$ be elements of $\hat G$, with $\hat h$ loxodromic, and such that $\hat g ,\, \dots , \hat h ^{-m} \hat g \hat h ^{m}$ generate an elliptic subgroup $\hat E$ of $\hat G$. For simplicity of notation, for $j \in \{ 0 , \dots , m\}$, write $\hat g _{j}= \hat h ^{-j} \hat g \hat h ^{m}$ and $\hat S = \{ \hat g_0 , \dots , \hat g _m \}$. We will consider two different cases.

    \textbf{Case 1:} there is $\hat v \in \apices$ such that $C_{\hat E}$ intersects $B(\hat v , \rho- 50 \hat \delta)$. Let $\hat x \in C_{\hat E}$ be at distance at most $\rho - 50 \hat \delta$ from $\hat v$. The elements of $\hat E$ move the points in $C_{\hat E}$ by at most $11 \hat \delta$. Hence, for every $\hat g \in \hat E$ the triangle inequality yields \[ d_{\hat X} (\hat v , \hat g \cdot \hat v) \leq d_{\hat X} (\hat v , \hat x) + d_{\hat X}(\hat g \cdot \hat x , \hat x) + d_{\hat X}(\hat g \cdot \hat v , \hat g \cdot \hat x) \leq 2\rho - 89 \delta. \] Thus, $\hat E$ (and so $\hat S$ as well) is contained in $ \stab (\hat v)$.

    We claim that $\hat S$ cannot contain a strict rotation: indeed, suppose that $\hat g_k$ is a strict rotation. By Lemma \ref{balls around apices and stabilizers sc theorem} \ref{stab of apex strict rot only fixes apex}, the apex $\hat v$ is the only one fixed by $\hat g _k$. However, a conjugate of $\hat g _k$ by $\hat h$ or by $\hat h ^{-1}$ is also in $\hat S$ and thus fixes $\hat v$, so $\hat g_k$ fixes either $\hat h \cdot \hat v$ or $\hat h ^{-1} \cdot \hat v$, and we get that $\hat h ^{-1} \cdot \hat v = \hat v$ or $\hat h  \cdot \hat v = \hat v$, contradicting in both cases the assumption that $\hat h$ is loxodromic.

    Thus, every element of $\hat S$ is contained in a reflection group at $\stab (\hat v)$.
    
    Notice that if for some $j \in \{0, \dots, m\}$ we have that $\hat g = \hat h ^{-j}\hat g \hat h ^{j}$, then we get that $\hat g$ centralizes the loxodromic element $\hat h ^{j}$, so it must fix its accumulation points at infinity, which coincide with those of $\hat h$. Therefore, we obtain $\hat g$ and $\hat h$ generate an elementary subgroup of $\hat G$.
    
    Since the action is tame, by Lemma \ref{apex stab for tame action} it is enough to consider two further subcases.

    \textit{Case 1a:} if $\stab (\hat v)$ is isomorphic to $C_n$ or $C_n \times C_2$, these groups contain at most one non-trivial element that is not a strict rotation, and thus, since $m \geq 2$, there is some $j \in \{0, \dots, m\}$ such that $\hat g = \hat h ^{-j}\hat g \hat h ^{j}$, so we conclude as before that $\hat g$ and $\hat h$ generate an elementary subgroup.

    \textit{Case 1b:} $\stab (\hat v)$ is isomorphic to $D_n$. If two of the $\hat g_j$'s coincide, then again as before we conclude that $\hat g$ and $\hat h$ generate an elementary subgroup. Otherwise, all the $\hat g_j$'s must be distinct involutions. In particular, $\hat g'=\hat g_0 \hat g_1$ is a strict rotation, and in consequence its unique fixed apex is $\hat v$. However, since $m \geq 2$, then $\hat h^{-1} \hat g' \hat h=\hat g_1 \hat g_2$ is also an element stabilizing $\hat v$, and therefore $\hat g'$ fixes the apex $\hat h \cdot \hat v$. Thus, we have $\hat h \cdot \hat v = \hat v$, contradicting the assumption that $\hat h$ is loxodromic.
    
    \textbf{Case 2:} there is no $\hat v \in \apices$ such that $C_{\hat E}$ intersects $B(\hat v , \rho - 50 \hat \delta)$. Then, $C_{\hat E}$ contains a point $\hat x$ in the $50 \hat \delta$-neighbourhood of $\zeta(X)$. Let $x$ be a preimage of $\hat x$ in $\dot X$. Consider the hull of $\hat E \cdot \hat x$, this is a $6\hat \delta$-quasi-convex subset contained in the $67 \hat  \delta$-neighbourhood of $\zeta(X)$. Therefore, by Lemmas \ref{isom proj from dot X of a set far away from apices} and \ref{lift of a subset Xhat far from apices} there exists an elliptic subgroup $E$ of $G$ (by its action on $\dot X$) such that the projection map $\pi: G \longrightarrow \hat G$ induces an isomorphism from $E$ onto $\hat E$, and for every $j \in \{0, \dots , m\}$, the preimage $g_j$ of $\hat g _j$ in $E$ satisfies \[ d_{\dot X}(g_j \cdot x , x) = d_{\hat X}(\hat g _j \cdot \hat x , \hat x) \leq 166 \hat \delta. \] In particular, we get that for every $j \in \{ 0, \dots , m-1 \}$, \[ d_{\hat X}(\hat g_j \hat h \cdot \hat x , \hat h \cdot \hat x) = d_{\hat X}(\hat g _{j+1} \cdot \hat x , \hat x) \leq 166 \hat \delta. \]

     We now fix a preimage $h$ of $\hat h$ such that $d_{\dot X}( h \cdot x , x)\leq d_{\hat X}( \hat h \cdot \hat x , \hat x) + \hat \delta/2$, and we let $\gamma$ be a $(1, \hat \delta/2)$-quasi-geodesic joining $x$ and $h \cdot x$. The path $\hat \gamma$ induced by $\gamma$ is a $(1,  \hat \delta)$-quasi-geodesic joining $ \hat x$ and $\hat h \cdot \hat x$. Suppose that there is some $\hat v \in \apices$ such that $\langle \hat x , \hat h \cdot \hat x \rangle \leq \rho/4$. Then, Lemma \ref{distances within quadrangles} together with the triangle inequality gives that every element of $\hat S ' = \{ \hat g_0, \dots \hat g _{m-1}  \}$ is in $\stab (\hat v)$. Now, we conclude as in Case 1 that $\hat g$ and $\hat h$ generate an elementary subgroup.
    
    If there is no $\hat v \in \apices$ such that $\langle \hat x , \hat h \cdot \hat x \rangle \leq \rho/4$, we can apply Lemma \ref{lift sc theorem close to apices locally trivial} to the path $\hat \gamma$ and the set $\hat S ' $, and so for every $j \in \{ 0, \dots , m-1 \}$ we get that $  d_{\dot X}( g_j  h \cdot  x ,  h \cdot  x) = d_{\hat X}(\hat g_j \hat h \cdot \hat x , \hat h \cdot \hat x)$. We claim now that for $j \in \{ 0, \dots , m-1 \}$, $h^{-1}g_jh=g_{j+1}$. Indeed, we have that \[d_{\dot X}(g_{j+1} \cdot x , x) =d_{\hat X}(\hat g_{j+1} \cdot \hat  x , \hat x)=d_{\hat X}(\hat g_j \hat h \cdot \hat x, \hat h \cdot \hat x)= d_{\dot X}( g_j  h \cdot  x,  h \cdot  x)= d_{\dot X}( h^{-1} g_j  h \cdot  x,    x).\] But then, \[ d_{\dot X}(h^{-1}g_j ^{-1}hg_{j+1} \cdot x , x) \leq d_{\dot X}(g_{j+1} \cdot x , x) + d_{\dot X}(g_{j+1} \cdot x , x) = 2 d_{\hat X}(\hat g_{j+1} \cdot \hat x , \hat x) \leq 334 \hat \delta. \]We have that $g_{j+1}$ and $h^{-1}g_j h$ are two preimages of the same element $\hat g_{j+1}$, therefore we get that $h^{-1}g_j ^{-1}hg_{j+1}$ is in $K$. Hence, by Theorem \ref{small cancellation theorem} \ref{kernel acting freely on non-vertices}, we get that $h^{-1}g_jh=g_{j+1}$.

    In consequence, for every $j \in \{0, \dots , m\}$, the element $h^{-j}gh^j$ belongs to $E$, so $g, h^{-1}gh, \dots h^{-m}gh^m$ generate an elliptic subgroup of $G$ (by its action on $\dot X$). Furthermore, by Lemma \ref{lift of a subset Xhat far from apices} applied to $C_{\hat E}$, the subset $C_{E}$ contains a point in the $67 \hat \delta$-neighbourhood of (the image by inclusion into $\dot X$ of) $X$. Let $p$ be a $\hat \delta$-projection of $C_{E}$ on $X$. For all $g \in E$ the point $g \cdot p $ is a $\hat \delta$-projection of $g \cdot x$ on $X$. By the triangle inequality and Lemma \ref{dist on x compared cone off} we get that \[ \mu (d_{X}(p, g \cdot p)) \leq d_{\dot X}(p , g \cdot p) \leq 2d_{\dot X}(x,p) + d_{\dot X}(x, g \cdot x) \leq 147 \hat \delta < \rho. \] Hence, by Lemma \ref{bound on mu for small t} we get that for all $g \in E$ we have that $d_X (p, g \cdot p) \leq \pi \sinh (\rho /2)$. In particular, the orbits of $E$ in $X$ are bounded, so $E$ is elliptic by its action on $X$. Therefore, since $m \geq \nu (G,X)$, we get that $g$ and $h$ generate an elementary subgroup of $G$, and thus so do $\hat g$ and $\hat h$.
\end{proof}

\begin{proposition}
\label{tameness in the quotient}
    If the action of $G$ on $X$ is tame, then so is the action of $\hat G$ on $\hat X$.
\end{proposition}

\begin{proof}
    First, assume that $\hat G$ has a subgroup $\hat F$ of order 4. Notice that $\hat F$ is necessarily elliptic. Then, since $G$ has no subgroup of order 4, $\hat F$ cannot lift. Therefore, by Proposition \ref{characterization of lifting elliptics sc theorem}, it is contained in $\stab (\hat v)$ for some $\hat v \in \apices$. However, no apex stabilizer contains a subgroup of order 4 by Lemma \ref{apex stab for tame action}.

    Now, let $\hat E$ be a loxodromic subgroup of $\hat G$, and let $\hat F$ be its maximal normal finite subgroup. If $\hat E$ lifts, then $\hat F$ has order at most two since the action of $G$ on $X$ is tame. Otherwise, by Proposition \ref{maximal normal fin sbgrp of loxo sgrp sc theorem}, $\hat F$ is contained in a strict reflection group for some $\hat v \in \apices$. However, by Lemma \ref{apex stab for tame action} these subgroups have order at most two themselves.
\end{proof}

From Propositions \ref{nu in the quotient} and \ref{tameness in the quotient} and the definition of the parameter $\tau$ we obtain Corollary \ref{tau in the quotient}.

\begin{corollary}
\label{tau in the quotient}
    If the action of $G$ on $X$ is tame, then the invariant $\tau(\hat G, \hat X)$ is at most $\tau(G,X)$.
\end{corollary}

The next result appears as Proposition 5.29 in \cite{coulon_1}. As it was explained before, the small cancellation assumptions in that context are not exactly the same as in this article. However, the proof of that result works \textit{verbatim} for our case.

\begin{proposition}\cite[Proposition 5.29]{coulon_1}
\label{intersection of axis in the quotient}
    Let $m$ be an integer, $\hat g_0, \dots \hat g_m$ be elements of $\hat G$ of translation length at most $L_S \hat  \delta$. Then one of the following holds:
    \begin{enumerate}
        \item There is $\hat v \in \apices$ such that $\hat g_i \in \text{Stab}(\hat v)$ for every $i \in \{0, \dots , m\}$.
        \item There exist preimages $g_i$ of $\hat g_i$ for $i \in \{ 0, \dots ,m \}$ with translation length at most $\pi \sinh ((L_S + 34) \hat \delta)$ such that \[ A(\hat g_0 , \dots , \hat g_m ) \leq A( g_0 , \dots ,  g_m ) + \pi \sinh ((L_S + 34) \hat \delta) + (L_S + 45) \hat \delta. \]
    \end{enumerate}
\end{proposition}

The next corollary allows us to control the parameter $\Omega(\hat G , \hat X)$ in terms of $\Omega(G,X)$ and $\tau( G,  X)$. It is a direct adaptation of Corollary 5.30 in \cite{coulon_1}, modulo replacing $A$ by $\Omega$ and $\nu$ by $\tau$ when appropriate. For the sake of completeness, we include the proof.

\begin{corollary}
\label{omega in the quotient}
    The invariant $\Omega(\hat G, \hat X)$ is bounded by \[ \Omega(\hat G, \hat X) \leq \Omega(G,X) + (\tau(G ,  X)+4) \pi \sinh (2L_S \hat \delta). \]
\end{corollary}

\begin{proof}
    We denote by $\hat \tau$ the invariant $\tau(\hat G, \hat X)$, and, similarly, $\tau = \tau(G,X)$. Write $\mathcal{A}'$ for the set of $(\hat \tau + 1)$-tuples $(\hat g_0 , \dots \hat g_{\hat \tau})$ of elements of $\hat G$ generating a non-elementary subgroup and of translation length at most $L_S \hat \delta$.

    Since these elements generate a non-elementary subgroup, there is no $\hat v \in \apices$ such that they are all contained in $\text{Stab} (\hat v)$. Therefore, by Proposition \ref{intersection of axis in the quotient}, we have preimages $g_i$ of $\hat g_i$ for $i \in \{0,\dots , \hat \tau\}$ such that
    \begin{itemize}
        \item the translation length is bounded by $\pi \sinh ((L_S + 34) \hat \delta)$; and
        \item $A(\hat g_0 , \dots , \hat g_{\hat \tau} ) \leq A( g_0 , \dots ,  g_{\hat \tau} ) + \pi \sinh ((L_S + 34) \hat \delta) + (L_S + 45) \hat \delta.$
    \end{itemize}
    Since the image of an elementary subgroup of $G$ (by its action on $X$) is elementary, the elements $g_0 , \dots , g_{\hat \tau}$ must generate a non-elementary subgroup. We have, by Corollary \ref{tau in the quotient}, that $\hat \tau \leq \tau$, thus by Corollary \ref{less than tau elements gen non-elem} we obtain
    \begin{equation*}
        \begin{split}
            A(\hat g_0 , \dots , \hat g_{\hat \tau}) &\leq (\tau + 2)\pi \sinh ((L_S + 34) \hat \delta) + \Omega(G,X) + 680 \delta + \pi \sinh ((L_S + 34) \hat \delta) + (L_S + 45) \hat \delta \\
            &\leq (\tau + 4)\pi \sinh (2L_S \hat \delta) + \Omega(G,X)
        \end{split}
    \end{equation*}
    Since this inequality holds for any $(\hat \tau + 1)$-tuple in $\mathcal{A}'$, the announced conclusion follows.
\end{proof}

Definition \ref{def stable family}, Lemma \ref{asymp trans length stable subset} and Corollary \ref{inj radius of stable subset} appear in \cite{amelio_andre_tent} as Definition 4.22, Lemma 4.23 and Corollary 4.24 (for the more general case of a WPD action, here the reader may replace `non-elliptic' for `loxodromic'). We include the statements here. The proofs of the results work exactly as in the reference (thus, in particular, as in Proposition 5.31 of \cite{coulon_1}).

\begin{definition}
\label{def stable family}
    Let $Q'$ be a subset of $G$ and $\hat{Q}'$ its image in $\hat{G}$. We say that $Q'$ is \emph{stable} with respect to $\mathcal{Q}$ if the following property is satisfied: let $\hat{g}$ be a non-elliptic element of $\hat{Q}'$. Suppose that there is a subset $A$ of $\dot{X}$ such that the projection $\zeta: \dot X \longrightarrow \hat{X}$ induces an isometry from $A$ onto the axis $A_{\hat{g}^{m}}$ for some $m \in \mathbb{N}$ and the projection $G \twoheadrightarrow \hat{G}$ induces an isomorphism from $\text{Stab}(A)$ onto $\text{Stab}(A_{\hat{g}^{m}})$. Let $g$ be the preimage of $\hat{g}$ in $\text{Stab}(A)$. Then $g \in Q'$.
\end{definition}


For all the applications in this article, the stable family can be taken to be just the whole set of loxodromic elements of $G$. However, the more general results from Section \ref{section partial periodic quotients} may be useful when trying to impose torsion only on a reduced subset of elements of $G$, as is in the case of \cite{amelio_andre_tent} (where the family $Q'$ is taken to be the loxodromic translations of $G$).

\begin{lemma}
\label{asymp trans length stable subset}\emph{(Compare \cite[Proposition 5.31]{coulon_1})}
    Let $Q'$ be a stable subset of $G$. Denote by $l$ the infimum over the asymptotic translation length in $X$ of loxodromic elements of $Q'$ that do not belong to $\text{Stab}(Y)$ for $(H,Y) \in \mathcal{Q}$. Let $\hat{g}$ be a non-elliptic element of $\hat{Q}'$. If every preimage of $\hat{g}$ in $G$ is loxodromic, then we have \[[\hat{g}]^{\infty} \geq \text{min} \biggl ( \biggl \{ \frac{l \hat{\delta}}{\pi \text{sinh}(26 \hat{\delta})} , \hat{\delta} \biggr \} \biggr ).\]
\end{lemma}

\begin{corollary}
\label{inj radius of stable subset}
    Let $Q'$ be a stable subset of $G$. Denote by $l$ the infimum over the asymptotic translation length in $X$ of loxodromic elements of $Q'$ that do not belong to $\text{Stab}(Y)$ for $(H,Y) \in \mathcal{Q}$. Then we have \[ r_{\text{inj}}(\hat{Q}', \hat{X}) \geq \text{min} \biggl ( \biggl \{ \frac{l \hat{\delta}}{\pi \text{sinh}(26 \hat{\delta})} , \hat{\delta} \biggr \} \biggr ).\]
\end{corollary}

\section{Partial periodic quotients of groups with even torsion}
\label{section partial periodic quotients}

The goal of this section is to prove a small cancellation result for certain groups with even torsion exhibiting some form of negative curvature. We will obtain such a result by taking a sequence of quotients of the original group as in Section \ref{section small cancellation} (we will call them \textit{SC-quotients}), where the small cancellation parameters $\Delta$ and $T$ from Theorem \ref{small cancellation theorem} will be controlled by the parameters $\rinj$, $\tau$ and $\Omega$, as well as by the tameness of the actions. Then, we will pass to the limit of this sequence (we will call the group obtained in the limit a \textit{PP-quotient}).

\subsection{The induction step: SC-quotients}
\label{subsection SC-quotient}

In this subsection, we prove the induction step in the aforementioned construction. More concretely, we will prove the following result.

\begin{proposition}
\label{induction step quotient}
    There exist positive constants $\rho_0$,  $\delta_1$ and $L_S$ such that for every integer $\tau_0 \geq 3$ there is a positive integer $n_0$ with the following properties. Let $G$ be a group acting by isometries on a $\delta_1$-hyperbolic length space $X$.
	We assume that this action is acylindrical and non-elementary.
	Let $n_1 \geq n_0$. Let $\expfamily=\{ n^{(m)} \, : \, 1 \leq m \leq l' \}$ be a finite family of odd integers such that $n^{(m)} \geq n_1$ for $m \in \{1, \dots ,l'\}$. Let $Q$ be a conjugation invariant set of elements of $G$. We make the following assumptions.
    
    \begin{enumerate}[label=(\arabic*)]
        \item \label{tame in SC quotient}
        The action of $G$ on $X$ is tame,
        \item \label{tau in SC-quotient}
		$\tau(G,X) \leq \tau_0$,
        \item \label{Omega in SC-quotient}
		$\Omega(G,X) \leq   6\pi \tau_0 \sinh (2L_S\delta_1)$,
		\item \label{rinj in SC-quotient}
		$ r_{\text{inj}} (Q,X) \geq 2\delta_1 \sqrt {\frac {L_S\sinh (\rho_0)}{n_1\sinh (26 \delta_1)}}$.
    \end{enumerate}

    We define an equivalence relation on the set of primitive loxodromic elements of $Q$ as follows: we say that $h$ and $h'$ are equivalent if they generate an elementary subgroup or if they are conjugate in $G$. Let $P$ be a maximal subset of loxodromic elements $h$ of $Q$ which are primitive, such that $ [h] \leq L_S\delta_1 $ and such that no two elements are equivalent.

    Let $K=\langle h^{n_h} \, : \, h \in P\rangle^G$ with $n_h \in \expfamily$ for all $h \in P$. Put $\hat G=G/K$.
	Then there exists a $\delta_1$-hyperbolic length space $\hat X$ on which $\hat G$ acts by isometries. This action is non-elementary, tame and acylindrical.

    In addition, the action of $\hat{G}$ on $\hat{X}$ satisfies Assumptions \ref{tau in SC-quotient} and \ref{Omega in SC-quotient}. We define a family $\mathcal{Q}$ as follows: $\mathcal{Q}= \{ (\langle h^{n_h} \rangle , Y_h) \, : \, h \in P' \}$, where the set $P'$ is defined as $P'= \bigcup \limits_{g \in G}g^{-1}Pg$, and the exponents $n_h$ are taken to be exactly as in the definition of the subgroup $K$ for the elements of $P$ and invariant under conjugation. If $Q$ is stable with respect to the family $\mathcal{Q}$, then $ r_{\text{inj}} (\hat Q, \hat X) \geq 2\delta_1 \sqrt {\frac {L_S\sinh (\rho_0)}{n_1\sinh (26 \delta_1)}}$, where $\hat Q$ is the set of images of elements of $Q$ in the quotient that remain loxodromic for their action on $\hat X$.

    Furthermore, for the quotient map $G \twoheadrightarrow \hat G$, denote by $\hat g$ the image of an element $g$ (respectively, denote by $\hat E$ the image of a subgroup $E$). Then, this map has the following properties.
	\begin{itemize}
		\item  For every $g \in G$ we have 
		\begin{displaymath}
			[{\hat g}]^{\infty}_{\hat X} \leq \frac 1{\sqrt {n_1}} \left(\frac {4\pi}{\delta_1}\sqrt{\frac {\sinh (\rho_0)\sinh (26 \delta_1)}{L_S}}\right)[g]^{\infty}_{ X}. 
		\end{displaymath}
		\item For every elliptic subgroup $E$ of $G$, the quotient map $G \twoheadrightarrow \hat G$ induces an isomorphism from $E$ onto its image $\hat E$, which is itself elliptic.
		\item Let $\hat g$ be an elliptic element of $\hat G$. 
		Either there is $n \in \expfamily$ such that $\hat g^{2n} = 1$, or $\hat g$ is the image of an elliptic element of $G$.
		\item Let $u,u' \in G$ be such that $ [u] < \rho_0 /100$ and $u'$ is elliptic.
		Assume that $\hat u$ and $\hat u'$ are conjugate in $\hat G$, say by an element $\hat h$, and that $u'$ is not contained in a loxodromic subgroup of dihedral type on which it is not in its maximal normal finite subgroup. Then, $u$ and $u'$ are conjugate in $G$, and the conjugating element can be taken to be a preimage $h$ of $\hat h$.
	\end{itemize}
\end{proposition}

\begin{remark}
    For the applications of this article, the reader can just take $Q$ to be the family of all loxodromic elements of $G$ (which is, indeed, stable).
\end{remark}

\begin{remark}
\label{remark controlling max elliptic of loxodromic}
    In Subsection \ref{subsection PP-quotient}, we will iteratively apply Proposition \ref{induction step quotient} \emph{for the same family of positive integers} $\expfamily$, so that in particular we obtain a bound in the exponent of the new torsion we are creating. Therefore, we need to be able to take the same value of the parameter $n_1$ of Proposition \ref{induction step quotient} at every step of the construction. For this purpose, we need control over the elliptic subgroups normalized by loxodromic elements, both in the group $G$ and in the quotient $\hat{G}$. In the setting of \cite{coulon_1} and \cite{coulon_4}, this is accomplished by the parameters $e(G,X)$ and $\mu(\mathcal{E})$ (plus an assumption on the structure of dihedral pairs) respectively. In our case, Assumption \ref{tame in SC quotient} will play this role.
\end{remark}

\begin{proof}
    The proof closely follows the proofs of Proposition 6.1 in \cite{coulon_1} and of Proposition 4.25 in \cite{amelio_andre_tent}. We include some details of the proof, focusing on the construction of the space $\hat X$ for later reference, and we will refer to the aforementioned proofs whenever parts of this proof work exactly as in those cases.

    Fix now an integer $\tau_0 \geq 3$. Let $\rho_0$, $L_S$, $\delta_0$, $\delta_1$ and $\Delta_0$ be the parameters coming from small cancellation as in Theorem \ref{small cancellation theorem}. We define a rescaling parameter $\lambda_l$ depending on an integer $l$ as \[ \lambda_l = \frac{4 \pi}{\delta_1} \sqrt{\frac{\sinh (\rho_0)\sinh(26 \delta_1)}{l L_S}}. \] Now we set the critical exponent $n_0$: let $n_0$ be the smallest integer greater than 100 such that for every $l \geq n_0$ we have \[ \lambda_l \delta_1 \leq \delta_0, \] \[ \lambda_l (6\pi \tau_0 \sinh (2L_S\delta_1) + 118 \delta_1) \leq \min \{ \Delta_0 , \pi \sinh (2L_S \delta_1) \}, \] \[ \lambda_l \frac{L_s \delta_1 ^2}{2 \pi \sinh(26 \delta_1)} \leq \delta_1, \] and \[ \lambda_l \rho _0 \leq \rho_0. \] Let $n_1 \geq n_0$ and $\expfamily=\{ n^{(m)} \, : \, 1 \leq m \leq l' \}$ be a family of odd integers, all greater than $n_1$. Set $ \lambda = \lambda_{n_1}$.

    Assume that we have a group $G$ acting on a hyperbolic space $X$ and $Q$ a set of elements of $G$ such that the assumptions of Proposition \ref{induction step quotient} are satisfied for $\tau_0$ and $n_1$. For the rest of the proof, unless explicitly stated otherwise, we will consider the action of $G$ on the rescaled space $\lambda X$. This space will be $\delta$-hyperbolic with $\delta = \lambda \delta_1 \leq \delta_0$. The action of $G$ on $\lambda X$ is still acylindrical, non-elementary and tame. Let $P$, $P'$ and $\mathcal{Q}$ be as in the statement of the proposition. By design, the family $P'$ is closed under conjugation, since this is the case for the family $Q$, and the tameness of the action will give (in virtue of Remark \ref{virt cyc subgroups of tame action}) that two equivalent primitive loxodromic elements of $Q$ differ only by inversion or multiplication by a central element of their maximal loxodromic subgroup. Moreover, this classification of loxodromic subgroups of $G$ gives $\langle h^{n_h} \rangle \trianglelefteq \stab (Y_h)$. In particular, if $(H,Y)$ and $(H',Y')$ are pairs in $\calQ$ such that $Y=Y'$, then $H=H'$.
    
    \textit{Claim 1:$\Delta(\mathcal{Q}) \leq \Delta_0$ and $T(\mathcal{Q}) \geq 8 \pi \sinh (\rho_0)$.} Let $h_1$ and $h_2$ be two elements of $P$ such that $(\langle h_1^{n_{h_1}} \rangle , Y_{h_1}) \neq (\langle h_2^{n_{h_2}} \rangle , Y_{h_2})$. By Lemma \ref{axis vs cylinder} we have that $Y_{h_i}$ is contained in the $52\delta$-neighbourhood of $A_{h_i}$, so by Lemma \ref{intersection of quasi-convex subsets}, since $Y_{h_i}$ is strongly quasi-convex, we obtain \[ \diam (Y_{h_1}^{+5\delta} \cap Y_{h_2}^{+5\delta}) \leq \diam (A_{h_1}^{+13\delta} \cap A_{h_2}^{+13\delta}) + 118\delta. \]By construction of $P'$, we have that $h_1$ and $h_2$ generate a non-elementary subgroup of $G$, and their translation length in $\lambda X$ is at most $L_S \delta$. On the other hand, by Remark \ref{invariants in rescaled space} we get $\Omega(G,\lambda X) \leq \lambda \Omega(G,X)$. In consequence, we obtain \[ \diam (Y_{h_1}^{+5\delta} \cap Y_{h_2}^{+5\delta}) \leq \lambda \Omega(G,X) + 118 \lambda \delta_1 \leq \lambda (6 \pi \tau_0 \sinh (2L_s \delta_1)+ 118 \delta_1), \]and therefore, by the choice of $\lambda$ we get $\Delta(\mathcal{Q}) \leq \Delta_0$.

    For the second part of the claim, again by Remark \ref{invariants in rescaled space}, Assumption \ref{rinj in SC-quotient} and the definition of the rescaling parameter $\lambda$ we have that for all $h \in P'$ \[ \rinj (Q, \lambda X) = \lambda \rinj(Q,X) \geq \frac{8 \pi \sinh(\rho_0)}{n_1} \geq \frac{8 \pi \sinh(\rho_0)}{n_h}.\] Therefore, since $P'$ is a subset of $Q$, we obtain $[h^{n_h}]^\infty = n_h [h]^\infty \geq 8 \pi \sinh(\rho_0)$. Thus, from Lemma \ref{asympt and trans length relation} we conclude $T(\mathcal{Q}) \geq 8 \pi \sinh(\rho_0)$.

    In view of the previous claim, we can now apply Theorem \ref{small cancellation theorem} to the action of $G$ on the rescaled space $\lambda X$. We denote by $\dot X$ the cone-off space of radius $\rho_0$ over $X$ relative to the family $\mathcal{Y}=\{Y \, : \, (H,Y) \in \mathcal{Q}\}$ and by $\hat X$ the quotient of $\dot X$ by the action of $K$. By Theorem \ref{small cancellation theorem}, $\hat X$ is a $\hat \delta$-hyperbolic length space with $\hat \delta \leq \delta_1$ on which $\hat G$ acts by isometries. Since the action of $G$ on $X$ is tame and the set $\expfamily$ is finite, we have $\lvert \stab(Y_h) / \langle h^{n_h} \rangle \rvert \leq 2 \max (\expfamily )$, so by Lemma \ref{acylindricity of xhat} the action of $\hat G$ on $\hat X$ is acylindrical. Furthermore, by Lemma \ref{tameness in the quotient} this action is tame.

    The fact that $\tau(\hat G, \hat X) \leq \tau_0$ is a consequence of Remark \ref{invariants in rescaled space} and Corollary \ref{tau in the quotient}.

    \textit{Claim 2: $\Omega(\hat G, \hat X) \leq 6 \pi \tau_0 \sinh (2L_S \delta_1)$.} By Corollary \ref{omega in the quotient} we have \[ \Omega(\hat G,  \hat X) \leq \Omega(G,\lambda X) + (\tau_0+4) \pi \sinh (2L_S \delta_1). \]Now, Remark \ref{invariants in rescaled space} and the definition of $\lambda$ give us \[ \Omega(G, \lambda X)= \lambda \Omega(G,X) \leq \lambda 6 \pi \tau_0 \sinh(2L_S \delta_1) \leq \pi \sinh(2L_S \delta_1). \]Therefore, $\Omega(\hat G , \hat X)$ is at most $(\tau_0+5)\pi \sinh(2L_S \delta_1)$, and the fact that $\tau_0 \geq 1$ completes the proof of the claim.

    \textit{Claim 3: if $Q$ is stable with respect to $\mathcal{Q}$, then $\rinj(\hat Q , \hat X) \geq 2\delta_1 \sqrt {\frac {L_S\sinh (\rho_0)}{n_1\sinh (26 \delta_1)}}$}. Let $h$ be a loxodromic element of $Q$ that is not in the stabilizer of $Y_g$ for $g \in P'$. By construction of $P'$, $h$ has translation length at least $L_S \delta_1$ in $X$, and thus by Remark \ref{invariants in rescaled space} asymptotic translation length greater than $\lambda L_S \delta_1 /2$ in $\lambda X$. If $Q$ is stable with respect to $\mathcal{Q}$, we can apply Corollary \ref{inj radius of stable subset} (and our definition of $\lambda$) to get \[ \rinj (\hat Q , \hat X) \geq \min \{ \frac{\lambda L_S \delta_1^2}{2 \pi \sinh (26 \delta_1)} , \delta_1 \} = \frac{\lambda L_S \delta_1^2}{2 \pi \sinh (26 \delta_1)}=2\delta_1 \sqrt{\frac{L_S \sinh(\rho_0)}{n_1 \sinh(26 \delta_1)}}.\]

    We now focus on the announced properties of the map $G \longrightarrow \hat G$. The claim about $[\hat g]_{\hat X}^\infty$ follows from the fact that $[g]_{\lambda X}^\infty \leq \lambda [g]_{X}^\infty$ and that the map $\lambda X \longrightarrow \hat X$ shortens distances. The second claim is contained in Lemma \ref{sc theorem induces iso of elliptic}.

    \textit{Claim 4: let $\hat g$ be an elliptic element of $\hat G$. 
		Either $\hat g^{2n} = 1$ or $\hat g$ is the image of an elliptic element of $G$.} By Lemma \ref{characterization of lifting elliptics sc theorem} applied to $\langle \hat g \rangle$, either this subgroup can be lifted (and therefore $\hat g$ has an elliptic preimage) or it contains a strict rotation, and thus $\hat g$ is itself a strict rotation in some $\hat v \in \apices$. In this last case, there is some $n \in \expfamily$ such that $\hat g ^n$ is locally trivial in $\hat v$, so it is contained in an elliptic subgroup $\hat F$ of $\hat G$ that is the isomorphic image of an elliptic subgroup $F$ of $G$ normalized by a loxodromic element. Now, the action of $G$ on $X$ is tame, so the order of $F$ (and thus of $\hat F$) is at most $2$. Therefore, $(\hat g ^{n})^{2}=1$.

        The last property of the projection map follows directly from Proposition \ref{sc theorem elliptics with conjugate image} applied to $\langle u \rangle$ and $\langle u' \rangle$, and the fact that strict reflection groups of $\hat G$ are images of elliptic subgroups of $G$ contained in loxodromic subgroups but not in their maximal elliptic normal subgroup.
\end{proof}

\textbf{Terminology.} For the remainder of this article, we will call a group obtained as the quotient of a group $G$ as provided by Proposition \ref{induction step quotient} a \textit{small cancellation quotient} (or an \textit{SC-quotient}) of $G$, and we will denote it by $\hat G$. Similarly, if $X$ is the length space on which we consider the action of $G$ to apply Proposition \ref{induction step quotient}, we will write $\hat X$ for the hyperbolic length space on which $\hat G$ acts as provided by the proposition.

\subsection{The limit step: PP-quotients}
\label{subsection PP-quotient}

In this subsection, we construct partial periodic quotients of a group $G$ exhibiting negative curvature features by taking a sequence of SC-quotients as in Proposition \ref{induction step quotient}.

For this purpose, notice that an SC-quotient of a group $G$ exists whether the family $Q$ considered in Proposition \ref{induction step quotient} is stable or not. However, if we want to iteratively apply the proposition \emph{for the same family} $\mathcal{N}$, we will need to take an SC-quotient of $\hat G$ \emph{with the same value of the parameter} $n_1$, and for this we need to bound the injectivity radius of $\hat Q$. This is the point where the stability of the family $Q$ is key. In particular, if $\hat Q$ is stable with respect to $\hat {\mathcal{Q}}$ (where $\hat{\mathcal{Q}}$ is the family constructed from $\hat G$ and $\hat Q$ exactly as $\mathcal{Q}$ is constructed from $G$ and $Q$) then all the assumptions of Proposition \ref{induction step quotient} are satisfied. To ensure that the necessary bound on the injectivity radius holds throughout the whole inductive process, in \cite[Definition 4.26]{amelio_andre_tent} the notion of a \textit{strongly stable} family is introduced. We re-introduce it now, with a slight modification to better suit our setting.

\begin{definition}
\label{def strongly stable family}
    Let $Q$ be a subset of $G$. We say that $Q$ is \emph{strongly stable} if the following property holds.
    
    Suppose $\{ \hat{G}_{i} \, : \, 0 \leq i \leq k \}$ is a finite sequence of quotients obtained from $G=\hat{G}_0$ by successive applications of Theorem~\ref{small cancellation theorem}, where the space on which $\hat{G}_0$ acts is (a possibly rescaled version of) $X$, and the space $\hat{X}_{i+1}$ on which $\hat{G}_{i+1}$ acts is (a possibly rescaled version of) the space $\hat{X}$ provided by Lemma~\ref{small cancellation theorem} when we consider the action of $\hat{G}_i$ on $\hat{X}_i$. Suppose furthermore that at step $i$ the family $\mathcal{Q}_i$ was constructed taking the subsets $H_i$ of $\hat{G}_i$ to be the $n_h$-th power of a primitive loxodromic element $\hat h$ in the image of $Q$ in $\hat{G}_i$ and $n_h \in \expfamily$, and $Y$ to be the cylinder of this $n_h$-th power.

    Then  the image $\hat{Q}_i$ of $Q$ in $\hat{G}_i$ is stable with respect to $\mathcal{Q}_i$.
\end{definition}

The previous \textit{ad hoc} definition has the following immediate consequence.

\begin{lemma}
\label{strongly stable passes to quotient}
    In the setting of Definition \ref{def strongly stable family}, if $Q$ is strongly stable, then so is $\hat Q$.
\end{lemma}

In particular, with the additional assumption that the family $Q$ is strongly stable, we can indeed iterate the application of Proposition \ref{induction step quotient}.

\begin{remark}
    For all the applications of this article, the reader can just take the family $Q$ to be the set of all loxodromic elements of $G$, which will indeed be a strongly stable family. However, this result can be useful when trying to impose torsion only on a certain subset of elements of $G$, as is the case in \cite{amelio_andre_tent}. Thus, we will keep this more general version.
\end{remark}

\begin{theorem}\emph{(Compare \cite[Theorem 6.9]{coulon_1} and \cite[Theorem 4.28]{amelio_andre_tent})}
\label{limit step quotient}
    Let $X$ be a $\delta$-hyperbolic length space, and let $G$ be a group acting acylindrically and non-elementarily by isometries on $X$. We suppose that the action is tame. Let $Q$ be a conjugation invariant strongly stable family of elements of $G$. We assume, in addition, that the invariants $\tau(G,X)$ and $\Omega(G,X)$ are finite and that $\rinj (Q,X)$ is positive. Then, there exists a normal subgroup $K$ and a critical exponent $n_1$ depending only on $\delta$, $\tau(G,X)$, $\Omega(G,X)$ and $\rinj (Q,X)$ such that, for every finite family $\expfamily=\{ n^{(m)} \, : \, 1 \leq m \leq l' \}$ of odd integers with $n^{(m)} \geq n_1$ for all $m \in \{1, \dots , l' \}$, the following holds. Write $\bar G = G/K$:
    \begin{enumerate}[label=(\arabic*)]
        \item \label{elliptic is preserved by pp quotient} if $E$ is an elliptic subgroup of $G$, the projection $G \twoheadrightarrow \bar G$ induces an isomorphism from $E$ onto its image;
        \item \label{kernel of pp quotient is purely loxodromic} every non-trivial element of $K$ is loxodromic;
        \item \label{pp quotient is partially periodic} for every element $\bar g$ of $\bar G$ of finite order, either $\bar{g}^{2n}=1$ for some $n \in \expfamily$ or $\bar g$ is the image of an elliptic element of $G$. Moreover, for every element $h \in Q$, either its image $\bar h$ satisfies $\bar{h}^{2n}=1$ for some $n \in \expfamily$ or it is identified with the image of an elliptic element of $G$;
        \item \label{pp quotient is infinite} there are infinitely many elements in $\bar G$ that are not the image of an elliptic element of $G$.
    \end{enumerate}
\end{theorem}

\begin{proof}
    As stated above, we will prove Theorem \ref{limit step quotient} by iteratively applying Proposition \ref{induction step quotient}. More concretely, we will put $G_0=G$ and then produce a sequence $(G_i)_{i \in \mathbb N}$ where we obtain $G_{i+1}$ from $G_i$ by adding new relations of the form $h^{n_h}$ for $h$ a loxodromic element of the image of $Q$ in $G_i$ that is primitive in $G_i$ and $n_h \in \expfamily$ (as prescribed by Proposition \ref{induction step quotient}). The group $\bar G$ will be the limit of this sequence.

    Even though most of the proof is similar to those of Theorem 6.9 in \cite{coulon_1} and Theorem 4.28 in \cite{amelio_andre_tent}, we include a proof here for the sake of completeness.

    Keeping the same notation for the constants $L_S$, $\rho_0$ and $\delta_1$ as in Proposition \ref{induction step quotient}, write $\tau_0= \tau(G,X)$. Let $G_0 =G$, $Q_0 = Q$ and $X_0 = \lambda ' X$ where $\lambda '$ is the greatest real number such that $\delta ' = \lambda \delta \leq \delta_1$ and $\Omega(G,X_0) \leq 6\pi \tau_0 \sinh(2L_S \delta_1)$. Therefore, by Remark \ref{invariants in rescaled space}, we have that $X_0$ is $\delta '$-hyperbolic, and that the action of $G_0$ on $X_0$ is tame, acylindrical, non-elementary and satisfies $\tau(G_0,X_0)=\tau(G,X)$, $\rinj (Q_0,X_0)= \lambda \rinj (Q,X)$ and $\Omega(G_0,X_0)= \lambda \Omega(G,X)$. Moreover, the family $Q$ is strongly stable. We define the critical exponent $n_1$ as the smallest positive integer such that \[ \rinj (Q_0,X_0) \geq 2\delta_1 \sqrt{\frac{L_S \sinh(\rho_0)}{n_1 \sinh(26 \delta_1)}} \] and \[ 1 > \frac{1}{\sqrt{n_1}} \left( \frac{4\pi}{\delta_1} \sqrt{\frac{\sinh (\rho_0) \sinh (26 \delta_1)}{L_S}} \right). \]Notice that, indeed, $n_1$ depends only on $\delta$, $\tau(G,X)$, $\Omega(G,X)$ and $\rinj (Q,X)$. Write $c_1$ for the constant appearing in the first of these equations and $c_2$ for the constant appearing in the second one. Fix now a family $\expfamily =\{ n^{(m)} \, : \, 1 \leq m \leq l' \}$ of odd integers such that $n^{(m)} \geq n_1$ for all $m \in \{ 1, \dots , l' \}$. Denote by $P_0$ a maximal set of primitive loxodromic elements $h$ of $Q_0$ with translation length at most $L_S \delta_1$ that are non-equivalent under the equivalence relation defined in the statement of Proposition \ref{induction step quotient}, and we set $P'_0= \bigcup \limits_{g \in G_0}g^{-1}P_0 g$. Put $\mathcal{Q}_0 = \{ (\langle h^{n_h} \rangle , Y_{h^{n_h}}) : h \in P'_0 \}$ with $n_h \in \expfamily$ and with $n_h=n_{h'}$ if $h$ and $h'$ are conjugate. Since $Q$ is a strongly stable subset of $G$, then $Q_0$ is stable with respect to $\mathcal{Q}_0$, so we have that $G_0$, $X_0$ and $Q_0$ satisfy the hypotheses of Proposition \ref{induction step quotient} for $\tau_0$, the exponent $n_1$ and the family $\mathcal{Q}_0$.

    Now, suppose that we have constructed a quotient group $G_i$ of $G$ acting on a hyperbolic space $X_i$. Let $Q_i$ be the image of $Q_0$ in $G_i$, and suppose that $G_i$, $X_i$ and $Q_i$ satisfy the assumptions of Proposition \ref{induction step quotient} for $\tau_0$ and exponent $n_1$, except that \textit{a priori} $Q_i$ is not required to be stable with respect to some family. Suppose furthermore that for every $1 \leq j \leq i$, the group $G_j$ has been obtained from $G_{j-1}$ as a quotient by a subgroup $K'_{j-1}=\langle h'^{n_h} : h' \in P'_{j-1} \rangle^{G_{j-1}}$ where $P'_{j-1}$ is a conjugation invariant set of primitive loxodromic elements in the image $Q_{j-1}$ of $Q$ in $G_{j-1}$ and $n_{h'} \in \expfamily$ for all $h' \in P'_{j-1}$. Denote by $P'_i$ a maximal set of primitive loxodromic elements $h$ of the image $Q_i$ of $Q$ in $G_i$ such that $[h]_{X_i} \leq L_S \delta_1$ and such that no two elements are equivalent under the equivalence relation defined in the statement of Proposition \ref{induction step quotient}. Set $P'_i= \bigcup \limits_{g \in G_i}g^{-1}P_i g$. Put $\mathcal{Q}_i = \{ (\langle h^{n_h} \rangle , Y_{h^{n_h}}) : h \in P'_i \}$ with $n_h \in \expfamily$ and with $n_h = n_{h'}$ if $h$ and $h'$ are conjugate. Since $Q$ is a strongly stable family of elements of $G$, we have that $Q_i$ is stable with respect to $\mathcal{Q}_i$, so indeed $G_i$, $X_i$ and $Q_i$ satisfy the hypotheses of Proposition \ref{induction step quotient} for $\tau_0$, $n_1$ and $\mathcal{Q}_i$. Let $K_{i}=\langle h^n : h \in P_{i} \rangle^{G_{i}}$. By Proposition \ref{induction step quotient}, the quotient $G_{i+1}=G_i / K_i$ acts on a hyperbolic space $X_{i+1}$, and if we write $Q_{i+1}$ for the image of $Q$ on $G_{i+1}$, then $G_{i+1}$, $X_{i+1}$ and $Q_{i+1}$ satisfy themselves the hypotheses of Proposition \ref{induction step quotient} for $\tau_0$ and $n_1$, except that \textit{a priori} $Q_{i+1}$ is not stable with respect to some family.

    Therefore, the sequence $(G_i)_{i \in \mathbb{N}}$ is well-defined. Write $\bar G$ for the limit of this sequence. We have that $\bar G$ is indeed a quotient of $G$ by a normal subgroup $K$. We claim that this group satisfies the announced properties.

    \textit{Claim 1: if $E$ is an elliptic subgroup of $G$, the projection $G \longrightarrow \bar G$ induces an isomorphism from $E$ onto its image.} Indeed, since every $G_i$ is obtained from $G_{i-1}$ by applying Proposition \ref{induction step quotient}, an inductive argument shows that the projection $G \longrightarrow G_i$ induces an isomorphism from $E$ onto its image, which is itself elliptic. Therefore, the projection onto the limit also induces the claimed isomorphism.

    \textit{Claim 2: every element of $K$ is loxodromic.} Let $g$ be a non-trivial element of $K$, and suppose towards a contradiction that $g$ is elliptic. Then, by Claim 1 applied to the elliptic subgroup $\langle g \rangle$, the projection $G \longrightarrow \bar G$ induces an isomorphism onto its image, so in particular, the image $\bar g$ of $g$ in $\bar G$ would be non-trivial, and we arrive at a contradiction.

    \textit{Claim 3: for every element $\bar g$ of $\bar G$ of finite order, either $\bar{g}^{2n}=1$ for some $n \in \expfamily$ or $\bar g$ is the image of an elliptic element of $G$. Moreover, for every element $h \in Q$, either its image $\bar h$ satisfies $\bar{h}^{2n}=1$ for some $n \in \expfamily$ or it is identified with the image of an elliptic element of $G$.} An inductive argument using Proposition \ref{induction step quotient} shows that if $g'$ is an elliptic element of $G_i$, then either $g'^{2n}=1$ for some $n \in \expfamily$ or it is the image of an elliptic element of $G$. Let now $\bar g$ be an element of $\bar G$ of finite order that is not the image of an elliptic element of $G$. Denote by $g$ a preimage of $\bar g$ in $G$ and by $g_i$ the image of $g$ in $G_i$. Notice that $g$ is loxodromic, and that if $g_i$ was infinite for all $i \in \mathbb N$, then $\bar g$ would be of infinite order itself. Thus, there exists $j \in \mathbb N$ such that $g_j$ is loxodromic and $g_{j+1}$ is of finite order, and the first part of our claim follows from the third property of the projection map from Proposition \ref{induction step quotient}.

    For the second part of the claim, with the notation of the previous paragraph, assume that $\bar g$ is in the image $\bar Q$ of $Q$ in $\bar G$. By the construction of the sequence $(G_i)_{i \in \mathbb N}$, we have that $[g_i]_{X_i} \leq (c_2)^i [g]_X$. Therefore, there is $j \in \mathbb N$ such that $[g_j]_{X_j} < c_1 \leq \rinj (Q_j , X_j)$. In particular, since $g_j \in Q_j$, it is elliptic, and again the conclusion follows from the third property of the projection map from Proposition \ref{induction step quotient}.

    \textit{Claim 4: there are infinitely many elements in $\bar G$ that are not the image of an elliptic element of $G$.} Notice that it is enough to prove the claim for $Q$ being the set of all loxodromic elements of $G$. For such a case, denote by $D$ the set of loxodromic elements of $G$ that are not identified with an elliptic element of $G$ in $\bar G$, and assume that its image in $\bar G$ is finite. In particular, there exists a finite set $S$ of $G$ such that $D \subseteq S \cdot K$. An argument analogous to the one at the end of the previous claim applied to the finite set $S$ gives that there is some $j \in \mathbb N$ such that all elements of the image $S_j$ of $S$ in $G_j$ are elliptic. Fix now a preimage $g \in D$ of an element $g_j$ of $P'_j$. By construction, $g_j$ is loxodromic with $[g_j]_{X_j} \leq L_S \delta_1 < \rho_0 /100$ and its image $g_{j+1}$ in $G_{j+1}$ is elliptic. Furthermore, since $D$ is contained in $S \cdot K$, there is some $l > j $ such that $g_l$ is in $S_l$. Now, by the choice of the exponents, the order of $g_l$ is not a power of 2, so the element of $S_l$ with which it is identified cannot be contained in a strict reflection group at any step of the inductive process (since all the actions under consideration are tame). Thus, an inductive argument using the fourth property of the projection map from Proposition \ref{induction step quotient} shows that $g_j$ is conjugate to an element of $S_j$. However, $g_j$ is loxodromic and all elements of $S_j$ are elliptic, so we arrive at a contradiction.
\end{proof}

\textbf{Terminology.} For the remainder of this article, we will call a group obtained as a quotient of a group $G$ as provided by Theorem \ref{limit step quotient} a \textit{partial periodic quotient} (or a \textit{PP-quotient}) of $G$, and we will denote it by $\bar G$.

\begin{remark}
\label{remark further rescaling for rinj less 2}
    Notice that, up to making our critical exponent $n_1$ larger, we may take a smaller value for the rescaling parameter $\lambda '$. In particular, in Subsection \ref{subsection pp-quotient of wstprime is in wst} we will be interested in the greatest real number $\lambda ''$ that, in addition to all requirements for $\lambda '$, satisfies \[ \lambda'' 2 \leq L_S \delta_1. \] We will call $n'_1$ the critical exponent imposed by this additional assumption.
\end{remark}

\begin{remark}
\label{remark why we need tameness}
    Let us now include some discussion and a `toy example' to illustrate the way in which the tameness assumption allows us to control the small cancellation assumptions throughout the proof of Theorem \ref{limit step quotient} (or more precisely, what can go wrong when the action is not tame), and how this relates to the necessity of the assumption that the prime number $p \equiv 3 \pmod{4}$ in Theorem \ref{main theorem}.

    We consider the minimal relaxation of tameness that would still make sense if we wanted to define classes analogous to $\classwst{p,q_1,q_2}$ and $\classwstprime{p,q_1,q_2}$ for a prime $p \equiv 1 \pmod{4}$ (see further below in this remark for more details): suppose that we have a group $G$ acting acylindrically on a hyperbolic space $X$ in such a way that every finite subgroup normalized by a loxodromic element has order at most 2, and that no two involutions commute, but $G$ is allowed to have cyclic subgroups of order 4. We want to understand whether we can get a PP-quotient of $G$ analogous to the one in Theorem \ref{limit step quotient}.
    
    The relaxation of the requirement that $G$ has no subgroup of order 4 allows for one more isomorphism type in the classification from Remark \ref{virt cyc subgroups of tame action}: a loxodromic subgroup can now be isomorphic to $C_4 \ast_{C_2} C_4$. Therefore, if we take a small cancellation quotient as in Theorem \ref{small cancellation theorem} (or more concretely, as in the inductive step analogous to Proposition \ref{induction step quotient}), from Lemma \ref{maximal normal fin sbgrp of loxo sgrp sc theorem} we obtain that there may be a loxodromic subgroup of the quotient group $\hat G$ (for its action on the space $\hat X$) whose maximal normal finite subgroup is $C_4$.
    
    Now, $C_4$ admits a non-trivial automorphism of order 2. This means that we could have a loxodromic subgroup $\hat H$ of $\hat G$ whose maximal normal finite subgroup $\hat F$ is not central in $\hat H$. Even more: consider one such subgroup, and take $\hat h \in \hat H$ to be a loxodromic element. Since $C_4$ admits an automorphism of order 2, an odd power $\hat h^n$ of $\hat h$ is not guaranteed to centralize $\hat F$, and thus, the subgroup $\hat H'=\langle \hat h^n \rangle$ is not guaranteed to be normal in $\hat H$, which, as noted in Remark \ref{H is normal in Stab Y}, is key in order to have that the small cancellation conditions are satisfied. If we wanted to be able to guarantee that $\hat H'$ is normal in $\hat H$, we would need to take an \emph{even} exponent $n$. This in turn would create several complications of a different nature: first, as was explained before, controlling the algebraic structure of an even exponent quotient of a negatively curved group is a much more delicate matter than doing so for an odd exponent. Also, we are creating in this process new involutions, over which we have no control: for example, they need not be conjugate to the involutions coming from isomorphic images of subgroups of order 2 in $G$.

    At this point, one may have the hope of being able to avoid, throughout the induction process, the appearance of loxodromic elements inverting a generator of a subgroup isomorphic to $C_4$ by some other method. Let us show why this strategy is bound to fail in our setting.

    The key point to be able to assert that an action (on the Bass-Serre tree) of an HNN-extension of a group in class $\classwst{p,q_1,q_2}$ is tame is that the primer number $p$ is congruent to 3 $\pmod{4}$. In this case, $\agl$ has no subgroups of order 4 (since, in this case, 4 does not divide the order of this group, which is $p(p-1)$). On the other hand, if we want to consider a prime number $p \equiv 1 \pmod{4}$, then $\agl$ has, indeed, cyclic subgroups of order 4, since the multiplicative group of $\mathbb{F}_p$ is cyclic of order $p-1$. In particular, when taking HNN-extensions as in the construction in Section \ref{section outline proof}, we will find ourselves conjugating an involution contained in a subgroup isomorphic to $C_4$ to other involutions. This will create, in fact, elementary subgroups isomorphic to $C_4 \ast _{C_2} C_4$. But even more so: it will create longer chains of (non-elementary) subgroups isomorphic to several copies of $C_4$ amalgamated over the same involution.

    We now show how this creates loxodromic elements inverting the generator of a subgroup isomorphic to $C_4$. Let $G$ be a group isomorphic to four copies of $C_4$ amalgamated over the same copy of $C_2$, that is, a subgroup with the following presentation: \[ \langle s,t,u,v \lvert s^4 , t^4,u^4,v^4,s^2=t^2=u^2=v^2  \rangle. \] Consider the action of this group on the Bass-Serre tree $X$ corresponding to the previously mentioned splitting. Let $\hat G$ be the quotient of $G$ by the normal closure of the set (of loxodromic elements) $\{ (st)^{n_1},(tu)^{n_2}, (uv)^{n_3}, (vs)^{n_4} \}$. If we pick the $n_i$'s to be large enough, this is indeed a small cancellation quotient of $G$, so by Theorem \ref{small cancellation theorem}, the images $\hat s$, $\hat t$, $\hat u$ and $\hat v$ of the generators still have order 4, and there is a hyperbolic space $\hat X$ on which $\hat G$ acts. Now, if all the $n_i$'s are odd, consider the element $\hat g=(\hat t \hat s)^{\frac{n_1 -1}{2}}(\hat u \hat t)^{-\frac{n_2 -1}{2}}(\hat v \hat u)^{\frac{n_3 -1}{2}}(\hat s \hat v)^{-\frac{n_4 -1}{2}}$. This element will be loxodromic (by its action on the space $\hat X$). Now, a simple calculation (for example, using van Kampen diagrams) will show that $\hat g^{-1} \hat s \hat g=\hat s^{-1}$.
\end{remark}

\section{Proof of Proposition \ref{proposition pp-quotient of wstprime}}
\label{section pp-quotient of wstprime}

The goal of this section is to provide a proof of Proposition \ref{proposition pp-quotient of wstprime}. For the sake of completeness, we now state this result again.

\begin{proposition}
\label{restatement proposition pp-quotient of wstprime}
    Let $(G,X)$ be a pair in class $\classwstprime{p,q_1,q_2}$ for some prime $p$ and odd numbers $q_1$ and $q_2$. Then, $G$ has a quotient group $\bar G$ that is in class $\classwst{p,q_1,q_2}$ with the following additional properties.
    \begin{enumerate}[label=(\arabic*)]
        \item \label{no new invol pp-quotient of wstprime restatement} Every involution of $\bar G$ is the image of an involution of $G$.
        \item \label{elliptics embed pp-quotient of wstprime restatement} If $F$ is an elliptic subgroup of $G$ (for its action on $X$), then the projection map $G \twoheadrightarrow \bar G$ induces an isomorphism from $F$ onto its image.
        \item \label{image of paff is paff and of pmin is pmin restatement} The image of a pair $(r,s) \in \invset{G}{(2)}$ of $p$-affine (respectively, of $p$-minimal) type is again of $p$-affine (respectively, of $p$-minimal) type. Moreover, a pair $(\bar r , \bar s) \in \invset{\bar G}{(2)}$ is of $p$-affine type if and only if every preimage of the pair in $\invset{G}{(2)}$ is of $p$-affine type.
        \item \label{cent of element of order at least 3 in pp-quotient of wstprime restatement} Let $g \in G$ be an element of finite order $\geq 3$, and let $\bar g$ be its image on $\bar G$. Then, the projection map $G \twoheadrightarrow \bar G$ induces an isomorphism from $N_G (\langle g \rangle)$ onto $N_{\bar G}(\langle \bar g \rangle)$ (and thus also from $\Cen_G (g)$ onto $\Cen_{\bar G}( \bar g)$).
        \item \label{trans of inf order implies non-commuting translations restatement} If $G$ contains a translation of infinite order and translation length at most 2, then $\bar G$ contains non-commuting translations.
        \item \label{pp-quotient of wstprime elem not cent an inv restatement} If $G$ contains an element of infinite order that is not a translation, that has translation length 1 and that centralizes no involution, then $\bar G$ contains an element of order $q_1$ that is not a translation and centralizes no involution.
        \item \label{pp-quotient of wstprime elem cent an inv restatement} If $G$ contains an element of infinite order (which is not a translation), that has translation length 1 and that centralizes an involution, then $\bar G$ contains an element of order $q_2$ which is not a translation and centralizes an involution.
    \end{enumerate}
\end{proposition}

This result will be obtained by applying Theorem \ref{limit step quotient} to a pair $(G,X)$ in class $\classwstprime{p,q_1,q_2}$, where the additional properties will be obtained by a refined study of the quotients provided by Proposition \ref{induction step quotient} in this particular setting.

We start by defining an auxiliary class denoted by $\classwstprimezero{p,q_1,q_2}$, analogous to class $\classcprimezero$ from \cite[Section 5]{amelio_andre_tent}. The purpose of this class is the following: as was stated before, we will deduce further properties of some particular PP-quotients of a group $G$ from a pair $(G,X)$ in class $\classwstprime{p,q_1,q_2}$ by further studying the inductive step quotients in the construction of a PP-quotient (the SC-quotients of one such group, as in Proposition \ref{induction step quotient}). However, class $\classwstprime{p,q_1,q_2}$ is not stable under taking SC-quotients, in the sense that a pair $(\ghat , \xhat)$ obtained from a pair $(G,X)$ in class $\classwstprime{p,q_1,q_2}$ will not be in this class (for example, the space $\hat X$ will not be a tree and the parameters will not be bounded as in Definition \ref{def class wstprime} \ref{class wstprime parameters}). Instead, $\classwstprimezero{p,q_1,q_2}$ will, in fact, be stable under taking SC-quotients as in Proposition \ref{induction step quotient}.

\begin{definition}
\label{definition class wstprime0}
    Let $G$ be a group acting on a $\delta_1$-hyperbolic length space $X$ (where $\delta_1$ is the constant provided by Proposition \ref{induction step quotient}). We say that the pair is in class $\classwstprimezero{p,q_1,q_2}$ if it satisfies the conditions for class $\classwstprime{p,q_1,q_2}$ (Definition \ref{def class wstprime}) except that the space $X$ is not required to be a tree and Condition \ref{class wstprime parameters} is replaced by the following.
    \begin{enumerate}[label=(\arabic*'')]
        \setcounter{enumi}{1}
        \item \label{parameters in wstprimezero} The tuple $(G,X)$ satisfies the assumptions of Proposition \ref{induction step quotient} for $Q=G$, $\tau_0=5$ and family of integers $\expfamily = \{p,q_1,q_2\}$.
    \end{enumerate}
\end{definition}

\subsection{Stability of class \texorpdfstring{$\classwstprimezero{p,q_1,q_2}$}{wstprimezero} under SC-quotients}
\label{subsection classwstprimezero SC-quotient}

The goal of this subsection is to prove that the newly defined class $\classwstprimezero{p,q_1,q_2}$ is stable under SC-quotients. More concretely, we will prove the following result, whose purpose is to serve as the inductive step in proving Proposition \ref{restatement proposition pp-quotient of wstprime}.

\begin{proposition}
\label{proposition SC-quotient of wstprimezero}
    Let $(G,X)$ be a pair in class $\classwstprimezero{p,q_1,q_2}$. Let $(\hat G , \hat X)$ be the pair obtained from $(G,X)$ by applying Proposition \ref{induction step quotient} with $\expfamily = \{ p,q_1,q_2\}$, $Q$ the family of all loxodromic elements of $G$ and by picking the exponent $n_h$ for a primitive loxodromic element $h$ of $G$ as follows.
    \begin{itemize}
        \item If the maximal loxodromic subgroup containing $h$ is $D_\infty$, then $n_h = p$.
        \item If the maximal loxodromic subgroup containing $h$ is $\mathbb Z \times C_2$, then $n_h=q_2$.
        \item If the maximal loxodromic subgroup containing $h$ is $\mathbb Z $, then $n_h=q_1$.
    \end{itemize}
    Then, the pair $(\hat G , \hat X)$ is in class $\classwstprimezero{p,q_1,q_2}$. Furthermore, the quotient map $G \twoheadrightarrow \hat G$ has the following properties.
    \begin{enumerate}[label=(\arabic*)]
        \item \label{no new inv sc-quotient wstprimezer0} Every involution of $\hat G$ is the image of an involution of $G$.
        \item \label{quotient map ind iso of elliptic sc-quotient wstprimezer0} If $F$ is an elliptic subgroup of $G$, then the projection map $G \twoheadrightarrow \hat G$ induces an isomorphism from $F$ onto its image.
        \item \label{image of paff is paff sc-quotient wstprimezero} The image of a pair $(r,s) \in \invset{G}{(2)}$ of $p$-affine (respectively, of $p$-minimal) type is again of $p$-affine (respectively, of $p$-minimal) type.
        \item \label{cent of element of order at least 3 in sc-quotient of wstprimezero} Let $F$ be a subgroup of $G$ of finite order $\geq 3$. Then, the projection map $G \twoheadrightarrow \hat G$ induces an isomorphism from $N_G (F)$ onto $N_{\hat G}(\hat F)$.
    \end{enumerate}
\end{proposition}

For the remainder of this subsection, we fix a pair $(G,X)$ in class $\classwstprimezero{p,q_1,q_2}$ and we let $( \hat G, \hat X)$ be the pair obtained as in the statement of Proposition \ref{proposition SC-quotient of wstprimezero}. The remainder of this subsection is devoted to prove that $(\hat G , \hat X)$ indeed has the announced properties.

The following useful result allows us to classify apex stabilizers in the quotient space $\hat X$.

\begin{lemma}
\label{apex stab in sc-quotient of wstprimezero}
    Let $\hat v \in \apices$. The subgroup $\stab (\hat v)$ is isomorphic to one of the following groups:
    \begin{itemize}
        \item $C_{q_1}$,
        \item $C_{q_2} \times C_2$; or
        \item $D_p$.
    \end{itemize}
\end{lemma}

\begin{proof}
    The desired result follows directly from Lemma \ref{apex stab for tame action} and our choice of exponents in Proposition \ref{proposition SC-quotient of wstprimezero}.
\end{proof}

\begin{remark}
    Notice that, since $q_2$ is an odd integer, we have that $C_{q_2} \times C_2 \cong C_{2q_2}$. However, we choose to keep the notation $C_{q_2} \times C_2$ for these apex stabilizers, since it makes more explicit the geometry of the action of this subgroup on the points of $\hat X$ that are close to $\hat v$.
\end{remark}

The next result is an immediate consequence of Proposition \ref{induction step quotient}.

\begin{lemma}
    The pair $(\hat G , \hat X)$ satisfies Condition \ref{class wstprime action non-elementary} of Definition \ref{def class wstprime} and Condition \ref{parameters in wstprimezero} of Definition \ref{definition class wstprime0}.
\end{lemma}

\begin{lemma}
\label{invol in sc-quotient of wstprimezero are images of inv}
    Property \ref{no new inv sc-quotient wstprimezer0} of Proposition \ref{proposition SC-quotient of wstprimezero} holds: let $\hat r$ be an involution of $\hat G$. Then, $\hat r$ has a preimage in $G$ that is an involution.
\end{lemma}

\begin{proof}
    This is a consequence of Lemma \ref{characterization of lifting elliptics sc theorem} and the fact that $\langle \hat r \rangle$ cannot contain a strict rotation (since they all have an order that has a proper odd divisor).
\end{proof}

The next lemma is again an immediate consequence of Proposition \ref{induction step quotient}.

\begin{lemma}
\label{sc-quotient of wstprimezero image of elliptic}
    Property \ref{quotient map ind iso of elliptic sc-quotient wstprimezer0} of Proposition \ref{proposition SC-quotient of wstprimezero} holds: if $F$ is an elliptic subgroup of $G$, then the quotient map $G \twoheadrightarrow \hat G$ induces an isomorphism from $F$ onto its image.
\end{lemma}

Now we prove that the pair $(\hat G , \hat X)$ satisfies Condition \ref{class wstprime infinte order implies loxo} of Definition \ref{def class wstprime}.

\begin{lemma}
\label{sc-quotient of wstprimezero elem on inf order loxo}
    Every element of infinite order of $\hat G$ is loxodromic by its action on $\hat X$.
\end{lemma}

\begin{proof}
    Let $\hat g$ be an element of infinite order of $\hat G$. Write $\hat F = \langle \hat g \rangle$. Assume towards a contradiction that $\hat F$ is elliptic.

    Since $G$ is in class $\classwstprimezero{p,q_1,q_2}$, it contains no elliptic element of infinite order, and therefore $\hat F$ cannot lift.

    Thus, by Lemma \ref{characterization of lifting elliptics sc theorem} we get that $\hat F$ must be contained in $\stab (\hat v)$ for some $\hat v \in \apices$. However, from Lemma \ref{apex stab in sc-quotient of wstprimezero} we know that no such subgroup contains an element of infinite order, and we arrive at a contradiction.
\end{proof}

\begin{lemma}
\label{normalizer of elliptic geq 3 in sc-quotient of wstprimezero}
    The quotient map $G \twoheadrightarrow \hat G$ satisfies Property \ref{cent of element of order at least 3 in sc-quotient of wstprimezero} of Proposition \ref{proposition SC-quotient of wstprimezero}: let $F$ be a finite subgroup of $G$ of order at least 3, and let $\hat F$ be its image on $\hat G$. Then the quotient map induces an isomorphism from $N_G(F)$ onto $N_{\hat G}(\hat F)$
\end{lemma}

\begin{proof}
    By Lemma \ref{cent of element of order >3 in wstprime is elliptic}, $N_G(F)$ is elliptic, so the quotient map induces an isomorphism from $N_G(F)$ onto its image $E$, which is contained in $N_{\hat G}(\hat F)$.
    Furthermore, $\hat F$ is not contained in a strict reflection group at some apex (since all of them are of order at most 2), and therefore Lemma \ref{sc theorem elliptics with conjugate image} applied to $F_1=F_2=F$ gives that the preimage of $N_{\hat G}(\hat F)$ is a subgroup of $N_G(F)$. From this, the desired conclusion follows.
\end{proof}

Lemmas \ref{sc-quotient of wstprimezero has paff or pmin} to \ref{pp quotient of wstprimezero of q1q2 bdd exp} prove that $(\hat G , \hat X)$ is a weakly sharply 2-transitive group of characteristic $p$ of $(q_1,q_2)$-almost bounded exponent.

\begin{lemma}
\label{sc-quotient of wstprimezero has paff or pmin}
    The pair $(\hat G , \hat X)$ satisfies Condition \ref{def class C pairs of minimal or affine type} of Definition \ref{definition weakly s2t}: every translation is either of order $p$ or of infinite order, and every pair $(\hat r, \hat s) \in \mathcal{I}_{\hat G}^{(2)}$ such that $\hat r \hat s$ is of order $p$ is either of $p$-minimal type or of $p$-affine type.

    Moreover, the following properties hold.

    \begin{itemize}
        \item The image of a pair $(r,s) \in \invset{G}{(2)}$ of $p$-affine (respectively, of $p$-minimal) type is of $p$-affine (respectively, of $p$-minimal) type (that is, the quotient map satisfies Property \ref{image of paff is paff sc-quotient wstprimezero} of Proposition \ref{proposition SC-quotient of wstprimezero}).
        \item A pair $(\hat r , \hat s) \in \invset{\hat G}{(2)}$ is of $p$-affine type if and only if it has a preimage in $G$ of $p$-affine type, and every such preimage is of $p$-affine type.
    \end{itemize}
\end{lemma}

\begin{proof}
    Let $(\hat r, \hat s) \in \mathcal{I}_{\hat G}^{(2)}$ be such that $\hat r \hat s$ is of finite order. The subgroup $\hat F=\langle \hat r , \hat s \rangle$ is a finite dihedral group, and thus it is elliptic by its action on $\hat X$. If $\hat F$ lifts, then $\hat r \hat s$ is of order $p$ (since $G$ is weakly sharply 2-transitive of characteristic $p$). If $\hat F$ does not lift, then by Lemma \ref{characterization of lifting elliptics sc theorem} it is contained in $\stab (\hat v)$ for some $\hat v \in \apices$. Now, by Lemma \ref{apex stab in sc-quotient of wstprimezero}, the only isomorphism class of apex stabilizers that contains more than one involution is $D_p$, and since any two distinct involutions of $D_p$ generate the whole subgroup, then $\hat r \hat s$ is indeed of order $p$.

    Notice that by Lemma \ref{sc-quotient of wstprimezero image of elliptic} the image of a pair $(r,s) \in \invset{G}{(2)}$ such that $rs$ has order $p$ is a pair $(\hat r , \hat s)$ in $\invset{\hat G}{(2)}$ since $r$ and $ s$ generate a finite subgroup.

    Now, let $(\hat r , \hat s) \in \invset{\hat G}{(2)}$ be such that $\hat r \hat s$ has order $p$. Notice first that, since $\langle \hat r \hat s \rangle$ is characteristic in $D_{\hat r , \hat s}$, then $D_{\hat r , \hat s} \leq N_{\hat G}(D_{\hat r , \hat s}) \leq N_{\hat G}(\langle \hat r  \hat s \rangle) $. Now, by Lemma \ref{cent of element of order >3 in wstprime is elliptic} we have that $N_{\hat G}(\langle \hat r  \hat s \rangle)$ is elliptic.

    We claim that $N_{\hat G}(\langle \hat r  \hat s \rangle)$ lifts if and only if so does $D_{\hat r , \hat s}$: indeed, if $N_{\hat G}(\langle \hat r  \hat s \rangle)$ does not lift, then by Lemma \ref{characterization of lifting elliptics sc theorem} it is a subgroup of $\stab (\hat v)$ for some $\hat v \in \apices$ containing more than one involution, and the only possible isomorphism class for such a group is $D_p$. Thus, $N_{\hat G}(\langle \hat r  \hat s \rangle)=D_{\hat r , \hat s}$. In particular, a pair of distinct involutions of $\hat G$ generating a non-lifting dihedral subgroup of order $2p$ is of $p$-minimal type.

    Assume now that $N_{\hat G}(\langle \hat r  \hat s \rangle)$ lifts to an elliptic subgroup $E$ of $G$ (and therefore, $D_{\hat r , \hat s}$ also lifts). By Lemma \ref{normalizer of elliptic geq 3 in sc-quotient of wstprimezero}, the quotient map induces an isomorphism from $N_G (\langle rs \rangle)$ onto $N_{\hat G} (\langle \hat r \hat s \rangle)$.

    Now, since $(G,X)$ is in class $\classwstprimezero{p,q_1,q_2}$, $(r,s)$ is either of $p$-minimal or of $p$-affine type. We immediately obtain from the previous paragraph that $(\hat r , \hat s)$ is of $p$-minimal type if and only if so is $(r,s)$. If $(r,s)$ is of $p$-affine type, then $D_{r,s}$ is contained in a finite subgroup $H$ isomorphic to $\agl$. Since $H$ is necessarily elliptic, then the quotient map induces an isomorphism from $H$ onto its image, and thus $(\hat r , \hat s)$ is of $p$-affine type. This proves that indeed every pair $(\hat r , \hat s) \in \invset{\hat G}{(2)}$ such that $\hat r \hat s$ has order $p$ is either of $p$-affine or of $p$-minimal type.

    In a similar way to the previous paragraph we can conclude that the image of any pair $(r,s) \in \invset{G}{(2)}$ of $p$-affine type is a pair of $p$-affine type.
    
    Now, let $(r,s)$ be a pair of $p$-minimal type, and assume towards a contradiction that the image $(\hat r , \hat s)$ is of $p$-affine type. Let $\hat H$ be a subgroup isomorphic to $\agl$ containing $\hat r$ and $\hat s$. Since there is no $\hat v \in \apices$ such that $\hat H$ is isomorphic to a subgroup of $\stab (\hat v)$, then $\hat H$ lifts to a subgroup $H$. Let $r'$ and $s'$ be the preimages of $\hat r$ and $\hat s$ in $H$. In particular, $D_{r',s'}$ is a proper subgroup pf $N_G(\langle r's' \rangle)$. Now, $rs$ and $r's'$ not contained in a loxodromic subgroup of $G$ of dihedral type (since all of these are isomorphic to $D_\infty$ and thus contain no elements of odd order). Thus, by the fourth property of the quotient map from Proposition \ref{induction step quotient}, $rs$ and $r's'$ are conjugate in $G$ (as preimages of the same element of $\hat G$), so $N_G(rs)$ contains an element not in $D_{r,s}$, contradicting the assumption that $(r,s)$ is of $p$-minimal type. Therefore, $(\hat r , \hat s)$ is itself of $p$-minimal type.
    %
\end{proof}

\begin{lemma}
    The group $\hat G$ satisfies Condition \ref{def class C G trans on affine trans} of Definition \ref{definition weakly s2t}: the set of pairs $(\hat r , \hat s) \in \invset{\hat G}{(2)}$ of $p$-affine type is non-empty and $\hat G$ acts transitively on it by conjugation.
\end{lemma}

\begin{proof}
    Since $(G,X)$ is in class $\classwstprimezero{p,q_1,q_2}$, we have that there is a pair $(r,s) \in \invset{G}{(2)}$ of $p$-affine type. Let $H$ be the subgroup of $G$ isomorphic to $\agl$ containing $r$ and $s$. Since this subgroup is finite, we have by Lemma \ref{sc-quotient of wstprimezero image of elliptic}, the quotient map induces an isomorphism from $H$ onto its image, and thus the image $(\hat r , \hat s)$ of $(r,s)$ is a pair of $p$-affine type.

    Now, let $(\hat r , \hat s)$ and $(\hat r ' , \hat s ')$ be pairs of $\invset{\hat G}{(2)}$ of $p$-affine type. By Lemma \ref{sc-quotient of wstprimezero has paff or pmin} the pairs have preimages $(r,s)$ and $(r',s')$ (respectively) in $\invset{G}{(2)}$, both of $p$-affine type. The group $G$ is weakly sharply 2-transitive of characteristic $p$, so the pairs $(r,s)$ and $(r',s')$ are conjugate. Therefore, so are $(\hat r , \hat s)$ and $(\hat r' , \hat s ')$. Thus, $\hat G$ acts transitively by conjugation on the pairs of $p$-affine type of $\invset{\hat G}{(2)}$.
\end{proof}

\begin{lemma}
    The group $\hat G$ satisfies Condition \ref{def class C cent of trans is cyclic} of Definition \ref{definition weakly s2t}: for every pair $(\hat r, \hat s) \in \mathcal{I}_{\hat G}^{(2)}$ the subgroup $\text{Cen}(\hat r \hat s)$ is cyclic and generated by a translation.
\end{lemma}

\begin{proof}
    Let $(\hat r, \hat s) $ be a pair in $ \mathcal{I}_{\hat G}^{(2)}$.

    If $\hat r  \hat s$ is of infinite order, by Lemma \ref{sc-quotient of wstprimezero elem on inf order loxo} it is loxodromic. Any element centralizing $\hat r  \hat s$ is contained in the maximal loxodromic subgroup containing $\hat r  \hat s$ (and this subgroup itself contains $\hat r$ and $\hat s$). By the classification of loxodromic subgroups of a tame action (Remark \ref{virt cyc subgroups of tame action}) the only isomorphism type of one such group containing more than one involution is $D_\infty$, and in this group, the centralizer of a translation is cyclic and generated by a translation.

    Now, let $\hat r  \hat s$ be of order $p$. By Lemma \ref{cent of element of order >3 in wstprime is elliptic} we have that $N_{\hat G}(\langle \hat r \hat s \rangle)$ is elliptic. Notice that $\Cen_{\hat G}(\hat r \hat s) \leq N_{\hat G}(\langle \hat r \hat s \rangle)$.

    If $N_{\hat G}(\langle \hat r \hat s \rangle)$ lifts, then $\Cen (\hat r \hat s)$ lifts to a subgroup of $\Cen_G (rs)$ (where $r$ and $s$ are the preimages of $\hat r$ and $\hat s$ respectively in the lift of $N_{\hat G}(\langle \hat r \hat s \rangle)$), which is itself cyclic and generated by a translation (since $G$ is weakly sharply 2-transitive of characteristic $p$). Since a power of a translation is still a translation, we get that $\Cen_{\hat G}(\hat r \hat s)$ is indeed cyclic and generated by a translation.

    If $N_{\hat G}(\langle \hat r \hat s \rangle)$ does not lift, then there is some $\hat v \in \apices$ such that $N_{\hat G}(\langle \hat r \hat s \rangle) \leq \stab (\hat v)$. Now, $N_{\hat G}(\langle \hat r \hat s \rangle)$  contains two distinct involutions $\hat r$ and $\hat s$. By Remark \ref{apex stab in sc-quotient of wstprimezero} the only possible isomorphism type for $\stab (\hat v)$ containing more than one involution is $D_p$, and in this subgroup centralizers of translations are indeed cyclic and generated by a translation.
\end{proof}

We finalize the proof of Proposition \ref{proposition SC-quotient of wstprimezero} and this subsection with the following lemma.

\begin{lemma}
\label{pp quotient of wstprimezero of q1q2 bdd exp}
    The group $\hat G$ is of $(q_1,q_2)$-almost bounded exponent, i.e., it satisfies Condition \ref{def class C almost bounded exponent} of Definition \ref{definition weakly s2t}: for every subgroup $\hat E$ of finite order of $\hat G$, either $\hat E$ is contained in a subgroup of $\hat G$ that embeds into $\agl$ or $\hat E$ falls into one of the following cases.
        \begin{enumerate}
            \item The subgroup $\hat E$ is contained in a subgroup isomorphic to $C_{q_1}$ and no non-trivial element of $\hat E$ centralizes an involution.
            \item The subgroup $\hat E$ is contained in a subgroup isomorphic to $C_{2q_2}$ (and thus every element of $\hat E$ centralizes an involution).
        \end{enumerate}
\end{lemma}

\begin{proof}
    Let $\hat E$ be a subgroup of $\hat G$ of finite order. We consider two cases:
    
    \textbf{Case 1:} the subgroup $\hat E$ centralizes an involution $\hat r$. Write $\hat F = \langle \hat E , \hat r \rangle$. This subgroup is finite and therefore elliptic by its action on $\hat X$. We consider two further subcases.

    \textit{Case 1a:} the subgroup $\hat F$ lifts. Write $F$ (respectively, $E$ and $r$) for the lift of $\hat F$ (respectively, of $\hat E$ and $\hat r$). Notice that the subgroup $E$ centralizes the involution $r$. Since $G$ is of $(q_1,q_2)$-almost bounded exponent, we get that $E$ is contained in a subgroup $H$ that either embeds into $\agl$ or is isomorphic to $C_{2q_2}$ (and thus the same holds for $\hat E$ since by Lemma \ref{sc-quotient of wstprimezero image of elliptic} the quotient map induces an isomorphism from $H$ onto its image).

    \textit{Case 1b:} the subgroup $\hat F$ does not lift. By Lemma \ref{characterization of lifting elliptics sc theorem} there is some $\hat v \in \apices$ such that $\hat F$ is contained in $\stab (\hat v)$. By Remark \ref{apex stab in sc-quotient of wstprimezero} the only isomorphism types containing involutions are $D_p$ (which embeds into $\agl$) and $C_{2q_2}$.

    \textbf{Case 2:} the subgroup $\hat E$ centralizes no involution. We again consider two further subcases.

    \textit{Case 2a:} the subgroup $\hat E$ lifts. Write $E$ for the preimage of $\hat E$. Notice that $E$ centralizes no involution: if it did centralize an involution $r$, then $\langle E , r \rangle$ would be an elliptic subgroup of $G$ and thus by Lemma \ref{sc-quotient of wstprimezero image of elliptic} it would map isomorphically onto its image in $\hat G$, yielding that $\hat E$ centralizes an involution. Thus, since $G$ is of $(q_1,q_2)$-almost bounded exponent, we get that $E$ is contained in a subgroup $H$ that either embeds into $\agl$ or is isomorphic to $C_{q_1}$ (and thus the same holds for $\hat E$, since by Lemma \ref{sc-quotient of wstprimezero image of elliptic} the quotient map induces an isomorphism from $H$ onto its image).

    \textit{Case 2b:} the subgroup $\hat E $ does not lift. By Lemma \ref{characterization of lifting elliptics sc theorem} there is some $\hat v \in \apices$ such that $\hat E$ is contained in $\stab (\hat v)$. By Remark \ref{apex stab in sc-quotient of wstprimezero} the only isomorphism types containing elements not centralizing an involution are $D_p$ (which embeds into $\agl$) and $C_{q_1}$.
\end{proof}

\subsection{Stability of the classes under PP-quotients}
\label{subsection pp-quotient of wstprime is in wst}

In this subsection, we will prove Proposition \ref{restatement proposition pp-quotient of wstprime}. Fix for the rest of this subsection a pair $(G,X)$ in class $\classwstprime{p,q_1,q_2}$ (see Definition \ref{def class wstprime}) and set $Q$ to be the family of all loxodromic elements of $G$. We will iteratively apply the quotient provided by Proposition \ref{proposition SC-quotient of wstprimezero} to obtain a quotient group $\bar G$ as in Theorem \ref{limit step quotient} (with the smaller constant $\lambda ''$ provided by Remark \ref{remark further rescaling for rinj less 2}) and prove the extra properties that were claimed to hold. These inductive step quotients are in particular SC-quotients as in Proposition \ref{induction step quotient}, so in order to initialize the process, we will consider the pair $(\hat G_0 , \hat X_0)$ with $G_0=G$ and $X_0$ the rescaled version of $X$ given in the initializing step of Theorem \ref{limit step quotient}, so that the pair $(\hat G_0 , \hat X_0)$ is in fact in class $\classwstprimezero{p,q_1,q_2}$.

For every positive integer $m$ we will write $(\hat G _m , \hat X _m)$ for the pair obtained by applying $m$ times Proposition \ref{restatement proposition pp-quotient of wstprime}. We start with an easy remark that will allow us to lift certain equations in $\bar G$ to some intermediate step $\hat G _m$.

\begin{remark}
\label{equations of pp-quotient of wstprime lift to intermediate step}
    Let $\bar g ^{(1)}, \dots, \bar g ^{(k)}$ be elements of $\bar G$ such that $\bar g ^{(1)} \dots \bar g ^{(k)}=1$. Then, there exist $m \in \mathbb N$ and preimages $\hat g^{(i)}_{m}$ of $\bar g ^{(i)}$ for $1 \leq i \leq k$ in $\hat G _m$ such that $\hat g_m ^{(1)} \dots \hat g_m ^{(k)}=1$.
\end{remark}

Notice first that, since every elliptic element of a group in class $\classwstprime{p,q_1,q_2}$ is of finite order, then, indeed, every element of $\bar G$ is of finite order by Consequence \ref{pp quotient is partially periodic} of Theorem \ref{limit step quotient} (since we took $Q$ to be the family of all loxodromic elements of $G$).

The next result is a direct application of Remark \ref{equations of pp-quotient of wstprime lift to intermediate step}.

\begin{lemma}
\label{invol of pp-quotient of wstprime lift}
    The quotient map $G \twoheadrightarrow \bar G$ satisfies Property \ref{no new invol pp-quotient of wstprime restatement} of Proposition \ref{restatement proposition pp-quotient of wstprime}: every involution $\bar r$ of $\bar G$ is the preimage of an involution of $G$.
\end{lemma}

\begin{proof}
    By Remark \ref{equations of pp-quotient of wstprime lift to intermediate step}, there is some $m \in \mathbb N$ such that a preimage $\hat r _m$ of $\bar r$ in $\hat G _m$ is an involution. Now, the claim follows from an inductive application of Lemma \ref{invol in sc-quotient of wstprimezero are images of inv}.
\end{proof}

\begin{lemma}
\label{pp-quotient of wstprime induces iso for elliptic}
    The quotient map $G \twoheadrightarrow \bar G$ satisfies Property \ref{elliptics embed pp-quotient of wstprime restatement} of Proposition \ref{restatement proposition pp-quotient of wstprime}: let $F$ be an elliptic subgroup of $G$. Then $G \twoheadrightarrow \bar G$ induces an isomorphism from $F$ onto its image.
\end{lemma}

\begin{proof}
    This is a direct application of Property \ref{quotient map ind iso of elliptic sc-quotient wstprimezer0} of Proposition \ref{proposition SC-quotient of wstprimezero} (since then at every step of the inductive process the map $G \twoheadrightarrow \hat G_m$ induces an isomorphism from $F$ onto its image $\hat F _m$).
\end{proof}

The next lemma provides is the key point in the argument for proving Properties \ref{image of paff is paff and of pmin is pmin restatement} and \ref{cent of element of order at least 3 in pp-quotient of wstprime restatement} of Proposition \ref{restatement proposition pp-quotient of wstprime}.

\begin{lemma}
\label{norm of fin sgrp order greater 3 induces iso pp-quotient wstprime}
    Let $F$ be a finite subgroup of $G$ of order $\geq 3$. Then, the quotient map induces an isomorphism from $N_G (F)$ onto $N_{\bar G} (\bar F)$.
\end{lemma}

\begin{proof}
    An inductive application of Property \ref{cent of element of order at least 3 in sc-quotient of wstprimezero} from Proposition \ref{proposition SC-quotient of wstprimezero} gives that the quotient map $G \twoheadrightarrow \hat G _m$ induces an isomorphism from $N_G(F)$ onto $N_{\hat G _m}(\hat F _m)$. Now, let $\overline {N_G(F)}$ be the image of $N_G(F)$ in $\bar G$ (isomorphic to $N_G(F)$ by Lemmas \ref{cent of element of order >3 in wstprime is elliptic} and \ref{pp-quotient of wstprime induces iso for elliptic}), and assume that there is some element $\bar g \in N_{\bar G} (\bar F) \backslash \overline {N_G(F)}$. The subgroup $\langle \bar g , \bar F \rangle$ is finite and, as such, its multiplication table is determined by a finite number of equations involving elements of $\bar F$ and powers of $\bar g$. Thus, by Remark \ref{equations of pp-quotient of wstprime lift to intermediate step} there is some $m \in \mathbb N$ and preimages $\hat g _m$ and $\hat F _m$ of $\bar g$ and $\bar F$ respectively such that the quotient map $\hat G _m \twoheadrightarrow \bar G$ induces an isomorphism from $\langle \hat g_m , \hat F_m \rangle$ onto $\langle \bar g , \bar F \rangle$. In particular, $\hat g _m$ normalizes $\hat F_m$. Notice that $\hat g_m$ cannot be in the image of $N_G(F)$ in $\hat G_m$ (since otherwise its image $\bar g$ would be in $\overline{N_G(F)}$). Then, we arrive at a contradiction by applying Property \ref{cent of element of order at least 3 in sc-quotient of wstprimezero} of Lemma \ref{proposition SC-quotient of wstprimezero} $m$ times.
\end{proof}

\begin{lemma}
\label{pp-quotient of wstprime image of paff and of pmin is paff and pmin}
    The quotient map $G \twoheadrightarrow \bar G$ satisfies Property \ref{image of paff is paff and of pmin is pmin restatement} of Proposition \ref{restatement proposition pp-quotient of wstprime}: the image of a pair $(r,s) \in \invset{G}{(2)}$ of $p$-affine (respectively, of $p$-minimal) type is again of $p$-affine (respectively, of $p$-minimal) type. Moreover, a pair $(\bar r , \bar s) \in \invset{\bar G}{(2)}$ is of $p$-affine type if and only if every preimage of the pair in $\invset{G}{(2)}$ is of $p$-affine type.
\end{lemma}

\begin{proof}
    Let $(r,s) \in \invset{G}{(2)}$ be a pair of $p$-affine type and let $H \cong \agl$ contain $r$ and $s$. Since $H$ is finite, it is elliptic, and thus by Lemma \ref{pp-quotient of wstprime induces iso for elliptic} the quotient map induces an isomorphism from $H$ onto its image. In particular, the image $(\bar r , \bar s)$ of the pair $(r,s)$ is itself of $p$-affine type.

    Now consider a pair $(r,s) \in \invset{G}{(2)}$ of $p$-minimal type, so that $N_G(\langle rs \rangle)=D_{r,s}$. Denote by $(\bar r , \bar s)$ the image of the pair $(r,s)$ in $\bar G$. Now, Lemma \ref{norm of fin sgrp order greater 3 induces iso pp-quotient wstprime} gives that the quotient map induces an isomorphism from $N_G(\langle rs \rangle)$ onto $N_{\bar G}(\langle \overline{rs} \rangle)$. Now, $\overline{rs}= \bar r \bar s$, and we get thus that $N_{\bar G}(\langle \bar r \bar s \rangle)$ is $D_{\bar r , \bar s}$, so the pair $(\bar r , \bar s)$ is of $p$-minimal type.

    Finally, consider a pair $(\bar r , \bar s) \in \invset{\bar G}{(2)}$ of $p$-affine type, and let $(r,s)$ be a preimage of the pair in $\invset{G}{(2)}$. By the previous paragraph, we know that $(r,s)$ cannot be of $p$-minimal type. Thus, since $G$ is weakly sharply 2-transitive of characteristic $p$, $(r,s)$ is either of $p$-affine type or it generates an infinite dihedral group. Assume towards a contradiction that we are in this last case. There is some $m \in \mathbb N$ such that the images $\hat r_m$ and $\hat s_m$ of $r$ and $s$ in $\hat G _m$ generate an infinite dihedral group, but the images $\hat r_{m+1}$ and $\hat s_{m+1}$ of $r$ and $s$ in $\hat G _{m+1}$ generate a finite dihedral group isomorphic to $D_p$. By Lemma \ref{sc-quotient of wstprimezero has paff or pmin}, we have that $(\hat r _{m+1} , \hat s _{m+1})$ is of $p$-minimal type. Now, an argument exactly as in the previous paragraph yields that the image $(\bar r , \bar s)$ in $\bar G$ must be of $p$-minimal type, and we arrive at a contradiction.
\end{proof}

\begin{lemma}
\label{pp-quotient of wstprime cent of an element of order 3}
    The quotient map $G \twoheadrightarrow \bar G$ satisfies Property \ref{cent of element of order at least 3 in pp-quotient of wstprime restatement} of Proposition \ref{restatement proposition pp-quotient of wstprime}: let $g \in G$ be an element of finite order $\geq 3$, and let $\bar g$ be its image on $\bar G$. Then, the projection map $G \twoheadrightarrow \bar G$ induces an isomorphism from $N_G (\langle g \rangle)$ onto $N_{\bar G}(\langle \bar g \rangle)$ (and thus also from $\Cen_G (g)$ onto $\Cen_{\bar G}( \bar g)$).
\end{lemma}

\begin{proof}
    This is a direct consequence of Lemma \ref{norm of fin sgrp order greater 3 induces iso pp-quotient wstprime} applied to $F= \langle g \rangle$.
\end{proof}

\begin{lemma}
\label{pp-quotient of wstprime image contains non-commuting translations}
    The quotient map $G \twoheadrightarrow \bar G$ satisfies Property \ref{trans of inf order implies non-commuting translations restatement} of Proposition \ref{restatement proposition pp-quotient of wstprime}: if $G$ contains a translation of infinite order and translation length at most 2, then $\bar G$ contains non-commuting translations.
\end{lemma}

\begin{proof}
    Let $(r,s)$ and $(r',s')$ be pairs in $\invset{G}{(2)}$ such that $(r,s)$ is of $p$-affine type and such that $r's'$ is of infinite order and of translation length at most 2. Notice the following fact: by the choice of $\lambda ''$, in the rescaled space $X_0$ the loxodromic element $r's'$ has translation length $\leq L_S \delta _1$. This has the following consequence (see the proof of Theorem \ref{limit step quotient}): let $E \cong D_\infty$ be the maximal loxodromic subgroup of $G$ containing $r's'$, then the quotient map $G \twoheadrightarrow \hat G_1$ induces an epimorphism $E \twoheadrightarrow \hat E_1$, where $\hat E_1 \cong D_p$. Without loss of generality, we may assume that the images $\hat r '_1$ and $\hat s_1'$ of $r'$ and $s'$ generate $\hat E_1$. Moreover, by Lemma \ref{sc-quotient of wstprimezero has paff or pmin}, the pair $(\hat r'_1, \hat s_1')$ is of $p$-minimal type. In particular, by an analogous argument to the one in Lemma \ref{norm of fin sgrp order greater 3 induces iso pp-quotient wstprime} applied to $\langle \hat r'_1 \hat s_1' \rangle $ (starting the induction from step 1 instead of step 0) we get that the image $(\bar r' , \bar s')$ of this pair in $\bar G$ is a pair of $p$-minimal type. By Lemma \ref{pp-quotient of wstprime image of paff and of pmin is paff and pmin}, the image $(\bar r , \bar s)$ of $(r,s)$ in $\bar G$ is a pair of $p$-affine type.
    
    Assume towards a contradiction that $\bar r \bar s$ and $\bar r' \bar s'$ commute. By Remark \ref{equations of pp-quotient of wstprime lift to intermediate step}, there is $m \in \mathbb N$ such that there are preimages $\hat r_m$, $\hat s_m$, $\hat r'_m$ and $\hat s'_m$ of $r$, $s$, $r'$ and $s'$ (respectively) such that all of them are involutions, $\hat r'_m \hat s'_m$ is of order $p$ (and thus $(\hat r'_m, \hat s'_m)$ is of $p$-minimal type) and the translations $\hat r_m \hat s_m$ and $\hat r'_m \hat s'_m$ commute. Notice that $(\hat r_m, \hat s_m)$ is of $p$-affine type. Now, $\hat G_m$ is in class $\classwstprimezero{p,q_1,q_2}$, so the centralizer of a translation is cyclic and generated by a translation. Since every translation is either of infinite order or of order $p$, a translation of order $p$ always generates its own centralizer. Therefore, the fact that $\hat r_m \hat s_m$ and $\hat r'_m \hat s'_m$ centralize each other implies $\langle \hat r_m \hat s_m \rangle= \langle \hat r'_m \hat s'_m \rangle$. However, we have that $N_{\hat G_m}(\langle \hat r_m \hat s_m \rangle)$ contains a subgroup isomorphic to $\agl$, and we arrive at a contradiction with the fact that $(\hat r'_m, \hat s'_m)$ is of $p$-minimal type (since this imposes $N_{\hat G_m}(\langle \hat r'_m \hat s'_m \rangle)=D_{\hat r'_m , \hat s'_m}$).
\end{proof}

\begin{lemma}
\label{pp-quotient of wstprime contains elem q1 not centralizing inv}
    The quotient map $G \twoheadrightarrow \bar G$ satisfies Property \ref{pp-quotient of wstprime elem not cent an inv restatement} of Proposition \ref{restatement proposition pp-quotient of wstprime}: if $G$ contains an element of infinite order that is not a translation, that has translation length 1 and that centralizes no involution, then $\bar G$ contains an element of order $q_1$ that is not a translation and centralizes no involution.
\end{lemma}

\begin{proof}
    Let $g \in G$ be a loxodromic element satisfying the hypotheses of the lemma. By Remark \ref{virt cyc subgroups of tame action}, the only possible isomorphism class for the maximal loxodromic subgroup $E$ containing $g$ is $\mathbb Z$. Without loss of generality, we may assume $g$ to be primitive. Notice the following fact: by the choice of $\lambda ''$, in the rescaled space $X_0$ the element $g$ has translation length $\leq L_S \delta _1$. This has the following consequence (see the proof of Theorem \ref{limit step quotient}): the quotient map $G \twoheadrightarrow \hat G_1$ induces an epimorphism $E \twoheadrightarrow \hat E_1$, where $\hat E_1 \cong C_{q_1}$.
    Moreover, by Lemma \ref{cent of element of order >3 in wstprime is elliptic} the normalizer $N_{\hat G _1}(\langle \hat g _1 \rangle)$ of the image $\langle \hat g_1 \rangle$ of $\langle g \rangle$ on $\hat G_1$ is elliptic, thus by Lemma \ref{characterization of lifting elliptics sc theorem} it is contained in $\stab (\hat v)= \hat E_1$. In particular, $\hat g _1$ is not a translation and it centralizes no involution.

    Now, an argument completely analogous to the proof of Lemma \ref{norm of fin sgrp order greater 3 induces iso pp-quotient wstprime} (with initializing step 1 instead of 0) applied to $\langle \hat g _1 \rangle$ gives that the quotient map $\hat G _1 \twoheadrightarrow \bar G$ induces an isomorphism from $N_{\hat G _1}(\langle \hat g _1 \rangle)$ onto $N_{\bar G}(\langle \bar g \rangle)$. In particular, $\bar g$ is of order $q_1$, it is not a translation and it centralizes no involution.
\end{proof}

\begin{lemma}
\label{pp-quotient of wstprime contains elem q2 centralizing inv}
    The quotient map $G \twoheadrightarrow \bar G$ satisfies Property \ref{pp-quotient of wstprime elem cent an inv restatement} of Proposition \ref{restatement proposition pp-quotient of wstprime}: if $G$ contains an element of infinite order (which is not a translation), that has translation length 1 and that centralizes an involution, then $\bar G$ contains an element of order $q_2$ which is not a translation and centralizes an involution.
\end{lemma}

\begin{proof}
    Let $g \in G$ be a loxodromic element satisfying the hypotheses of the lemma. By Remark \ref{virt cyc subgroups of tame action}, the only possible isomorphism class for the maximal loxodromic subgroup $E$ containing $g$ is $\mathbb Z \times C_2$. Without loss of generality, we may assume $g$ to be primitive. Notice the following fact: by the choice of $\lambda ''$, in the rescaled space $X_0$ the element $g$ has translation length $\leq L_S \delta _1$. This has the following consequence (see the proof of Theorem \ref{limit step quotient}): the quotient map $G \twoheadrightarrow \hat G_1$ induces an epimorphism $E \twoheadrightarrow \hat E_1$, where $\hat E_1 \cong C_{q_2} \times C_2$.
    Moreover, by Lemma \ref{cent of element of order >3 in wstprime is elliptic} the normalizer $N_{\hat G _1}(\langle \hat g _1 \rangle)$ of the image $\langle \hat g_1 \rangle$ of $\langle g \rangle$ on $\hat G_1$ is elliptic, thus by Lemma \ref{characterization of lifting elliptics sc theorem} it is contained in $\stab (\hat v)= \hat E_1$. In particular, $\hat g _1$ is not a translation.

    Now, an argument completely analogous to the proof of Lemma \ref{norm of fin sgrp order greater 3 induces iso pp-quotient wstprime} (with initializing step 1 instead of 0) applied to $\langle \hat g _1 \rangle$ gives that the quotient map $\hat G _1 \twoheadrightarrow \bar G$ induces an isomorphism from $N_{\hat G _1}(\langle \hat g _1 \rangle)$ onto $N_{\bar G}(\langle \bar g \rangle)$. In particular, $\bar g$ is of order $q_2$, it is not a translation and it centralizes an involution.
\end{proof}

It only remains to prove that the group $\bar G$ is weakly sharply 2-transitive of characteristic $p$ and of $(q_1,q_2)$-almost bounded exponent.

\begin{lemma}
\label{pp-quotient of wstprime every pair is of paff or pmin}
    The group $\bar G$ satisfies Condition \ref{def class C pairs of minimal or affine type} of Definition \ref{definition weakly s2t}: every translation of $\bar G$ is of order $p$ and every pair $(\bar r , \bar s) \in \invset{\bar G}{(2)}$ is either of $p$-affine or of $p$-minimal type.
\end{lemma}

\begin{proof}
    The fact that every translation is of order $p$ follows directly from the fact that the pair $(G,X)$ is in class $\classwstprime{p,q_1 , q_2}$, that the family $Q$ is the set of all loxodromic elements of $Q$ and from the choice of $\expfamily$.

    Let $(\bar r , \bar s)$ be a pair of distinct involutions such that $D_{\bar r , \bar s}$ is a proper subgroup of $N_{\bar G}(\langle \bar r \bar s \rangle)$. Let $(r,s) \in \invset{G}{(2)}$ be a preimage of $(\bar r , \bar s)$ in $G$. The pair $(r,s)$ cannot be of $p$-minimal type: otherwise, by Lemma \ref{norm of fin sgrp order greater 3 induces iso pp-quotient wstprime} applied to $\langle rs \rangle$ we would get that $(\bar r , \bar s)$ is of $p$-minimal type as well. Assume now that $(r,s)$ is of infinite order. Then, an argument exactly as in the proof of Lemma \ref{pp-quotient of wstprime image of paff and of pmin is paff and pmin} yields that the pair $(\bar r , \bar s)$ is again of $p$-minimal type. 
    
    Thus, since the group $G$ is weakly sharply 2-transitive of characteristic $p$, we get that the pair $( r ,  s)$ has to be of $p$-affine type. Then, Lemma \ref{norm of fin sgrp order greater 3 induces iso pp-quotient wstprime} applied to the subgroup $\langle rs \rangle$ yields that $(\bar r , \bar s)$ is of $p$-affine type.
\end{proof}

\begin{lemma}
\label{pp-quotient of wstprime trans on paff}
    The group $\bar G$ satisfies Condition \ref{def class C G trans on affine trans} of Definition \ref{definition weakly s2t}: the set of pairs $(\bar r,\bar s) \in \mathcal{I}_{\bar G}^{(2)}$ of $p$-affine type is non-empty and $\bar G$ acts transitively on it by conjugation.
\end{lemma}

\begin{proof}
    Since the group $G$ is weakly sharply 2-transitive of characteristic $p$, the set of pairs of $p$-affine type of $G$ is non-empty. Let $(r,s)$ be one such pair, then, by Lemma \ref{pp-quotient of wstprime image of paff and of pmin is paff and pmin}, its image $(\bar r , \bar s)$ in $\bar G$ is a pair of $p$-affine type. Now, let $(\bar r , \bar s)$ and $(\bar r' , \bar s')$ be two pairs of $p$-affine type. Again, Lemma \ref{pp-quotient of wstprime image of paff and of pmin is paff and pmin} gives that the pairs have preimages $(r,s)$ and $(r' , s')$ (respectively) that are pairs of $p$-affine type. Since $G$ is weakly sharply 2-transitive of characteristic $p$, there is an element $g$ of $G$ conjugating $(r,s)$ to $(r',s')$, and thus its image $\bar g$ conjugates $(\bar r , \bar s)$ to $(\bar r' , \bar s')$. Therefore, $\bar G$ acts transitively on the set of pairs of $\invset{\bar G}{(2)}$ of $p$-affine type.
\end{proof}

\begin{lemma}
\label{pp-quotient of wstprime cent of translation is cyclic}
    The group $\bar G$ satisfies Condition \ref{def class C cent of trans is cyclic} of Definition \ref{definition weakly s2t}: for every pair $(\bar r, \bar s) \in \mathcal{I}_{\bar G}^{(2)}$ the subgroup $\text{Cen}(\bar r \bar s)$ is cyclic and generated by a translation.
\end{lemma}

\begin{proof}
    Let $(\bar r , \bar s) \in \invset{\bar G}{(2)}$.

    If $(\bar r , \bar s)$ is of $p$-minimal type, then $N_{\bar G} (\langle \bar r \bar s \rangle) =D_{\bar r , \bar s}$, and then the conclusion follows from the fact that $\Cen_{\bar G}(\bar r \bar s) \leq N_{\bar G} (\langle \bar r \bar s \rangle)$ and in $D_p$ centralizers of translations are cyclic and generated by translations.
    
    If $(\bar r , \bar s)$ is of $p$-affine type, let $(r,s) \in \invset{G}{(2)}$ be a preimage of $(\bar r , \bar s)$ on $G$. By Lemma \ref{pp-quotient of wstprime image of paff and of pmin is paff and pmin}, necessarily $(r,s)$ is of $p$-affine type. Now, Lemma \ref{norm of fin sgrp order greater 3 induces iso pp-quotient wstprime} applied to $\langle rs \rangle$ gives that the quotient map induces an isomorphism from $N_G (\langle rs \rangle)$ onto $N_{\bar G} (\langle \bar r \bar s \rangle)$ (and thus from $\Cen_G (rs)$ onto $\Cen_{\bar G}(\bar r \bar s)$), and then the desired conclusion follows from the fact that, since $G$ is weakly sharply 2-transitive of characteristic $p$, we have that $N_G (\langle rs \rangle) \cong \agl$, and in this group a centralizer of a translation is cyclic and generated by a translation.
\end{proof}

\begin{lemma}
\label{pp-quotient of wstprime is of almost bounded exponent}
    The group $\bar G$ satisfies Condition \ref{def class C almost bounded exponent} of Definition \ref{definition weakly s2t}: $\bar G$ is of $(q_1 , q_2)$-almost bounded exponent. That is, for every subgroup $\bar E$ of finite order, either $\bar E$ is contained in a subgroup of $\bar G$ that embeds into $\agl$ or $\bar E$ falls into one of the following cases.
        \begin{enumerate}
            \item The subgroup $\bar E$ is contained in a subgroup isomorphic to $C_{q_1}$ and no non-trivial element of $\bar E$ centralizes an involution.
            \item The subgroup $\bar E$ is contained in a subgroup isomorphic to $C_{2q_2}$ (and thus every element of $\bar E$ centralizes an involution).
        \end{enumerate}
\end{lemma}

\begin{proof}
    Let $\bar E$ be a subgroup of $\bar G$ of finite order. If $\bar E$ is of order 2, then it embeds into $\agl$.

    Now, suppose that $\bar E$ has finite order at least 3. Assume first that $ \bar E $ lifts, i.e., that there is a preimage $E$ of $\bar E$ of the same order as $\bar E$. By Lemma \ref{norm of fin sgrp order greater 3 induces iso pp-quotient wstprime}, we have that the quotient map induces an isomorphism from $N_G(E )$ onto $N_{\bar G}(\bar E)$ (and thus from $\Cen_G (E)$ onto $\Cen_{\bar G}(\bar E)$). The desired conclusion follows now from the fact that $G$ is of $(q_1 , q_2)$-almost bounded exponent.

    Suppose now that $\bar E$ does not lift, and let $E$ be a preimage of $\bar E$ in $G$. Notice that $E$ is necessarily of infinite order: otherwise, by Lemma \ref{pp-quotient of wstprime induces iso for elliptic}, the quotient map would induce an isomorphism from $E$ onto $\bar E$ and thus $\bar E$ would lift. Thus, there exists $m \in \mathbb N$ such that the image $\hat E_m$ of $E$ on $\hat G _m$ is of finite order at least 3. Therefore, since $\hat G _{m}$ is of $(q_1 , q_2)$-almost bounded exponent, $\hat E_{m}$ falls into one of the cases from Condition \ref{def class C almost bounded exponent} of Definition \ref{definition weakly s2t}. Now, an argument completely analogous to the one in the proof of Lemma \ref{norm of fin sgrp order greater 3 induces iso pp-quotient wstprime} (initialized at step $m$ instead of step 0) gives that the quotient map $\hat G _{m} \twoheadrightarrow \bar G$ induces an isomorphism from $N_{\hat G _{m}}(\hat E_m)$ onto $N_{\bar G}(\bar E)$ (and thus also from $\Cen_{\hat G_m} (\hat E_m)$ onto $\Cen _{\bar G}(\bar E)$), and from this the desired conclusion follows as in the previous paragraph.
    %
\end{proof}

\section{Proof of Proposition \ref{proposition hnn of wst}}
\label{section HNN-ext of wst}

The goal of this section is to prove Proposition \ref{proposition hnn of wst}. For the sake of completeness, we restate this result here.

\begin{proposition}
\label{restatement proposition hnn of wst}
    Let $G$ be a group in class $\classwst{p,q_1,q_2}$ for integers $p $, $q_1$ and $q_2$ at least $ n'_1$. Let $(r,s)$ and $(r',s')$ be pairs in $\invset{G}{(2)}$ with $(r,s)$ of $p$-affine type and $(r',s')$ of $p$-minimal type (so that both $D_{r,s}$ and $D_{r',s'}$ are isomorphic to $D_p$). Then the following holds.
    \begin{enumerate}[label=(\arabic*)]
        \item \label{hnn of wst trivial subgroups restatement} Let $\gstar= G \ast \mathbb{Z}$ and $X$ the Bass-Serre tree of the splitting of $\gstar$ as an HNN-extension of $G$ with trivial associated subgroups. Then, the pair $(\gstar , X)$ is in class $\classwstprime{p,q_1,q_2}$.
        \item \label{hnn of wst trivial subgroups dihedral restatement} Let $\gstar$ be the following HNN-extension: \[ \langle G,t \, \vert \, t^{-1}rt=r'  , \, t^{-1}st=s' \rangle, \](an HNN-extension of $G$ with associated subgroups $D_{r,s}$ and $D_{r',s'}$). Let $X$ be the Bass-Serre tree of this splitting of $\gstar$. Then, the pair $(\gstar , X)$ is in class $\classwstprime{p,q_1,q_2}$.
    \end{enumerate}
    Moreover, the group $\gstar$ has the following additional properties.
    \begin{enumerate}[label=(\arabic*')]
        \item \label{hnn of wst trans of inf order restatement} In case \ref{hnn of wst trivial subgroups restatement}, $\gstar$ contains a translation of infinite order and translation length at most 2.
        \item \label{hnn of wst inf order element cent no inv restatement} In case \ref{hnn of wst trivial subgroups restatement}, $\gstar$ contains an element of infinite order that is not a translation, that has translation length 1 and that centralizes no involution.
        \item \label{hnn of wst elem of inf order cent an inv restatement} In case \ref{hnn of wst trivial subgroups dihedral restatement}, if $\lvert D_{r,s} \cap D_{r',s'} \rvert=2$, then $\gstar$ contains an element of infinite order (which is not a translation), that has translation length 1 and that centralizes an involution.
    \end{enumerate}
\end{proposition}

For the remainder of this section, we fix a group $G$ and pairs $(r,s)$ and $(r',s')$ satisfying the hypotheses of Proposition \ref{restatement proposition hnn of wst}. Notice first that it is immediate that the pair $(\gstar , X)$ satisfies Condition \ref{def class wstprime} of Definition \ref{class wstprime infinte order implies loxo}, since elliptic elements are conjugate into the base group $G$, all of whose elements have finite order.

We start by stating the following result, which appeared in \cite{amelio_andre_tent} as Lemma 6.1.

\begin{lemma}
\label{hnnext of quasi-malnormal path stab small}
Let $G$ be a group, and let $(K,K')$ be a jointly quasi-malnormal pair (see Definition~\ref{def jointly quasi-malnormal}) of isomorphic subgroups of $G$. Let $\alpha : K\rightarrow K'$ be an isomorphism, and consider the group $\gstar=\langle G,t \ \vert \ tkt^{-1}=\alpha(k) \, : \, k \in K \rangle$. Let $X$ be the corresponding Bass-Serre tree. Then stabilizers of paths with at least three edges in $X$ have order at most two.
\end{lemma}

Notice that Lemma \ref{hnnext of quasi-malnormal path stab small} applies to both possible HNN-extensions of Proposition \ref{restatement main theorem}. We begin by proving some results on the action of $\gstar$ on $X$ that will be useful when proving that $\gstar$ is weakly sharply 2-transitive of characteristic $p$.

\begin{lemma}
\label{action of hnn of wstprime acyl and tame}
    The action of $\gstar$ on $X$ is acylindrical, non-elementary and tame (so in particular the pair $(\gstar , X)$ satisfies Condition \ref{class wstprime action non-elementary} of Definition \ref{def class wstprime}).
\end{lemma}

\begin{proof}
    Acylindricity follows directly from Lemmas \ref{paff and pmin pair is quasimalnormal} and \ref{hnnext of quasi-malnormal path stab small}. Moreover, the action is non-elementary since $\gstar$ is not virtually cyclic.

    For tameness, notice that no loxodromic element $g$ can normalize a subgroup of $\gstar $ of order greater than 2, since then this element would fix pointwise the axis of $g$, contradicting Lemma \ref{hnnext of quasi-malnormal path stab small}. Moreover, $\gstar$ cannot contain a subgroup of order 4 since this would be elliptic and thus conjugate to a subgroup of order 4 of $G$, which cannot exist given that this group is in class $\classwstprime{p,q_1,q_2}$: indeed, a finite subgroup in this class embeds either into $C_{q_1}$, into $C_{2q_2}$ or into $\agl$, and none of these groups contains subgroups of order 4 (since $q_1$ and $q_2$ are odd, while $\agl$ is of order $p(p-1)$ and $p \equiv 3 \pmod{4}$).
\end{proof}

This result has the following useful consequence.

\begin{lemma}
\label{normalizer of hnn-ext of wst of order geq 3 is elliptic}
    Let $F$ be a finite subgroup of $\gstar$ of order at least 3. Then, its normalizer $N_{\gstar}(F) $ is elliptic.
\end{lemma}

\begin{proof}
    This is a direct consequence of Lemmas \ref{cent of element of order >3 in wstprime is elliptic} and \ref{action of hnn of wstprime acyl and tame}.
\end{proof}

\begin{lemma}
\label{hnn-ext of wst trans of paff or pmin}
    The group $\gstar$ satisfies Condition \ref{def class C pairs of minimal or affine type} of Definition \ref{definition weakly s2t}: every translation is either of order $p$ or of infinite order, and every pair $(r,s) \in \mathcal{I}_{\gstar}^{(2)}$ such that $rs$ is of order $p$ is either of $p$-minimal type or of $p$-affine type.
\end{lemma}

\begin{proof}
    Let $(r,s) \in \mathcal{I}_{\gstar}^{(2)}$. If $rs$ has finite order, then $D_{r,s}$ is elliptic by its action on $X$, and so is $N_{\gstar}(D_{r,s})$ by Lemma \ref{normalizer of hnn-ext of wst of order geq 3 is elliptic}. Therefore, $N_{\gstar}(D_{r,s})$ is conjugate into the base group $G$, and since $G$ is in class $\classwstprime{p,q_1,q_2}$, it follows that $rs$ has order $p$. Now, by Remark \ref{remark on paff and pmin} \ref{paff and pmin invariant under conj}, the pair $(r,s)$ is either of $p$-affine or of $p$-minimal type.
\end{proof}

\begin{lemma}
\label{hnn-ext of wst transitive on paff}
    The group $\gstar$ satisfies Condition \ref{def class C G trans on affine trans} of Definition \ref{definition weakly s2t}: the set of pairs $(r,s) \in \mathcal{I}_{\gstar}^{(2)}$ of $p$-affine type is non-empty and $\gstar$ acts transitively on it by conjugation.
\end{lemma}

\begin{proof}
    The fact that $\gstar$ has pairs of $p$-affine type follows from the fact that $G$ embeds into it and this group is weakly sharply 2-transitive of characteristic $p$.

    Now, let $(r,s) \in \mathcal{I}_{\gstar}^{(2)}$ and let $H \cong \agl$ contain $r$ and $s$. This subgroup is finite, and thus it is elliptic. In particular, it is conjugate to a subgroup $H'$ of $G$ isomorphic to $\agl$, so $(r,s)$ is conjugate to a pair $(r',s') \in \invset{G}{(2)}$ of $p$-affine type. Now, the desired conclusion follows from the fact that $G$ is weakly sharply 2-transitive of characteristic $p$, and as such it acts transitively on $\invset{G}{(2)}$.
\end{proof}

\begin{lemma}
\label{hnn-ext of wst cent of trans is cyclic}
    The group $\gstar$ satisfies Condition \ref{def class C cent of trans is cyclic} of Definition \ref{definition weakly s2t}: for every pair $(r,s) \in \mathcal{I}_{\gstar}^{(2)}$ the subgroup $\Cen(rs)$ is cyclic and generated by a translation.
\end{lemma}

\begin{proof}
    Let $(r,s) \in \mathcal{I}_{\gstar}^{(2)}$.
    
    If $rs$ is of order $p$, then by Lemma \ref{normalizer of hnn-ext of wst of order geq 3 is elliptic} its centralizer is conjugate into the base group $G$, where centralizers of translations are cyclic and generated by a translation.

    If $rs$ is of infinite order, it is loxodromic, and every element of $\gstar$ centralizing $rs$ is in the maximal loxodromic subgroup $E$ containing $rs$. Notice that both $r$ and $s$ are also in $E$. Now, by Remark \ref{virt cyc subgroups of tame action}, $E$ is isomorphic to $D_\infty$, since this is the only possible isomorphism type containing more than one involution, and in this group, centralizers of translations are cyclic and generated by a translation.
\end{proof}

\begin{lemma}
\label{hnn-ext of wst of almost bounded exponent}
    The group $\gstar$ satisfies Condition \ref{def class C almost bounded exponent} of Definition \ref{definition weakly s2t}: it is of $(q_1 , q_2)$-almost bounded exponent, i.e., for every subgroup $E$ of finite order, either $E$ is contained in a subgroup of $\gstar$ that embeds into $\agl$ or $E$ falls into one of the following cases.
        \begin{enumerate}
            \item The subgroup $E$ is contained in a subgroup isomorphic to $C_{q_1}$ and no non-trivial element of $E$ centralizes an involution.
            \item The subgroup $E$ is contained in a subgroup isomorphic to $C_{2q_2}$ (and thus every element of $E$ centralizes an involution).
        \end{enumerate}
\end{lemma}

\begin{proof}
    Let $E$ be a subgroup of finite order.

    If $E$ is of order 2, then $E$ embeds into $\agl$.

    If $E$ has order $\geq 3$, then by Lemma \ref{normalizer of hnn-ext of wst of order geq 3 is elliptic} we have that $N_{\gstar}(E)$ (and thus also $\Cen _{\gstar}(E)$) is elliptic and therefore conjugate into $G$. The desired conclusion follows from the fact that $G$ is of $(q_1,q_2)$-almost bounded exponent.
\end{proof}

The next lemma finishes the proof that the pair $(\gstar , X)$ is in class $\classwstprime{p,q_1,q_2}$.

\begin{lemma}
\label{hnn-ext of wst satsfies parameters}
    The pair $(\gstar , X)$ satisfies Condition \ref{class wstprime parameters} from Definition \ref{def class wstprime}: the action is tame and is such that $\tau (\gstar ,X) \leq 5$ and $\Omega(\gstar ,X)=0$; and the integers $p$, $q_1$ and $q_2$ are at least $n'_1$.
\end{lemma}

\begin{proof}
    Tameness of the action was proved in Lemma \ref{action of hnn of wstprime acyl and tame}. The requirements on $p$, $q_1$ and $q_2$ hold by assumption (since $G$ is in class $\classwst{p,q_1,q_2}$).

    Now, the space $X$ is $0$-hyperbolic. Therefore, by Lemma \ref{hnnext of quasi-malnormal path stab small}, the parameters of the definition of an acylindrical action corresponding to $\varepsilon = 97 \delta $ can be taken to be $L=3$ and $M=2$ in the case where the associated subgroups of the HNN-extension are dihedral, or $L=M=1$ if the associated subgroups are trivial. In both cases, by Remark \ref{acylind characterization}, we get that $\nu (\gstar , X) \leq 5$, and thus also by Definition \ref{def parameter tau} that $\tau (\gstar , X) \leq 5$.

    Finally, for $\Omega (\gstar ,X)$, notice first that, since our space $X$ is $0$-hyperbolic, the axis of an element $g$ is the set $A_{g}= \{ x \in X : d(x,g \cdot x)=[g] \}$. Also, since $\tau=\tau(\gstar ,X)$ is finite, we are considering tuples $(g_{0}, \dots ,g_{\tau})$ of elements with $[g_{i}]=0$ for all $i \in \{ 0, \dots , \tau \}$ (that is, tuples of elliptic elements) such that they do not generate an elementary subgroup. In this case, $A_{g_{i}}$ is just the fixed-point set of $g_{i}$, and $A(g_{0}, \dots ,g_{\tau})=\text{diam}(A_{g_{0}} \cap \dots \cap A_{g_{\tau}})$ is the diameter of the intersection of their fixed-point sets. If all these elliptic elements had a fixed point in common, then every element of $\langle g_{0}, \dots , g_{\tau} \rangle$ would also fix this point, and therefore, they would generate an elementary subgroup. That is, we only need to consider tuples $(g_{0}, \dots , g_{\tau})$ of elliptic elements such that $A_{g_{0}} \cap \dots \cap A_{g_{\tau}}= \varnothing$. Thus, $A(g_{0}, \dots , g_{\tau})=0$, and therefore, $\Omega(\gstar,X)=0$.
\end{proof}

Now, we prove that the group $\gstar$ satisfies the extra properties announced in proposition \ref{proposition hnn of wst}.

\begin{lemma}
    The group $\gstar$ satisfies Property \ref{hnn of wst trans of inf order} from Proposition \ref{proposition hnn of wst}: in case \ref{hnn of wst trivial subgroups} of the aforementioned proposition, $\gstar$ contains a translation of infinite order and translation length at most 2.
\end{lemma}

\begin{proof}
    Let $t$ be the stable letter of the HNN-extension and $r$ an involution of $G$. Consider the translation $g=rt^{-1}rt$. This translation has infinite order (since $\gstar$ is isomorphic to the free product of $G$ with the infinite cyclic group generated by $t$). Moreover, this element translates the vertex $v=G$ fixed by $G$ to the vertex $v'=gG$ fixed by $gGg^{-1}$. Now, by construction, there is an edge $e=rt^{-1}r$ with terminal vertex $v'=rt^{-1}rtG$ and origin vertex $v''=rt^{-1}rG=rt^{-1}G$. Furthermore, there is an edge $e'=rt^{-1}$ with origin vertex $v''=rt^{-1}G$ and terminal vertex $v=rG=G$. Thus, $d(v,v') \leq 2$, and the translation length of $g$ is at most 2.
\end{proof}

\begin{lemma}
\label{hnn-ext of wst contains tr length 1 and no cent inv}
    The group $\gstar$ satisfies Property \ref{hnn of wst inf order element cent no inv} from Proposition \ref{proposition hnn of wst}: in case \ref{hnn of wst trivial subgroups} of the aforementioned proposition, $\gstar$ contains an element of infinite order that is not a translation, that has translation length 1 and that centralizes no involution.
\end{lemma}

\begin{proof}
    Let $t$ be the stable letter of the HNN-extension. It is a loxodromic element that cannot be a translation (since $\gstar$ is isomorphic to the free product of $G$ with the infinite cyclic group generated by $t$, loxodromic translations have normal form of length at least 4). Moreover, edge stabilizers in $X$ are trivial, so the maximal normal finite subgroup of every loxodromic subgroup is trivial. In particular, $t$ centralizes no involution. Moreover, $t$ translates the vertex $v=G$ to the vertex $v'=tG$, and these are connected by an edge labelled by the identity. Thus, the translation length of $t$ is exactly 1.
\end{proof}

The next result completes the proof of Proposition \ref{restatement proposition hnn of wst}.

\begin{lemma}
    The group $\gstar$ satisfies Property \ref{hnn of wst elem of inf order cent an inv} from Proposition \ref{proposition hnn of wst}: in case \ref{hnn of wst trivial subgroups dihedral} of the aforementioned proposition, if $\lvert D_{r,s} \cap D_{r',s'} \rvert=2$, then $\gstar$ contains an element of infinite order (which is not a translation), that has translation length 1 and that centralizes an involution.
\end{lemma}

\begin{proof}
    Without loss of generality, we may assume that $\lvert D_{r,s} \cap D_{r',s'} \rvert = \langle r' \rangle$. Now, all pairs of involutions of $D_{r,s}$ are conjugate by an element of $G$ (since the pair $(r,s)$ is assumed to be of $p$-affine type). Thus, up to further conjugating by one such element, we may assume that one of the defining relations of the HNN-extension is $t^{-1}r' t = r'$.

     Now, the element $t$ is loxodromic and it translates the vertex $v=G$ to the vertex $v'=tG$. These vertices are connected by an edge labelled by the identity. Thus, the translation length of $t$ is exactly 1, and the desired conclusion follows.
\end{proof}

\printbibliography

@Book{adian,
 Author = {Adian, S. I.},
 Title = {The {Burnside} problem and identities in groups. {Translated} from the {Russian} by {John} {Lennox} and {James} {Wiegold}},
 FSeries = {Ergebnisse der Mathematik und ihrer Grenzgebiete},
 Series = {Ergeb. Math. Grenzgeb.},
 Volume = {95},
 Year = {1979},
 Publisher = {Springer-Verlag, Berlin},
 Language = {English},
 zbMATH = {3648963},
 Zbl = {0417.20001}
}

@Article{adian_novikov,
 Author = {Novikov, P. S. and Adyan, S. I.},
 Title = {On infinite periodic groups. {I}--{III}},
 FJournal = {Mathematics of the USSR. Izvestiya},
 Journal = {Math. USSR, Izv.},
 ISSN = {0025-5726},
 Volume = {2},
 Pages = {209--236, 241--480, 665--685},
 Year = {1969},
 Language = {English},
 DOI = {10.1070/IM1968v002n01ABEH000637},
 zbMATH = {3308382},
 Zbl = {0194.03301}
}

@misc{amelio_andre_tent,
  title        = {Non-split sharply 2-transitive groups of odd positive characteristic},
  author       = {Marco Amelio and Simon André and Katrin Tent},
  year         = {2023},
  note         = {Accepted for publication in \emph{International Mathematics Research Notices}},
  eprint       = {2312.16992},
  archivePrefix= {arXiv},
  primaryClass = {math.GR}
}

@article{andre_guir_fin_gen_simple,
      title={Finitely generated simple sharply 2-transitive groups}, 
      author={Simon André and Vincent Guirardel},
      year={2022},
      eprint={2212.06020},
      archivePrefix={arXiv},
      primaryClass={math.GR}
}

@article {andre_tent,
    AUTHOR = {Andr\'{e}, Simon and Tent, Katrin},
     TITLE = {Simple sharply 2-transitive groups},
   JOURNAL = {Trans. Amer. Math. Soc.},
  FJOURNAL = {Transactions of the American Mathematical Society},
    VOLUME = {376},
      YEAR = {2023},
    NUMBER = {6},
     PAGES = {3965--3993},
      ISSN = {0002-9947,1088-6850},
   MRCLASS = {20B22 (20E06 20E32)},
  MRNUMBER = {4586803},
       DOI = {10.1090/tran/8846},
       URL = {https://doi.org/10.1090/tran/8846},
}

@article{arzhantseva_delzant,
  title={Examples of random groups},
  author={Arzhantseva, Goulnara and Delzant, Thomas},
  journal={preprint},
  volume={2011},
  year={2008}
}

@misc{atkarskaya_rips_tent,
      title={The Burnside problem for odd exponents}, 
      author={Agatha Atkarskaya and Eliyahu Rips and Katrin Tent},
      year={2023},
      eprint={2303.15997},
      archivePrefix={arXiv},
      primaryClass={math.GR}
}

@article{bowditch,
  title={Tight geodesics in the curve complex},
  author={Bowditch, Brian H},
  journal={Inventiones mathematicae},
  volume={171},
  number={2},
  pages={281--300},
  year={2008},
  publisher={Springer}
}

@book{bridson_haefliger,
  title={Metric spaces of non-positive curvature},
  author={Bridson, Martin R and Haefliger, Andr{\'e}},
  volume={319},
  year={2013},
  publisher={Springer Science \& Business Media}
}

@book{burago_burago_ivanov,
  title={A course in metric geometry},
  author={Burago, Dmitri and Burago, Yuri and Ivanov, Sergei},
  volume={33},
  year={2022},
  publisher={American Mathematical Society}
}

@book{coulon_bourbaki,
 Author = {Coulon, R{\'e}mi},
 Title = {Small cancellation theory: a geometric approach (after {F}. {Dahmani}, {V}. {Guirardel}, {D}. {Osin}, and {S}. {Cantat}, {S}. {Lamy})},
 BookTitle = {S\'eminaire Bourbaki. Volume 2014/2015. Expos\'es 1089--1103},
 ISBN = {978-2-85629-836-7},
 Pages = {1--33, ex},
 Year = {2016},
 Publisher = {Paris: Soci{\'e}t{\'e} Math{\'e}matique de France (SMF)},
 Language = {French},
 zbMATH = {7405490},
 Zbl = {1470.20002}
}

@article{cantat_lamy,
 Author = {Cantat, Serge and Lamy, St{\'e}phane and de Cornulier, Yves},
 Title = {Normal subgroups in the {Cremona} group},
 FJournal = {Acta Mathematica},
 Journal = {Acta Math.},
 ISSN = {0001-5962},
 Volume = {210},
 Number = {1},
 Pages = {31--94},
 Year = {2013},
 Language = {English},
 DOI = {10.1007/s11511-013-0090-1},
 zbMATH = {6191482},
 Zbl = {1278.14017}
}

@book{gromov_mesoscopic,
 Author = {Gromov, Misha},
 Title = {Mesoscopic curvature and hyperbolicity},
 BookTitle = {Global differential geometry: the mathematical legacy of Alfred Gray. Proceedings of the international congress on differential geometry held in memory of Professor Alfred Gray, Bilbao, Spain, September 18--23, 2000},
 ISBN = {0-8218-2750-2},
 Pages = {58--69},
 Year = {2001},
 Publisher = {Providence, RI: American Mathematical Society (AMS)},
 Language = {English},
 zbMATH = {1749240},
 Zbl = {1006.53036}
}

@article{burnside,
  title={On an unsettled question in the theory of discontinuous groups},
  author={Burnside, William},
  journal={Quart. J. Pure and Appl. Math.},
  volume={33},
  pages={230--238},
  year={1902}
}

@book{coornaert_delzant_papadopoulos,
  title={G{\'e}om{\'e}trie et th{\'e}orie des groupes: les groupes hyperboliques de Gromov},
  author={Coornaert, Michel and Delzant, Thomas and Papadopoulos, Athanase},
  volume={1441},
  year={2006},
  publisher={Springer}
}

@article{coulon_3,
  title={Asphericity and small cancellation theory for rotation families of groups},
  author={Coulon, R{\'e}mi},
  journal={Groups, Geometry, and Dynamics},
  volume={5},
  number={4},
  pages={729--765},
  year={2011}
}

@misc{coulon_4,
      title={Infinite periodic groups of even exponents}, 
      author={Rémi Coulon},
      year={2021},
      eprint={1810.08372},
      archivePrefix={arXiv},
      primaryClass={math.GR}
}

@article{coulon_2,
  title={On the geometry of Burnside quotients of torsion free hyperbolic groups},
  author={Coulon, R{\'e}mi},
  journal={International Journal of Algebra and Computation},
  volume={24},
  number={03},
  pages={251--345},
  year={2014},
  publisher={World Scientific}
}

@inproceedings{coulon_1,
  title={Partial periodic quotients of groups acting on a hyperbolic space},
  author={Coulon, R{\'e}mi B},
  booktitle={Annales de l'Institut Fourier},
  volume={66},
  number={5},
  pages={1773--1857},
  year={2016}
}

@book{dahmani_guirardel_osin,
  title={Hyperbolically embedded subgroups and rotating families in groups acting on hyperbolic spaces},
  author={Dahmani, Fran{\c{c}}ois and Guirardel, Vincent and Osin, Denis},
  volume={245},
  number={1156},
  year={2017},
  publisher={American Mathematical Society}
}

@Article{delzant_gromov,
 Author = {Delzant, Thomas and Gromov, Misha},
 Title = {Mesoscopic curvature and very small cancellation theory. (Courbure m{\'e}soscopique et th{\'e}orie de la toute petite simplification.)},
 FJournal = {Journal of Topology},
 Journal = {J. Topol.},
 ISSN = {1753-8416},
 Volume = {1},
 Number = {4},
 Pages = {804--836},
 Year = {2008},
 Language = {French},
 DOI = {10.1112/jtopol/jtn023},
 zbMATH = {5488495},
 Zbl = {1197.20035}
}

@article{Sela,
 Author = {Sela, Zlil},
 Title = {Acylindrical accessibility for groups},
 FJournal = {Inventiones Mathematicae},
 Journal = {Invent. Math.},
 ISSN = {0020-9910},
 Volume = {129},
 Number = {3},
 Pages = {527--565},
 Year = {1997},
 Language = {English},
 DOI = {10.1007/s002220050172},
 zbMATH = {1054509},
 Zbl = {0887.20017}
}

@book{drutu_kapovich,
  title={Geometric group theory},
  author={Dru{\c{t}}u, Cornelia and Kapovich, Michael},
  volume={63},
  year={2018},
  publisher={American Mathematical Soc.}
}

@article{hull_osin,
  title={Transitivity degrees of countable groups and acylindrical hyperbolicity},
  author={Hull, Michael and Osin, Denis},
  journal={Israel Journal of Mathematics},
  volume={216},
  number={1},
  pages={307--353},
  year={2016},
  publisher={Springer}
}

@Article{ivanov,
 Author = {Ivanov, Sergei V.},
 Title = {On the {Burnside} problem on periodic groups},
 FJournal = {Bulletin of the American Mathematical Society. New Series},
 Journal = {Bull. Am. Math. Soc., New Ser.},
 ISSN = {0273-0979},
 Volume = {27},
 Number = {2},
 Pages = {257--260},
 Year = {1992},
 Language = {English},
 DOI = {10.1090/S0273-0979-1992-00305-1},
 zbMATH = {94486},
 Zbl = {0821.20024}
}

@article{ivanov_2,
  title={The free Burnside groups of sufficiently large exponents},
  author={Ivanov, Sergei V},
  journal={International Journal of Algebra and Computation},
  volume={4},
  number={01n02},
  pages={1--308},
  year={1994},
  publisher={World Scientific}
}

@article{jabara,
  title={On sharply 2-transitive groups with point stabilizer of exponent 2 n⋅ 3},
  author={Jabara, Enrico},
  journal={Communications in Algebra},
  volume={46},
  number={2},
  pages={544--551},
  year={2018},
  publisher={Taylor \& Francis}
}

@article{jordan,
  title={Recherches sur les substitutions},
  author={Jordan, Camille},
  journal={Journal de Math{\'e}matiques Pures et Appliqu{\'e}es},
  volume={17},
  pages={351--367},
  year={1872}
}

@book{kerby,
  title={On infinite sharply multiply transitive groups},
  author={Kerby, William},
  number={6},
  year={1974},
  publisher={Vandenhoeck \& Ruprecht}
}

@article{kourovka,
  title={Unsolved problems in group theory. The Kourovka notebook},
  author={Khukhro, EI and Mazurov, VD},
  journal={arXiv preprint arXiv:1401.0300},
  year={2014}
}

@Article{lysenok,
 Author = {Lysenok, I. G.},
 Title = {Infinite {Burnside} groups of even exponent},
 FJournal = {Izvestiya: Mathematics},
 Journal = {Izv. Math.},
 ISSN = {1064-5632},
 Volume = {60},
 Number = {3},
 Pages = {453--654},
 Year = {1996},
 Language = {English},
 DOI = {10.1070/IM1996v060n03ABEH000077},
 zbMATH = {1047586},
 Zbl = {0926.20023}
}

@article{mayr,
  title={Sharply 2-transitive groups with point stabilizer of exponent 3 or 6},
  author={Mayr, Peter},
  journal={Proceedings of the American Mathematical Society},
  volume={134},
  number={1},
  pages={9--13},
  year={2006}
}

@article{mazurov,
  title={2-Transitive permutation groups},
  author={Mazurov, VD},
  journal={Siberian Mathematical Journal},
  volume={31},
  number={4},
  pages={615--617},
  year={1990},
  publisher={Springer}
}

@article{neumann,
  title={On the commutativity of addition},
  author={Neumann, Bernhard Hermann},
  journal={Journal of the London Mathematical Society},
  volume={1},
  number={3},
  pages={203--208},
  year={1940},
  publisher={Oxford University Press}
}

@article{olshanskii,
  title={Geometriya opredelyayushchikh sootnosheniy v gruppakh [Geometry of defining relations in groups]},
  author={Olshansky, A Yu},
  journal={M.: Nauka},
  year={1989}
}

@article{osin,
  title={Acylindrically hyperbolic groups},
  author={Osin, Denis},
  journal={Transactions of the American Mathematical Society},
  volume={368},
  number={2},
  pages={851--888},
  year={2016}
}

@article{rips_segev_tent,
  title={A sharply 2-transitive group without a non-trivial abelian normal subgroup},
  author={Rips, Eliyahu and Segev, Yoav and Tent, Katrin},
  journal={J. Eur. Math. Soc.(JEMS)},
  volume={19},
  number={10},
  pages={2895--2910},
  year={2017}
}

@article{rips_tent,
  title={Sharply 2-transitive groups of characteristic 0},
  author={Rips, Eliyahu and Tent, Katrin},
  journal={Journal f{\"u}r die reine und angewandte Mathematik (Crelles Journal)},
  volume={2019},
  number={750},
  pages={227--238},
  year={2019},
  publisher={De Gruyter}
}

@article{suchkov,
  title={Finiteness of some sharply doubly transitive groups},
  author={Suchkov, Nikolai Mikhailovich},
  journal={Algebra and Logic},
  volume={40},
  number={3},
  pages={190--193},
  year={2001},
  publisher={Springer}
}

@article{tent,
  title={Infinite sharply multiply transitive groups},
  author={Tent, Katrin},
  journal={Jahresbericht der Deutschen Mathematiker-Vereinigung},
  volume={118},
  number={2},
  pages={75--85},
  year={2016},
  publisher={Springer}
}

@book{tent_ziegler,
  title={A course in model theory},
  author={Tent, Katrin and Ziegler, Martin},
  number={40},
  year={2012},
  publisher={Cambridge University Press}
}

@article{tits,
  title={G{\'e}n{\'e}ralisations des groupes projectifs bas{\'e}es sur leurs propri{\'e}t{\'e}s de transitivit{\'e}},
  author={Tits, Jacques},
  journal={M{\'e}moires de la Classe des Sciences},
  volume={27},
  year={1952},
  publisher={Palais des Acad{\'e}mies}
}

@inproceedings{zassenhaus1,
  title={Kennzeichnung endlicher linearer Gruppen als Permutationsgruppen},
  author={Zassenhaus, Hans},
  booktitle={Abhandlungen aus dem Mathematischen Seminar der Universit{\"a}t Hamburg},
  volume={11},
  pages={17--40},
  year={1935},
  organization={Springer}
}

@inproceedings{zassenhaus2,
  title={{\"U}ber endliche Fastk{\"o}rper},
  author={Zassenhaus, Hans},
  booktitle={Abhandlungen aus dem mathematischen Seminar der Universit{\"a}t Hamburg},
  volume={11},
  pages={187--220},
  year={1935},
  organization={Springer}
}

\vspace{1cm} 

\textbf{Marco Amelio}

Universität Münster

Einsteinstraße 62

48149 Münster, Germany.

E-mail address: \href{mailto:mamelio@uni-muenster.de}{mamelio@uni-muenster.de}

\end{document}